%% file: main.tex
\documentclass[11pt]{amsart}

\usepackage{amsmath,amssymb,amsthm} 
\usepackage{mathrsfs}               
\usepackage{a4wide}                 

\usepackage{xcolor}                 
\usepackage[colorlinks=true,linkcolor=blue,citecolor=blue,urlcolor=blue]{hyperref}   
\usepackage{tikz}
\usepackage{tikz-cd}
\usetikzlibrary{arrows.meta,positioning,calc,fit}

\usepackage{enumitem}               

\makeatletter
\newcommand{\nameditem}[2]{%
  \item[#1]%
  \refstepcounter{enumi}%
  \protected@edef\@currentlabel{#1}%
  \label{#2}%
}
\makeatother

\usepackage{longtable}
\usepackage{booktabs}
\usepackage{array}
\usepackage{ragged2e}

\newcolumntype{L}[1]{>{\RaggedRight\arraybackslash}p{#1}}

\newcommand{\runinhead}[1]{%
  \par\medskip\noindent\textbf{#1.}\enspace\ignorespaces
}

\theoremstyle{plain}
\newtheorem{theorem}{Theorem}[section]
\newtheorem{lemma}[theorem]{Lemma}
\newtheorem{proposition}[theorem]{Proposition}
\newtheorem{corollary}[theorem]{Corollary}

\theoremstyle{definition}
\newtheorem{definition}[theorem]{Definition}
\newtheorem{remark}[theorem]{Remark}
\newtheorem{question}[theorem]{Question}
\newtheorem{example}[theorem]{Example}

\numberwithin{equation}{section}

\newcommand{\norm}[1]{\lVert #1 \rVert}
\newcommand{\abs}[1]{\left|#1\right|}

\newcommand{\Bop}{\mathcal{B}}
\newcommand{\N}{\mathbb{N}}
\newcommand{\C}{\mathbb{C}}

\newcommand{\Yspace}{\mathscr Y}

\newcommand{\Id}{\mathrm{Id}}
\newcommand{\cf}{\mathrm{cf}}
\newcommand{\supp}{\operatorname{supp}}
\newcommand{\spn}{\operatorname{span}}

\newcommand{\CH}{\mathrm{CH}}

\renewcommand{\le}{\leqslant}
\renewcommand{\ge}{\geqslant}
\renewcommand{\leq}{\leqslant}
\renewcommand{\geq}{\geqslant}

\makeatletter
\def\@tocline#1#2#3#4#5#6#7{\relax
  \ifnum #1>\c@tocdepth
  \else
    \par \addpenalty\@secpenalty\addvspace{#2}%
    \begingroup
      \hyphenpenalty\@M
      \@ifempty{#4}{%
        \@tempdima\csname r@tocindent\number#1\endcsname\relax
      }{%
        \@tempdima#4\relax
      }%
      \parindent\z@
      \leftskip#3\relax
      \advance\leftskip\@tempdima\relax
      \rightskip\@pnumwidth plus4em
      \parfillskip-\@pnumwidth
      #5\leavevmode\hskip-\@tempdima
        \ifcase #1
        \or\or \hskip 1em \or \hskip 2em \else \hskip 3em \fi
        #6\nobreak\relax
      \dotfill\hbox to\@pnumwidth{\@tocpagenum{#7}}\par
      \nobreak
    \endgroup
  \fi
}
\makeatother



\begin{document}

\title{Primariness and the Primary Factorisation Property}

\author[A. Acuaviva]{Antonio Acuaviva}
\address[A. Acuaviva]{School of Mathematical Sciences,
Fylde College,
Lancaster University,
LA1 4YF,
United Kingdom}
\email{ahacua@gmail.com}

\author[T. Kania]{Tomasz Kania}
\address[T. Kania]{Mathematical Institute,
Czech Academy of Sciences,
\v{Z}itn\'a 25,
115 67 Praha 1,
Czech Republic
\and
Institute of Mathematics and Computer Science,
Jagiellonian University,
\L{}ojasiewicza 6,
30-348 Krak\'ow,
Poland}
\email{kania@math.cas.cz, tomasz.marcin.kania@gmail.com}
\thanks{IM CAS (RVO 67985840).}
\date{\today}

\begin{abstract}
    We study the relation between primariness of Banach spaces and the stronger operator-theoretic notions of the primary factorisation property (PFP) and the uniform primary factorisation property (UPFP). We revisit several classical primariness arguments and isolate the additional information needed to factor the identity through arbitrary operators. In the separable setting, this recovers quantitative factorisation versions of the Casazza--Kottman--Lin method for spaces with symmetric bases and treats the exceptional cases of $\ell_1$ and $\ell_\infty$. We then develop support-reduction and free-selection tools for uncountable direct sums, allowing one to transfer primariness and the PFP/UPFP from countable building blocks to non-separable $\ell_p$-, $c_0$- and more general symmetric sums. As applications, we obtain, among others, the primariness of $C[0,1]^*$ under the negation of the Continuum Hypothesis and UPFP results for uncountable sums of ordinal $C(\alpha)$-spaces. Finally, using the finite-block representation of $\mathcal B(\ell_p)$, we prove a uniform primary factorisation theorem for the Banach space $\mathcal B(\ell_p)$, $1<p<\infty$, and end with open problems concerning the gap between primariness and factorisation.
\end{abstract}

\maketitle

\tableofcontents

\input{0_Introduction}

\input{1_Preliminaries}
\input{2_Separable_symmetric_basis}

\input{3_Uncountable_direct}
\input{4_PFP_other_spaces}
\input{5_Primary_filters}
\input{w_Open_Questions}

\bigskip
\noindent\textbf{Acknowledgements.} This paper forms part of the first-named author's PhD research at Lancaster University, conducted under the supervision of Professor N. J. Laustsen. He acknowledges with thanks the funding from the EPSRC (grant number EP/W524438/1) that has supported his studies. The second-named author gratefully acknowledges support received from NCN Sonata-Bis~13 (2023/50/E/ST1/00067).\\

For the purpose of open access, the authors have applied a Creative Commons Attribution (CC BY) licence to any Author Accepted Manuscript version arising. \\

\noindent\textbf{Data availability.} No data was used for the research described in the article.

\newpage

\appendix
\input{Appendix_Quantitative_Pelczynski}
\input{Appendix_Jp_Dual_Primary}

\input{z_Bibliography}
\end{document}

%% file: 0_Introduction.tex
\section{Introduction}\label{sec: introduction}

A Banach space $E$ is \emph{primary} if for every bounded projection $P\colon E\to E$ one has
\[
PE\simeq E \qquad\text{or}\qquad (\Id_E-P)E\simeq E.
\]
Equivalently, whenever $E$ is written as a topological direct sum $E=Y\oplus Z$, one of the two summands must already be isomorphic to $E$.
In this sense, primary spaces are building blocks for Banach-space decompositions; the systematic study of primary spaces was to an extend triggered by Lindenstrauss' note \cite{Lindenstrauss1971},
They belong to the same circle of ideas as prime spaces, indecomposable spaces and spaces with few operators, and they form one of the central structural classes in Banach-space theory.

Classical proofs of primariness usually proceed in two steps.
First, one establishes a dichotomy for projections: given a projection $P$ on $E$, either $PE$ or $(\Id_E-P)E$
contains a complemented copy of $E$.
Second, one invokes Pe{\l}czy\'nski's decomposition method to upgrade this complemented-copy dichotomy to an isomorphism statement.
This strategy underlies many of the classical primariness results for sequence spaces, function spaces and iterated constructions.

In this paper we focus on a stronger operator-theoretic property.
We say that $E$ has the \emph{primary factorisation property} (PFP) if for every bounded operator
$T\colon E\to E$, the identity operator factors either through $T$ or through $\Id_E-T$.
If this can be done with a uniform bound on the factorisation constants, then $E$ has the
\emph{uniform primary factorisation property} (UPFP).
Thus the PFP replaces the projection dichotomy by a dichotomy for all bounded operators.
Whenever $E$ is stable under an $\ell_p$- or $c_0$-sum, the PFP implies primariness via Pe{\l}czy\'nski's decomposition method;
however, the factorisation viewpoint is strictly stronger and is often the more natural one for the operator algebra $\Bop(E)$.
In particular, it is closely related to the problem of identifying the unique maximal ideal of $\Bop(E)$.

Our first aim is to revisit a number of classical primariness arguments and extract from them the stronger conclusion of PFP or UPFP.
The guiding principle is that many classical proofs already contain an operator dichotomy strong enough to produce factorisations,
but this is often left implicit.
A second aim is to make this mechanism systematic in the non-separable setting, where uncountable symmetry and free-selection arguments allow one to sharpen the classical perturbative constructions.

The paper has three main themes.
In Section~\ref{sec: symmetric-like-basis} we analyse Banach spaces with symmetric bases and refine the classical arguments of Casazza, Kottman and Lin.
This yields quantitative factorisation results in the separable setting and, in particular, uniform versions of several classical primariness theorems.
We also discuss variants involving $\ell_\infty$ and iterated constructions.

In Section~\ref{sec: uncountable-direct-sums} we prove transfer principles for uncountable direct sums.
Roughly speaking, once the relevant countable sum of a ``small'' Banach space $X$ is known to be primary or to have the PFP/UPFP,
one can often pass to uncountable $\ell_p$-, $c_0$- or more general symmetric sums of copies of $X$.
The key input is a support-reduction mechanism showing that operators into such large sums are controlled on a countable set of coordinates,
combined with free-selection arguments that force exact vanishing of the off-diagonal interaction.
These results lead to new examples of primary spaces and spaces with the UPFP, including uncountable sums of ordinal $C(\alpha)$-spaces.

In the final substantive section, we discuss further classes of primary spaces and the extent to which their known proofs yield the PFP or UPFP.
Some cases fit naturally into the general framework developed earlier, while others require additional ideas.
We also include a catalogue of examples and a final section of open problems, with particular emphasis on the gap between primariness and the factorisation properties.

A recurring point of the paper is that primariness and factorisation should be carefully distinguished.
For instance, the classical theorem of Casazza--Kottman--Lin gives, in great generality, a range dichotomy for complemented copies,
but a range dichotomy does not by itself yield a factorisation of the identity.
One of the purposes of the present work is precisely to isolate those situations in which the stronger factorisation conclusion can in fact, be obtained.

%% file: 1_Preliminaries.tex
\section{Preliminaries}\label{sec:prelim}

\subsection{Basic notation and primariness}

Throughout the paper, all Banach spaces are over the scalar field $\mathbb K\in\{\mathbb R,\mathbb C\}$, and all operators are assumed to be bounded and linear. For Banach spaces $E$ and $F$ we write $\Bop(E,F)$ for the Banach space of all bounded operators from $E$ to $F$, and $\Bop(E)=\Bop(E,E)$. If $Y\subseteq E$ is a closed subspace, we write $Y\stackrel{c}{\hookrightarrow}E$ to indicate that $Y$ is complemented in $E$. We denote by $\operatorname{ran}T$ and $\ker T$ the range and kernel of an operator $T$.

\begin{definition}
Let $E,F$ be Banach spaces and $C\ge 1$. We write $E\simeq_C F$ if there exists an isomorphism
$T\colon E\to F$ with $\norm{T}\,\norm{T^{-1}}\le C$. We set
\[
d_{BM}(E,F)=\inf\{\norm{T}\,\| T^{-1} \|:\ T\colon E\to F\text{ is an isomorphism}\}.
\]
Thus $E\simeq_C F$ iff $d_{BM}(E,F)\le C$.
\end{definition}

\begin{definition}
Let $C \colon [1,\infty) \to [1,\infty)$ be an increasing function. A Banach space $E$ is \emph{$C$-primary} if for every bounded projection $P\colon E\to E$ one has
\[
PE\simeq_{C(\norm{P})} E
\qquad\text{or}\qquad
(\Id_E-P)E\simeq_{C(\norm{P})} E.
\]
We call $E$ \emph{uniformly primary} if it is $C$-primary for some such function $C$.
\end{definition}

To show that a Banach space $E$ is primary, it suffices to prove that for every projection $P \colon E \to E$, either $PE \simeq E$ or $(\Id_E - P)E \simeq E$. In practice, this is typically achieved in two steps. First, one proves a dichotomy for projections: for every projection $P \colon E \to E$, either $PE$ or $(\Id_E - P)E$ contains a complemented subspace isomorphic to $E$. Second, one applies Pe{\l}czy\'nski's decomposition method to deduce that $E$ is isomorphic to either $PE$ or $(\Id_E - P)E$; we recall this method now.

\begin{theorem}[Pe{\l}czy\'nski decomposition method \cite{Pelczynski1960}] \label{th: pelcdecmethod}
    Let $X$ and $Y$ be Banach spaces such that $X$ is isomorphic to a complemented subspace of $Y$ and $Y$ is isomorphic to a complemented subspace of $X$. Suppose further that
    \begin{enumerate}[label = (\alph*), ref = (\alph*)]
        \item \label{dec:methoda} $X \simeq X \oplus X$ and $Y \simeq Y \oplus Y$, or
        \item \label{dec:methodb} $X \simeq c_0(X)$ or $X \simeq \ell_p(X)$ for some $1 \leq p \leq \infty$.
    \end{enumerate}
    Then $X \simeq Y$.
\end{theorem}

The argument in Theorem~\ref{th: pelcdecmethod} is qualitative. It is natural to ask whether one can obtain a quantitative version, controlling the Banach--Mazur distance between $X$ and $Y$ in terms of the corresponding complementability constants. We return to this question in Section~\ref{sec:open-problems}.

In general, it is difficult to control the structure of $PE$ (and that of $(\Id_E - P)E$) and hence to guarantee condition \ref{dec:methoda}. However, in many cases, condition \ref{dec:methodb} holds automatically for the space $E$, and thus primariness reduces to establishing the dichotomy for projections. In passing, we observe that by \cite[Proposition 1]{CKL1977}, one may replace the $\ell_p$- or $c_0$-sum in \ref{dec:methodb} by a direct sum with respect to any Banach space with a symmetric basis.

Finally, we note that the two steps above are, in general, independent: one may establish a dichotomy for projections while still being unable to deduce primariness, as is the case for certain Banach spaces with a symmetric basis; see \cite{CasazzaLin1974}.

\subsection{Primary factorisation property}

As discussed above, proofs of primariness typically proceed by establishing a dichotomy for projections. In many cases, one can obtain a stronger property, namely a dichotomy that holds for \emph{all bounded operators} on the space.

Recall that if $S, T \colon E \to E$ are bounded operators, we say that $S$ \emph{factors through} $T$ if there exist bounded operators $U, V \colon E \to E$ such that $S = U T V$.

\begin{definition}
Let $E$ be a Banach space.

\begin{enumerate}[label = (\alph*), ref = (\alph*)]
    \item \label{def: PFP} We say that $E$ has the \emph{primary factorisation property} (PFP) if for every bounded operator $T \colon E \to E$, the identity $\Id_E$ factors either through $T$ or through $\Id_E - T$.
    \item Let $C \geq 1$. We say that $E$ has the \emph{$C$-primary factorisation property} ($C$-PFP) if for every bounded operator $T \colon E \to E$ there exist bounded operators $U, V \colon E \to E$ such that either
    \begin{equation*}
        UTV = \Id_E \quad \text{ or } \quad U(\Id_E - T)V = \Id_E
    \end{equation*}
    and $\norm{U}\,\norm{V} \le C.$
    \item Let $C \geq 1$. We say that $E$ has the \emph{$C^{+}$-primary factorisation property} ($C^{+}$-PFP) if it has the $(C+\varepsilon)$-PFP for every $\varepsilon > 0$.
    \item We say that $E$ has the \emph{uniform primary factorisation property} (UPFP) if there exists $C \geq 1$ such that $E$ has the $C$-PFP.
\end{enumerate}
\end{definition}

Naturally, although the constant may change, the UPFP is invariant under isomorphisms, as the following proposition shows.

\begin{proposition}\label{prop:iso-invariance}
    Let $E, F$ be Banach spaces and let $S \colon E \to F$ be an isomorphism.
    If $F$ has the $C$-PFP, then $E$ has the $C\,\norm{S}^2 \norm{S^{-1}}^2$-PFP.
\end{proposition}

\begin{proof}
    Let $T \in \Bop(E)$ and set $\widetilde T = S T S^{-1} \in \Bop(F)$.
    By the $C$-PFP in $F$, there exist operators $\widetilde U, \widetilde V \in \Bop(F)$ such that either
    \begin{equation*}
    \widetilde U\,\widetilde T\,\widetilde V = \Id_F
    \quad \text{or} \quad
    \widetilde U(\Id_F - \widetilde T)\widetilde V = \Id_F,
    \end{equation*}
    and $\norm{\widetilde U}\norm{\widetilde V} \le C$.
    
    Set $U = S^{-1}\widetilde U S$ and $V = S^{-1}\widetilde V S$.
    Then $UTV = \Id_E$ (or $U(\Id_E - T)V = \Id_E$), and
    \begin{equation*}
    \norm{U}\norm{V}
    \le (\norm{S}\norm{S^{-1}})^2\,\norm{\widetilde U}\norm{\widetilde V}
    \le C\,\norm{S}^2 \norm{S^{-1}}^2.\qedhere
    \end{equation*}
\end{proof}

Note that the PFP is strictly stronger than the requirement that projection ranges contain complemented copies of the space. Nevertheless, it is often the more natural hypothesis in applications. Closely related factorisation dichotomies have appeared implicitly in the literature for some time. For example, in \cite{Drewnowski1989}, spaces with PFP are treated under the name of `nice' spaces for the purposes of that paper. The terminology of the primary factorisation property is now used explicitly in recent work on factorisation and maximal ideals; see, for instance, \cite{KaniaLechner2022, LechnerSpeckhofer2025}.

A closely related predecessor is due to Casazza--Kottman--Lin \cite{CKL1976}: if $X$ has a symmetric basis and $E$ contains a complemented copy of $X$, then for every bounded operator $T\colon E\to E$, either $T(E)$ or $(\Id_E-T)(E)$ contains a complemented subspace isomorphic to $X$.  This is most naturally stated as a range dichotomy, not as a factorisation theorem: the presence of a complemented copy of $X$ inside $\operatorname{ran}T$ does not, in general, produce a bounded right inverse for $T$.  Nevertheless, the diagonal part of their argument contains the additional information needed for factorisation.

Although the PFP is, a priori, stronger than primariness, it appears to be at least as prevalent. The same is true for the UPFP, which, through a careful analysis of classical proofs of primariness, can often be obtained. Indeed, while primariness alone does not impose any uniform operator-theoretic dichotomy, many such proofs implicitly establish the PFP, and in several cases even its uniform version.

As observed in the factorisation literature, Pe{\l}czy\'nski's decomposition method (Theorem~\ref{th: pelcdecmethod}) shows that the PFP implies primariness whenever the space is $\ell_p$-stable. However, verifying this structural assumption is not always straightforward, and a space can have the PFP without primariness being known. For instance, while $\ell_\infty/c_0$ is known to have the UPFP (see Proposition~\ref{prop: l_infty-quot-c0}), its primariness is only known under the additional assumption of $\CH$.

Finally, the PFP is closely related to the problem of the existence of a unique maximal ideal in $\Bop(E)$. Indeed, Dosev and Johnson \cite{DosevJohnson2010} introduced the set
\begin{equation*}
    \mathcal{M}_E = \{T \in \Bop(E) : \Id_E \text{ does not factor through } T \} \subseteq \Bop(E),
\end{equation*}
as a natural candidate for such an ideal. For a deeper discussion of this problem, and in particular for the treatment of the cases $L_p$, $1 \leq p < \infty$, we refer the reader to \cite{DosevJohnsonSchechtman2013}.

We have the following observation.

\begin{proposition}\label{prop:maximal-ideal-pfp}
Let $E$ be a Banach space and set
\[
    \mathcal M_E=\{T\in\Bop(E):\ \Id_E\text{ does not factor through }T\}.
\]
Then $E$ has the PFP if and only if $\mathcal M_E$ is a maximal ideal of $\Bop(E)$. In this case $\mathcal M_E$ is the unique maximal ideal of $\Bop(E)$.
\end{proposition}

\begin{proof}
    Assume first that $E$ has the PFP. Plainly $\Id_E\notin\mathcal M_E$, and $\mathcal M_E$ is stable under left and right multiplication by arbitrary operators: if $\Id_E$ factors through $ATB$, then it factors through $T$.
    
    We next check additivity. Suppose that $S,T\in\mathcal M_E$ and that $S+T\notin\mathcal M_E$. Then there are $A,B\in\Bop(E)$ such that
    \[
        A(S+T)B=\Id_E.
    \]
    Apply the PFP to the operator $ASB$. Since
    \[
        \Id_E-ASB=ATB,
    \]
    the identity factors either through $ASB$ or through $ATB$, and therefore through $S$ or through $T$, a contradiction. Thus $\mathcal M_E$ is a proper two-sided ideal. It is closed: if $T_n\to T$ and $\Id_E=ATB$, then $AT_nB\to\Id_E$, so for all large $n$ the operator $AT_nB$ is invertible and $\Id_E$ factors through $T_n$.
    
    Finally, every proper ideal $\mathcal J$ of $\Bop(E)$ is contained in $\mathcal M_E$; indeed, if $T\in\mathcal J$ and $\Id_E$ factors through $T$, then $\Id_E\in\mathcal J$. Hence $\mathcal M_E$ is the largest proper ideal, and therefore the unique maximal ideal.
    
    Conversely, suppose that $\mathcal M_E$ is a maximal ideal. If the PFP failed, then there would be $T\in\Bop(E)$ such that both $T$ and $\Id_E-T$ belong to $\mathcal M_E$. Since $\mathcal M_E$ is an ideal, it is a linear subspace, and hence
    \[
        \Id_E=T+(\Id_E-T)\in\mathcal M_E,
    \]
    contradicting properness. Thus $E$ has the PFP.
\end{proof}

\begin{remark}
    The preceding proposition identifies the PFP with the assertion that the specific set $\mathcal M_E$ is a maximal ideal. It should not be read as saying that the mere existence of a unique maximal ideal in $\Bop(E)$ automatically implies the PFP: in examples where a unique maximal ideal is known by other means, one must still verify that it coincides with $\mathcal M_E$.
\end{remark}

\subsection{A catalogue of primary spaces}

The following table is intended as a guide to representative primary Banach spaces and classes rather than an exhaustive historical catalogue. We have not attempted full completeness, and some examples, refinements, or later quantitative improvements may be missing. In the entries below, $X$ denotes a Banach space with a symmetric basis, while $X_n$ denotes the relevant finite-dimensional building blocks in a direct-sum construction.

For readability, we group together closely related classes of examples, even when the original results were obtained at different times. We have also included a few entries that are especially relevant to the present paper, namely those concerning mixed sums, operator spaces and factorisation-based methods.

Recent work on strategically reproducible bases, stopping-time spaces and Haar-system Hardy spaces shows that the factorisation viewpoint extends well beyond the classical sym\-met\-ric-basis setting; see, for example, \cite{KaniaLechner2022,LechnerSpeckhofer2025}.

\begin{longtable}{@{}L{0.30\textwidth}L{0.08\textwidth}L{0.29\textwidth}L{0.23\textwidth}@{}}
\toprule
\textbf{Space / class} & \textbf{Year} & \textbf{Authors} & \textbf{PFP/UPFP} \\
\midrule
\endfirsthead

\toprule
\textbf{Space / class} & \textbf{Year} & \textbf{Authors} & \textbf{PFP/UPFP} \\
\midrule
\endhead

$c_0$ and $\ell_p$, $1 \leq p < \infty$ & 1960 &
A. Pe\l czy\'nski \cite{Pelczynski1960} & UPFP \\

$\ell_\infty$ & 1967 &
J. Lindenstrauss \cite{Lin1967} & UPFP \\

$C[0,1]$ & 1971 &
J. Lindenstrauss and A.Pe\l czy\'nski \cite{LindenstraussPelczynski1971} & PFP \\

Some spaces with symmetric basis, including Pe{\l}czy\'nski’s universal space & 1974 &
P.~Casazza and B.-L.~Lin \cite{CasazzaLin1974} & UPFP \\

$L_p[0,1]$ for $1<p<\infty$ & 1975 \& 1977 &
P. Enflo via B. Maurey \cite{Maurey1975} \& D. E. Alspach, P. Enflo, E. Odell \cite{AEO1977} & PFP \\

$C(\alpha)$, for most ordinals $\alpha$ & 1977 &
D. E. Alspach, Y. Benyamini \cite{AlspachBenyamini1977} & UPFP \\

James quasi-reflexive space $J$ & 1977 &
P. G. Casazza \cite{Casazza1977} & ?/? \\

$\ell_p(X)$, $c_0(X)$, $\left(\bigoplus_{n \in \N} X_n \right)_p$,
$\left(\bigoplus_{n \in \N} X_n \right)_{c_0}$ and
$\ell_r(\ell_\infty)$, \mbox{\ensuremath{1 < p < \infty}}, $1 \le r < \infty$ & 1976 &
P. G. Casazza, C. A. Kottman, B.-L. Lin \cite{CKL1976, CKL1977} & UPFP for some \\

$C(\alpha)$, $\alpha$ countable ordinal & 1978 \& 1982 &
P. Billard \cite{Billard1978} \& J. Wolfe \cite{Wolfe1982} & UPFP \\

$L_1[0,1]$ & 1975 \& 1979 &
P. Enflo via B. Maurey \cite{Maurey1975} \& P. Enflo, T. W. Starbird \cite{ES1979} & PFP \\

$\ell_p(\ell_q)$, $c_0(\ell_p)$ \& $\ell_p(c_0)$ for $1 \leq p, q <\infty$  & 1978 &
C. Samuel \cite{Samuel1978} & UPFP \\

Poulsen simplex (equivalently, the Gurari\u{\i} space) & 1980 &
W. Lusky \cite{Lusky1980} & PFP \\

Schatten classes $\mathcal{C}_p$, $1 < p <\infty$ & 1980 &
J. Arazy \cite{Arazy1980} & ?/? \\

$E$-valued Schatten classes $\mathcal{C}_E$ for spaces $E$ with a symmetric basis & 1980 &
J. Arazy \cite{Arazy1980Remark} & ?/? \\

$\ell_p(L_1)$, $1 \leq p < \infty$ & 1980 &
M. Capon \cite{Capon1980} & ?/? \\

$\ell_p(L_r)$ \& $c_0(L_r)$; $1 \leq p < \infty$, $1 < r < \infty$ & 1980 &
M. Capon \cite{Capon1980b} & ?/? \\

$L_p(\ell_r)$, $1 < p,r < \infty$ & 1980 &
M. Capon \cite{Capon1980b} & ?/? \\

$L_p(L_r)$, $1 < p,r < \infty$ & 1982 &
M. Capon \cite{Capon1982} & ?/? \\

$\ell_r(X)$ and $\left(\bigoplus_{n \in \N} X_n \right)_r$, $r = 1$ or $r = \infty$ & 1982 &
M. Capon \cite{Capon1982PLMS} & UPFP for some \\

Exotic-simplex space & 1982 &
W. Lusky \cite{Lusky1982} & ?/? \\

$H^\infty$ & 1983 &
J. Bourgain \cite{Bourgain1983} & ?/? \\

$L_p(X)$, $1 \leq p < \infty$ & 1983 &
M. Capon \cite{Capon1983} & ?/? \\

James tree space $JT$ & 1983 &
A. D. Andrew \cite{Andrew1983} & ?/? \\

Vector-valued function spaces & 1987 &
N. Popa \cite{Popa1987} & ?/? \\

$H^1$ and $\mathrm{BMO}$ & 1988 &
P. F. X. M\"uller \cite{Muller1988} & ?/? \\

$C((\beta \N) \setminus \N, X)$ for some $X$; in particular $C((\beta \N) \setminus \N \times K)$ for $K$ compact metric space & 1989 &
L. Drewnowski \cite{Drewnowski1989} & UPFP \\

$\mathcal M_2=(\bigoplus_n\Bop(\ell_2^n))_{\ell_\infty}$ and $\Bop(\ell_2)$ & 1990 &
G. Blower \cite{Blower1990} & UPFP  \\

$\ell_\infty/c_0$ (assuming $\CH$) & 1991 &
L. Drewnowski, J. W. Roberts \cite{DrewnowskiRoberts1991} & UPFP \\

Reflexive spaces with symmetric basis & 1992 &
Y.-H.~Lin \cite{Lin1992} & UPFP \\
Schatten classes $\mathcal{C}_p$ for $1\leqslant p < \infty$, $\mathcal{K}(\ell_2)$, nest algebras in $c_1$ & 1993 &
A.~Arias \cite{Arias1993} & ?/? \\
General tensor products of $\ell_p$-spaces & 1996 &
A. Arias and J. D. Farmer \cite{AriasFarmer} & ?/? \\

$\Bop(\ell_p,\ell_r)$, $1 < p < r < \infty$ & 1997 &
H. M. Wark \cite{Wark1997} & ?/? \\

$D(0,1)$ & 2005 &
A. Michalak \cite{Michalak2005} & ?/? \\

$\ell_\infty(L_p)$, $1 < p < \infty$ & 2007 &
H. M. Wark \cite{Wark2007} & ?/? \\

A class of primary Banach spaces & 2007 &
H. M. Wark \cite{Wark2007b} & ?/? \\

Strictly quasi-prime Banach spaces $X_p$, $1<p<\infty$, and analogues
associated with $\ell_1$ and $c_0$ & 2008 &
S.~A.~Argyros and T.~Raikoftsalis \cite{ArgyrosRaikoftsalis2008} & ?/? \\

$\ell_\infty(H^1(\mathbb{T}))$ and $\ell_\infty(L_1)$ & 2012 &
P.~F.~X.~M\"uller \cite{Muller2012} & PFP \\

$C_0(L_0)$ & 2015 &
T. Kania, N. J. Laustsen \cite{KaniaLaustsen2015} & ?/? \\

Stopping time space $S^1$ & 2015 &
D. Apatsidis \cite{Apatsidis2015} & ?/? \\

$\mathcal{F}(\mathbb{R}^d)$ & 2016 &
M. C\'uth, M. Doucha, P. Wojtaszczyk \cite{CuthDouchaWojtaszczyk2016} & ?/? \\
$\ell_\infty^c(\kappa)$ & 2016 & W.~B.~Johnson, T.~Kania, G.~Schechtman \cite{JKS2016} & UPFP \\
$SL^\infty$ & 2018 &
R. Lechner \cite{Lechner2018} & ?/? \\

Mixed-norm Hardy spaces $H^p(H^q)$, $1\le p,q<\infty$, with the
bi-parameter Haar system & 2018 &
N.~J.~Laustsen, R.~Lechner, and P.~F.~X.~M\"uller
\cite{LaustsenLechnerMuller2018} & large-diagonal FP \\

Spaces with strategically reproducible bases; in particular Haar systems in
$L^1$, multi-parameter Lebesgue spaces, mixed-norm Hardy spaces, and
unconditional sums & 2020 &
R.~Lechner, P.~Motakis, P.~F.~X.~M\"uller, and Th.~Schlumprecht
\cite{LMMSS2020} & large-diagonal FP \\

$\left(\sum_{n\in\N} H_n^p(H_n^q)\right)_r$ and
$\left(\sum_{n\in\N} H_n^p(H_n^q)^*\right)_r$,
$1\le p,q<\infty$, $1\le r\le\infty$ & 2018 &
R.~Lechner \cite{Lechner2018Mixed} & UPFP \\

$X^*$ and $\ell_p(X^*)$, $1\le p\le\infty$, for dual spaces with a
subsymmetric weak$^*$ Schauder basis satisfying Lechner's condition & 2019 &
R.~Lechner \cite{Lechner2019Subsymmetric} & PFP \\

Direct sums (Orlicz/Marcinkiewicz) & 2020 &
J. L. Ansorena \cite{Ansorena2020} & ?/? \\

$\ell^\infty(X_k)$-sums satisfying the hypotheses of \cite{LMMSS2021} & 2021 &
R.~Lechner, P.~Motakis, P.~F.~X.~M\"uller, and Th.~Schlumprecht
\cite{LMMSS2021} & large-diagonal FP \\

Stopping-time spaces (maximal ideal viewpoint) & 2022 &
T. Kania, R. Lechner \cite{KaniaLechner2022} & ?/? \\

$L_1(L_p)$, $1 \leq p < \infty$ & 2023 &
R.~Lechner, P.~Motakis, P.~F.~X.~M\"uller, and Th.~Schlumprecht\cite{LMMSS2022} & ?/? \\

Haar system spaces with unconditional Haar basis, including suitable
rearrangement-invariant spaces & 2023 &
Kh.~V.~Navoyan \cite{Navoyan2023} & large-diagonal FP \\

Bi-parameter Haar system Hardy spaces, restricted to Haar multipliers & 2024 &
R.~Lechner, P.~Motakis, P.~F.~X.~M\"uller, and Th.~Schlumprecht
\cite{LMMSS2024} & Haar-multiplier PFP \\

Rosenthal spaces $X_{p,w}$ and Bourgain--Rosenthal--Schechtman spaces
$R_\omega^p$, $1<p<\infty$ & 2025 &
K.~Konstantos and P.~Motakis \cite{KonstantosMotakis2025Orthogonal} & PFP \\

Independent sums $Y_\omega$ of Haar system Hardy spaces $Y$ & 2025 &
K.~Konstantos and T.~Speckhofer \cite{KonstantosSpeckhofer2025} & PFP \\

Haar Hardy spaces & 2025 &
R. Lechner, T. Speckhofer \cite{LechnerSpeckhofer2025} & ?/? \\

$c_0(\ell_\infty)$
& present paper
& present authors
& UPFP\\

$C[0,1]^*$ assuming $\mathfrak{c} > \aleph_1$ & present paper &
present authors & PFP \\

Uncountable sums $\ell_p(\Gamma,C(\alpha))$, $1\le p<\infty$ & present paper &
present authors & UPFP\\

$\mathcal M_p=(\bigoplus_n\Bop(\ell_p^n))_{\ell_\infty}$ and $\Bop(\ell_p)$, $1<p<\infty$ & present paper &
G.~Blower for $p=2$; present authors & UPFP \\
$c_{0, \mathcal{F}}$ for certain filters $\mathcal{F}$; this includes $\ell_\infty^c(\kappa)$ & present paper & present authors & UPFP \\
\bottomrule
\end{longtable}
The table is deliberately representative rather than exhaustive.  We have kept a broadly chronological order, while grouping together closely related classes when this improves readability.  The last entries illustrate the point of view of the present paper: several classical primariness arguments can be sharpened to factorisation dichotomies, while the uncountable transfer principles below produce new non-separable examples once the relevant support-reduction hypotheses are available.

The present paper adds two kinds of entries to the classical list. First, several classical primariness arguments are upgraded to PFP or UPFP statements. Second, the uncountable transfer results below produce non-separable analogues such as $\ell_p(\Gamma,C(\alpha))$ and more general $E(\Gamma,X)$-sums under support-reduction hypotheses.

%% file: 2_Separable_symmetric_basis.tex
\section{Symmetric-like bases and the UPFP}\label{sec: symmetric-like-basis}

As we mentioned, the standard methods for proving primariness of a space rely on two main steps: proving a factorisation of the identity for (projection) operators and then applying the Pe{\l}czy\'nski decomposition method. In this section, we will discuss the first problem in the particular case of Banach spaces with a (sub-)symmetric basis. 

We begin by fixing some notation. Let $X$ be a Banach space with a normalised basis $(e_n)_{n\in\N}$ and biorthogonal functionals $(e_n^*)_{n\in\N}$. We say that the basis is \emph{symmetric} if it is equivalent to each of its permutations and sign-changes; after replacing the norm by an equivalent one, we may and shall assume that the basis is equipped with a symmetric norm. In particular, the basis is $1$-unconditional.

The aim of this section is twofold. First, we revisit the classical separable arguments of Casazza, Kottman and Lin and recast them in a form that yields factorisations of the identity through arbitrary operators, rather than only complemented copies inside operator ranges. Second, we record a number of natural extensions of these ideas, including constructions involving $\ell_1$ and $\ell_\infty$.

As will become apparent later, the case when the basis $(e_n)_{n \in \mathbb{N}}$ is equivalent to the unit basis of $\ell_1$ behaves rather differently from all other cases. Therefore, from now on and unless explicitly stated otherwise, we assume that the basis $(e_n)_{n \in \mathbb{N}}$ is not equivalent to the unit vector basis of $\ell_1$. 

\subsection{Banach spaces \texorpdfstring{$X$}{X} with  (sub-)symmetric basis} \label{subsec: symmetric-basis-separable}

We begin with what is arguably the simpler class of spaces known to possess the PFP factorisation property: Banach spaces with a symmetric basis. The key ingredients for the proof go back to the work of Casazza and Lin \cite{CasazzaLin1974}. 

Roughly speaking, the argument shows that for a given operator $T \colon  X \to X$, either $T$ or $\Id_X-T$ acts (almost) as a diagonal operator on some subsequence $(e_{n_k})_{k \in \mathbb{N}}$ of the basis $(e_n)_{n \in \mathbb{N}}$, with the diagonal entries uniformly bounded away from zero. The proof relies on a gliding-hump argument together with two key observations: first, under the action of $T$ each basis element $e_n$ can, up to a small error, affect only finitely many other elements; and second, each $e_n$ can, again up to a small error, be affected by only finitely many others. We will formalise this shortly. Throughout this section, we fix a Banach space $X$ with a symmetric norm and symmetric basis $(e_n)_{n \in \N}$.

We will need the following result of Casazza, Kottman, and Lin \cite[Lemma 2]{CKL1976}. For completeness, we provide a proof.

\begin{lemma}\label{lmm: sym-basis-affected}
    Let $(e_n)_{n \in \N}$ be a normalised unconditional basis of a Banach space $X$. Then no subsequence of $(e_n)_{n \in \N}$ spans a subspace isomorphic to $\ell_1$ if and only if $\lim_{m \to \infty} e_n^* (Te_m) = 0$  for any operator $T \colon  X \to X$ and any $n \in \N$.
\end{lemma}
\begin{proof}
Let $K$ be the unconditional constant of $(e_n)_{n\in\N}$.
    \medskip
\noindent\emph{($\Leftarrow$)} Assume that
\[
    \lim_{m\to\infty} e_n^*(T e_m)=0
    \qquad\text{for every }T\in\mathcal B(X)\text{ and every }n\in\N.
\]
Fix $f\in X^*$ and $n_0\in\N$, and consider the rank--one operator
\[
    T_f \colon X\to X, \qquad T_f(x)= f(x)\,e_{n_0}.
\]
Then, for every $m\in\N$,
\[
    e_{n_0}^*(T_f e_m)= e_{n_0}^*\bigl(f(e_m)e_{n_0}\bigr)= f(e_m).
\]
Hence $f(e_m)\to 0$ for every $f\in X^*$, i.e.\ $(e_m)_{m \in \N}$ is weakly null. If a subsequence $(e_{m_k})_{k \in \N}$ spanned a subspace isomorphic to $\ell_1$, then the corresponding isomorphism $S\colon [e_{m_k}: k \in \N] \to \ell_1$ would send the weakly
null sequence $(e_{m_k})$ to a weakly null sequence in $\ell_1$, hence (by the
Schur property of $\ell_1$) to a norm--null sequence. Since $S$ is an
isomorphism and $(e_{m_k})_{k \in \N}$ is normalised, this is impossible. Therefore, no
subsequence of $(e_n)_{n \in \N}$ spans a copy of $\ell_1$.
\medskip
\noindent\emph{($\Rightarrow$)} Assume now that no subsequence of $(e_n)_{n \in \N}$ spans
a subspace isomorphic to $\ell_1$. Suppose, towards a contradiction, that there
exist $T\in\mathcal B(X)$ and $n_0\in\N$ such that
\[
    \limsup_{m\to\infty}\bigl|e_{n_0}^*(T e_m)\bigr|>0.
\]
Then there are $\varepsilon>0$ and a strictly increasing sequence $(m_k)_{k \in \N}$ such
that for all $k\in\N$ we have
\(
    \bigl|e_{n_0}^*(T e_{m_k})\bigr|\ge \varepsilon.
\)
Set $f=T^*e_{n_0}^*\in X^*$. Passing to a subsequence and, in the complex case, to a sector of aperture strictly smaller than $\pi/2$, we may find a unimodular scalar $\lambda$ such that
\[
    \operatorname{Re}\bigl(\lambda f(e_{m_k})\bigr)\ge \varepsilon/2
    \qquad(k\in\N).
\]
In the real case this is just the usual passage to a subsequence of constant sign. Replacing $f$ by $\lambda f$, we therefore have
\[
    \operatorname{Re} f(e_{m_k})\ge \varepsilon/2
    \qquad(k\in\N).
\]
Let $(a_k)_{k\in\N}$ be finitely supported. By unconditionality,
\[
    \Bigl\|\sum_{k\in\N} a_k e_{m_k}\Bigr\|
    \ge \frac1K \Bigl\|\sum_{k\in\N} |a_k|e_{m_k}\Bigr\|.
\]
Moreover,
\[
\begin{aligned}
    \Bigl\|\sum_{k\in\N} |a_k|e_{m_k}\Bigr\|
    &\ge \frac1{\|f\|}\,
       \operatorname{Re} f\Bigl(\sum_{k\in\N}|a_k|e_{m_k}\Bigr)  \\
    &= \frac1{\|f\|}\sum_{k\in\N}|a_k|\operatorname{Re} f(e_{m_k})
    \ge \frac{\varepsilon}{2\|f\|}\sum_{k\in\N}|a_k|.
\end{aligned}
\]
Combining these estimates gives
\[
    \Bigl\|\sum_{k\in\N} a_k e_{m_k}\Bigr\|
    \ge \frac{\varepsilon}{2K\|f\|}\sum_{k\in\N}|a_k|.
\]
On the other hand, since the basis is normalised,
\[
    \Bigl\|\sum_{k \in \N} a_k e_{m_k}\Bigr\|
    \le \sum_{k \in \N} |a_k|.
\]
Hence $(e_{m_k})_{k \in \N}$ is equivalent to the unit vector basis of $\ell_1$, so its
closed linear span is isomorphic to $\ell_1$, contradicting the hypothesis.
Therefore $\lim_{m\to\infty} e_n^*(T e_m)=0$ for all $T\in \Bop (X)$ and all
$n\in\N$.
\end{proof}

We also require the following elementary observation.

\begin{lemma}\label{lmm: small-norm}
Let $X$ be a Banach space with a normalised $1$-unconditional basis $(e_n)_{n\in\N}$,
and let $T\colon X\to X$ be bounded.
Suppose there exists $\varepsilon>0$ such that
\[
|e_m^*Te_n|<\frac{\varepsilon}{2^{n+m}}
\qquad(m,n\in\N).
\]
Then $\norm{T}\le \varepsilon$.
\end{lemma}

\begin{proof}
For each $n\in\N$,
\[
\norm{Te_n}
= \Bigl\|\sum_{m\in\N}(e_m^*Te_n)e_m\Bigr\|
\le \sum_{m\in\N}|e_m^*Te_n|
< \sum_{m\in\N}\frac{\varepsilon}{2^{n+m}}
=\frac{\varepsilon}{2^n}.
\]
Now let $x=\sum_{n\in\N} a_n e_n\in B_X$. Since the basis is normalised and $1$-unconditional,
we have $\norm{e_n^*}=1$ for every $n$, and therefore
$|a_n|=|e_n^*(x)|\le 1$.
Hence
\[
\norm{Tx}
=\Bigl\|\sum_{n\in\N} a_n Te_n\Bigr\|
\le \sum_{n\in\N}|a_n|\,\norm{Te_n}
< \sum_{n\in\N}\frac{\varepsilon}{2^n}
=\varepsilon.
\]
This proves that $\norm{T}\le \varepsilon$.
\end{proof}

\begin{theorem}\label{th: Cassaza-Lin}
    Let $X$ be a separable Banach space with a symmetric basis, not equivalent to the $\ell_1$-basis, and a symmetric norm. Then $X$ has the $2^+$-PFP.
\end{theorem}
\begin{proof}
    Let $T\colon X\to X$ and fix $\varepsilon>0$. Choose $\delta>0$ small enough so that
    \[
    \frac{2}{1-\delta}<2+\varepsilon.
    \]
    For each $n\in\N$ we have
    \[
    1=e_n^*(\Id_X-T)e_n+e_n^*Te_n.
    \]
    Hence, there exists an infinite set $M\subseteq\N$ such that either
    $|e_n^*Te_n|\ge \tfrac12$ for all $n\in M$, or
    $|e_n^*(\Id_X-T)e_n|\ge \tfrac12$ for all $n\in M$.
    Replacing $T$ by $\Id_X-T$ if necessary, we may assume that
    $|e_n^*Te_n|\ge \tfrac12$ for all $n\in M$.
    Enumerate $M$ increasingly as $(n_k)_{k\in\N}$.
    
    For each fixed $k$, since $(e_n)_{n \in \N}$ is a normalised basis,
    \[
    \lim_{j\to\infty} e_{n_j}^*Te_{n_k}=0,
    \]
    and by Lemma~\ref{lmm: sym-basis-affected}, together with the assumption that the basis is not equivalent to the unit basis of $\ell_1$,
    \[
    \lim_{j\to\infty} e_{n_k}^*Te_{n_j}=0.
    \]
    We now thin the sequence explicitly.  Suppose $n_1,\ldots,n_{k-1}$ have been chosen.  The two preceding limits allow us to choose $n_k\in M$ larger than all previously chosen indices so that, for every $1\leq j<k$,
    \[
        |e_{n_j}^*Te_{n_k}|<\frac{\delta}{2^{j+k+1}}
        \quad\text{and}\quad
        |e_{n_k}^*Te_{n_j}|<\frac{\delta}{2^{j+k+1}}.
    \]
    Since all diagonal entries on $M$ have modulus at least $1/2$, the resulting subsequence, still denoted $(n_k)_{k\in\N}$, satisfies for all $j\ne k$
    \[
        \left|\frac{e_{n_j}^*Te_{n_k}}{e_{n_j}^*Te_{n_j}}\right|
        < \frac{\delta}{2^{j+k}}.
    \]
    Let $Y= [e_{n_k}:k\in\N ]$. Since $X$ is equipped with a symmetric norm, $Y$ is isometrically isomorphic to $X$.
    Let $\pi\colon X\to Y$ be the coordinate projection and $\iota\colon Y\hookrightarrow X$ the canonical inclusion.
    Define the diagonal operator $D\colon Y\to Y$ by
    \[
    De_{n_k}=\frac{1}{e_{n_k}^*Te_{n_k}}e_{n_k}
    \qquad(k\in\N).
    \]
    Since $|e_{n_k}^*Te_{n_k}|\ge \tfrac12$, we have $\norm{D}\le 2$.
    
    Set
    \[
    A=D\pi T\iota
    \qquad\text{and}\qquad
    S=\Id_Y-A.
    \]
    Then $e_{n_k}^*Se_{n_k}=0$ for every $k$, while for $j\neq k$,
    \[
        |e_{n_j}^*Se_{n_k}|=
        \left|\frac{e_{n_j}^*Te_{n_k}}{e_{n_j}^*Te_{n_j}}\right|
        < \frac{\delta}{2^{j+k}}.
    \]
    By Lemma~\ref{lmm: small-norm}, $\norm{S}<\delta$.
    Therefore $A=\Id_Y-S$ is invertible, and if we put $R=A^{-1}$ then
    \[
    \norm{R}\le \frac{1}{1-\delta}.
    \]
    Hence
    \[
    \Id_Y=R D\pi T\iota.
    \]
    If $J\colon X\to Y$ denotes the canonical isometric isomorphism induced by the subsequence $(e_{n_k})_{k \in \N}$, then
    \[
    \Id_X = J^{-1}R D\pi\, T\, \iota J.
    \]
    Thus $\Id_X$ factors through either $T$ or $\Id_X-T$, and the factorisation constant is bounded by
    \[
    \norm{J^{-1}RD\pi}\,\norm{\iota J}
    \le \frac{2}{1-\delta}
    <2+\varepsilon.
    \]
    This proves that $X$ has the $2^+$-PFP.
\end{proof}

\begin{remark}
    The same argument applies to spaces with a subsymmetric basis after replacing the isometric subsequence identifications by the canonical isomorphisms supplied by subsymmetry and tracking the corresponding constants. The non-separable setting is treated separately in Section~\ref{sec: uncountable-direct-sums}.
\end{remark}

\subsubsection{Iterated constructions not involving $\ell_1$}\label{subsec: symmetric-basis-separable-iterated}

For a Banach space $X$ with an unconditional basis $\mathcal{E} = (e_n)_{n \in \N}$ and a Banach space $Y$, we recall that we can define the $(X, \mathcal{E})$-sum of $Y$, as the space
\begin{equation*}
    (X,\mathcal{E})(Y) = \{(y_n)_{n \in \N} \in \prod_{n \in \N} Y: \sum_{n \in \N} \norm{y_n} e_n \in X\},
\end{equation*}
equipped with norm
\begin{equation*}
    \norm{(y_n)_{n \in \N}} = \norm{\sum_{n \in \N} \norm{y_n} e_n}
\end{equation*}

Observe that the previous sum depends not only on $X$ but also on the choice of unconditional basis $\mathcal{E} = (e_n)_{n \in \N}$. In an abuse of notation, we will simply write $X(Y)$, where the choice of basis of $X$ will be implicitly assumed.

For each $k \in \N$, let $\N^k$ be the set of sequences $\{(i_1, \dots, i_k): i_1, i_2, \ldots, i_k \in \N\}$. Given two sequences $I = (i_1, \dots, i_k) \in \N^k$ and $J = (j_1, \dots, j_r) \in \N^r$ we define its concatenation $I \smallfrown J = (i_1, \dots, i_k, j_1, \dots, j_r) \in \N^{k+r}$. For a sequence $I = (i_1, \dots, i_k)$ we denote by $|I|$ the length of the sequence, that is, $|I| = k$.

Let $X_1, \dots, X_N$ be Banach spaces with symmetric basis $\mathcal{E}_j = (e^j_n)_{n \in \N}$, $j = 1, \dots, N$ and symmetric norms, and let $Y = X_1(X_{2}(\dots (X_{N-1}(X_N))))$.

For each ordered $I = (i_1, \dots, i_N) \in \N^N$ we can define $e_I = (0, \dots, e_{J}, 0, \dots, 0)$ where $J = (i_2, \dots, i_N) \in \N^{N-1}$, $e_J \in X_2(X_3(\dots(X_N)))$ is defined recursively in the natural way and it is placed in the $i_1$th-coordinate. Let $\Phi: \N \to \N^N$ be the Cantor ordering. It is clear that $\mathcal{E} = (e_{\Phi(n)})_{n \in \N}$ is a $1$-unconditional basis for $Y$.

We now show that a Banach space $Y$ constructed as above always has the $2^+$-PFP. The proof is an iterative version of the argument used in Theorem~\ref{th: Cassaza-Lin}; see also the ideas underlying \cite[Theorem~3]{CKL1976}. Before proceeding to the proof, some notation is in order.

\begin{definition}
    Let $Y =X_1(X_{2}(\dots (X_{N-1}(X_N))))$ and $T \colon Y \to Y$ be an operator. We define the property of a sequence $J \in \bigcup_{k=1}^N \mathbb{N}^k$ \emph{supporting} $T$ by backward induction on the length of $J$:
    \begin{enumerate}
        \item Base Case ($|J|=N$): A sequence $I \in \mathbb{N}^N$ supports $T$ if $|e_I^* T e_I| \geq 1/2$.
        \item Inductive Step ($|J|<N)$: A sequence $J \in \mathbb{N}^k$ (for $1 \le k < N$) supports $T$ if the set of extensions
        \begin{equation*}
            \{ j \in \mathbb{N} : J \smallfrown (j) \text{ supports } T \}
        \end{equation*}
        is infinite.
    \end{enumerate}
\end{definition}

The following is immediate by backward induction, exploiting the fact that for any operator $T \colon  Y \to Y$ and any $I \in \N^N$, then $I$ supports either $T$ or $\Id_Y-T$.

\begin{lemma}\label{lmm: infinitely-support}
    Let $Y=X_1(X_2(\cdots(X_{N-1}(X_N))))$ and let $T\in\Bop(Y)$. Then one of the two operators $T$ and $\Id_Y-T$ is supported by infinitely many first-level nodes $(i_1)$.
\end{lemma}

\begin{proof}
    We prove by backward induction that every node $J\in\bigcup_{k=1}^{N-1}\N^k$ supports at least one of $T$ and $\Id_Y-T$. At level $N$, every terminal node $I\in\N^N$ supports at least one of them, since
    \[
        1=e_I^*Te_I+e_I^*(\Id_Y-T)e_I
    \]
    and hence at least one of the two summands has modulus at least $1/2$.
    
    Assume the assertion has been proved at level $k+1$. If a node $J\in\N^k$ supported neither $T$ nor $\Id_Y-T$, then only finitely many immediate successors $J\smallfrown(j)$ would support $T$, and only finitely many would support $\Id_Y-T$. This contradicts the induction hypothesis applied to all immediate successors. Thus, every node at level $k$ supports at least one of the two operators. In particular, every first-level node supports at least one of the two alternatives. Since there are only two alternatives, one of them is supported by infinitely many first-level nodes.
\end{proof}
\begin{lemma}[CKL tree selection]\label{lem:CKL-tree-selection}
    Let $Y=X_1(X_2(\cdots(X_{N-1}(X_N))))$, where each $X_j$ has a symmetric basis not equivalent to the unit vector basis of $\ell_1$. Let $T\in\Bop(Y)$ and let $\eta>0$. Suppose that infinitely many first-level nodes support $T$. Then there is a subset $\mathcal A\subseteq\N^N$, order-isomorphic to $\N^N$ in the natural tree sense, such that
    \[
        Z=[e_I:I\in\mathcal A]
    \]
    is isometrically isomorphic to $Y$, every $I\in\mathcal A$ satisfies
    \[
        |e_I^*Te_I|\ge \tfrac12,
    \]
    and, after enumerating $\mathcal A$ as $(I_k)_{k\in\N}$,
    \[
        |e_{I_j}^*Te_{I_k}|<\frac{\eta}{2^{j+k+1}}
        \qquad(j\neq k).
    \]
\end{lemma}
\begin{proof}
    The supporting condition gives an infinitely splitting subtree of terminal nodes on which the diagonal estimate holds. The product basis $(e_I)_{I\in\N^N}$ has no subsequence equivalent to the unit vector basis of $\ell_1$: this follows by induction on $N$, since an infinite sequence of terminal nodes either has infinitely many distinct first coordinates, in which case it is equivalent to a subsequence of the basis of $X_1$, or else it has an infinite subsequence contained in a single first fibre.

    Lemma~\ref{lmm: sym-basis-affected} therefore gives
    \(
        e_I^*Te_J\longrightarrow 0
    \)
    whenever $I$ is fixed and $J$ tends to infinity inside any infinitely splitting branch-subtree. Since $(e_I)$ is a Schauder basis, we also have
    \(
        e_J^*Te_I\longrightarrow 0
    \)
    for fixed $I$ as $J$ tends to infinity. Now enumerate the terminal nodes of the desired subtree as they are chosen. At the $m$th step only finitely many previously selected terminal nodes have to be protected, and the two displayed limits allow us to discard finitely many successors from finitely many nodes while keeping every remaining node infinitely splitting. Choosing the next terminal node from the remaining tree and iterating this thinning gives an infinitely splitting subtree whose terminal set can be enumerated as $(I_k)_{k\in\N}$ and satisfies
    \[
        |e_{I_j}^*Te_{I_k}|<\frac{\eta}{2^{j+k+1}}
        \qquad(j\neq k).
    \]
    The resulting subtree is order-isomorphic to $\N^N$, so symmetry of the bases identifies its closed span isometrically with $Y$.
\end{proof}

\begin{theorem}\label{th: Cassaza-Kottman-Lin}
    Let $X_1, \dots, X_N$ be Banach spaces with symmetric norms and symmetric bases $\mathcal{E}_j$, none of which is equivalent to the unit vector basis of $\ell_1$. Then the Banach space $Y = X_1(X_{2}(\dots (X_{N-1}(X_N))))$ has the $2^+$-PFP.
\end{theorem}
\begin{proof}
    Let $T\in\Bop(Y)$ and fix $\varepsilon>0$. Choose $\delta>0$ such that
    \[
        \frac{2}{1-\delta}<2+\varepsilon.
    \]
    By Lemma~\ref{lmm: infinitely-support}, after replacing $T$ by $\Id_Y-T$ if necessary, infinitely many first-level nodes support $T$.
    
    Apply Lemma~\ref{lem:CKL-tree-selection} with $\eta=\delta$ and write the selected terminal nodes as $(I_k)_{k\in\N}$. Let
    \[
        Z=[e_{I_k}:k\in\N].
    \]
    Then $Z$ is isometrically isomorphic to $Y$. Let $\pi\colon Y\to Z$ be the coordinate projection and $\iota\colon Z\hookrightarrow Y$ the inclusion. Define $D\colon Z\to Z$ by
    \[
        De_{I_k}=\frac{1}{e_{I_k}^*Te_{I_k}}e_{I_k}
        \qquad(k\in\N).
    \]
    Then $\|D\|\le 2$. As in the proof of Theorem~\ref{th: Cassaza-Lin}, the operator
    \(
        S=\Id_Z-D\pi T\iota
    \)
    satisfies $\|S\|<\delta$. Hence $D\pi T\iota$ is invertible; writing $R=(D\pi T\iota)^{-1}$, we have
    \[
        \Id_Z=R D\pi T\iota
        \qquad\text{and}\qquad
        \|RD\|\le \frac{2}{1-\delta}<2+\varepsilon.
    \]
    If $J\colon Y\to Z$ is the canonical isometric isomorphism induced by the selected subtree, then
    \[
        \Id_Y=J^{-1}RD\pi\,T\,\iota J.
    \]
    Thus the identity on $Y$ factors through either $T$ or $\Id_Y-T$ with constant smaller than $2+\varepsilon$.
\end{proof}

\subsection{The \texorpdfstring{$\ell_1$}{l-1}-case} \label{subsec: l_1-case-separable}

\subsubsection{The space $\ell_1$. } The fact that $\ell_1$ is primary, and indeed prime, goes back to Pe{\l}czy\'nski \cite{Pelczynski1960}, whose argument exploits the complete homogeneity of the unit vector basis of $\ell_1$. To establish the UPFP for this space, we shall use a different method, which will also be useful for the iterated constructions.

Observe that the key difference with other Banach spaces with a symmetric basis is that Lemma \ref{lmm: sym-basis-affected} is no longer true. We see now how to solve this.

\begin{lemma}\label{lmm: ell1-future-interference}
    Let $T \colon \ell_1 \to \ell_1$ be an operator. Then for any $\varepsilon > 0$ there exists an infinite subset $M = \{m_k: k \in \N \} \subseteq \N$ such that for all $k < j$ we have
    \(
        |e_{m_k}^* T e_{m_j}| < \varepsilon.
    \)
\end{lemma}
\begin{proof}
    Let $\varepsilon > 0$ be fixed. We will construct inductively a strictly increasing sequence $(m_k)_{k \in \N}$ and a decreasing sequence of infinite subsets $(M_k)_{k \in \N}$ of $\N$ such that $m_k \in M_{k-1}$, $M_k \subseteq M_{k-1}$, $m_k < n$ for every $n \in M_k$, and
    \begin{equation*}
        |e_{m_k}^* T e_n| < \varepsilon \quad \text{for all } n \in M_k.
    \end{equation*}
    Once this is done, if $k < j$, then $m_j \in M_{j-1} \subseteq M_k$, and hence
    \(
        |e_{m_k}^* T e_{m_j}| < \varepsilon,
    \)
    as required.

    Set $M_0 = \N$. Suppose that $M_{k-1}$ has already been chosen and is infinite. Choose $N_k \in \N$ such that $N_k > m_{k-1}$ if $k > 1$, and such that
    \begin{equation*}
        |M_{k-1} \cap \{1, \dots, N_k\}| > \norm{T}/\varepsilon.
    \end{equation*}
    Set $F_k = M_{k-1} \cap \{1, \dots, N_k\}$, so that $F_k$ is finite and $|F_k| > \norm{T}/\varepsilon$. For each $n \in M_{k-1}$ with $n > N_k$, define
    \begin{equation*}
        A_k(n) = \{m \in F_k : |e_m^* T e_n| \geq \varepsilon\}.
    \end{equation*}
    Since
    \begin{equation*}
        \norm{T} \geq \norm{Te_n} \geq \sum_{m \in A_k(n)} |e_m^* T e_n| \geq |A_k(n)| \varepsilon,
    \end{equation*}
    we have $|A_k(n)| \leq \norm{T}/\varepsilon < |F_k|$. Hence, for each $n \in M_{k-1}$ with $n > N_k$, there exists $m \in F_k$ such that $m \notin A_k(n)$, that is, $|e_m^* T e_n| < \varepsilon$.

    Since $F_k$ is finite and $M_{k-1} \cap \{n \in \N : n > N_k\}$ is infinite, the pigeonhole principle yields some $m_k \in F_k$ such that the set
    \begin{equation*}
        M_k = \{n \in M_{k-1} : n > N_k \text{ and } |e_{m_k}^* T e_n| < \varepsilon\}
    \end{equation*}
    is infinite. This completes the induction.
\end{proof}

We also have the following analogue to Lemma \ref{lmm: small-norm}.

\begin{lemma}\label{lmm: small-norm-ell1}
    Let $T \colon \ell_1 \to \ell_1$ be an operator. Suppose that there exists $\varepsilon > 0$ such that for all $n \in \N$ we have
    \begin{equation*}
        \sum_{m \in \N} |e^*_m T e_n| < \varepsilon.
    \end{equation*}
    Then $\norm{T} \leq \varepsilon$.
\end{lemma}
\begin{proof}
    Let $x = \sum_{n \in \N} x_n e_n \in B_{\ell_1}$. Then
    \begin{equation*}
        \norm{Tx} = \norm{\sum_{n \in \N} x_n T e_n} \leq \sum_{n \in \N} |x_n| \norm{T e_n} \leq \sup_{n \in \N} \norm{T e_n} \sum_{n \in \N} |x_n| \leq \sup_{n \in \N} \norm{T e_n}.
    \end{equation*}
    On the other hand, for each $n \in \N$ we have
    \begin{equation*}
        \norm{T e_n} = \sum_{m \in \N} |e_m^* T e_n| < \varepsilon.
    \end{equation*}
    Hence $\norm{Tx} < \varepsilon$ for every $x \in B_{\ell_1}$, and therefore $\norm{T} \leq \varepsilon$.
\end{proof}

\begin{theorem}\label{th: ell1-case}
    The space $\ell_1$ has the $2^+$-PFP.
\end{theorem}
\begin{proof}
    Let $T\colon \ell_1 \to \ell_1$ and fix $\varepsilon>0$. Choose $\delta>0$ small enough so that
    \[
    \frac{2}{1-\delta}<2+\varepsilon.
    \]
    As in the proof of Theorem~\ref{th: Cassaza-Lin}, after replacing $T$ by $\Id_{\ell_1}-T$ if necessary, passing to an isometric copy and reindexing, we may assume that
    \[
    |e_n^*Te_n|\ge \tfrac12
    \qquad(n\in\N).
    \]
    We now recursively choose a strictly increasing sequence $(m_k)_{k\in\N}$ and decreasing infinite sets
    \[
    \N=M_0\supseteq M_1\supseteq M_2\supseteq\ldots
    \]
    so that for every $k\in\N$:
    \begin{enumerate}[label=(\roman*), ref=(\roman*)]
        \item $m_k\in M_{k-1}$;
        \item $|e_{m_k}^*Te_n|<\delta/2^{k+2}$ for all $n\in M_k$;
        \item $\sum_{n\in M_k}|e_n^*Te_{m_k}|<\delta/2^{k+2}$.
    \end{enumerate}
    Suppose that $M_{k-1}$ has been chosen. By the argument used in Lemma~\ref{lmm: ell1-future-interference}, applied inside the infinite set $M_{k-1}$ with threshold $\delta/2^{k+2}$, we may choose $m_k\in M_{k-1}$ and an infinite set $A_k\subseteq M_{k-1}$ with $m_k<n$ for all $n\in A_k$ such that
    \[
    |e_{m_k}^*Te_n|<\delta/2^{k+2}
    \qquad(n\in A_k).
    \]
    Since $Te_{m_k}\in\ell_1$, there exists an infinite subset $M_k\subseteq A_k$ such that
    \[
    \sum_{n\in M_k}|e_n^*Te_{m_k}|<\delta/2^{k+2}.
    \]
    This completes the construction.
    
    Set $Y=[e_{m_k}:k\in\N]$. Then $Y$ is isometric to $\ell_1$. Let $\pi\colon \ell_1\to Y$ be the coordinate projection and $\iota\colon Y\hookrightarrow \ell_1$ the inclusion. Define the diagonal operator $D\colon Y\to Y$ by
    \[
    De_{m_k}=\frac{1}{e_{m_k}^*Te_{m_k}}e_{m_k}
    \qquad(k\in\N).
    \]
    Since $|e_{m_k}^*Te_{m_k}|\ge \tfrac12$, we have $\norm{D}\le 2$.
    
    Set
    \[
    A=D\pi T\iota
    \qquad\text{and}\qquad
    S=\Id_Y-A.
    \]
    Then $e_{m_k}^*Se_{m_k}=0$ for every $k$. Moreover, for fixed $j$,
    \begin{align*}
    \sum_{k\neq j}|e_{m_k}^*Se_{m_j}|
    &\le 2\sum_{k\neq j}|e_{m_k}^*Te_{m_j}| \\
    &=2\sum_{k<j}|e_{m_k}^*Te_{m_j}|+2\sum_{k>j}|e_{m_k}^*Te_{m_j}| \\
    &< 2\sum_{k<j}\frac{\delta}{2^{k+2}} + 2\sum_{n\in M_j}|e_n^*Te_{m_j}| \\
    &< \frac\delta2 + \frac\delta{2^{j+1}} < \delta.
    \end{align*}
    Hence, by Lemma~\ref{lmm: small-norm-ell1}, $\norm{S}\le \delta$. Therefore $A=\Id_Y-S$ is invertible, and if we write $R=A^{-1}$ then
    \[
    \norm{R}\le \frac{1}{1-\delta}.
    \]
    Thus
    \[
    \Id_Y=RD\pi T\iota.
    \]
    If $J\colon \ell_1\to Y$ is the canonical isometric isomorphism induced by the subsequence $(e_{m_k})$, then
    \[
    \Id_{\ell_1}=J^{-1}RD\pi\,T\,\iota J.
    \]
    Therefore $\Id_{\ell_1}$ factors through either $T$ or $\Id_{\ell_1}-T$, and the factorisation constant is bounded by
    \[
    \norm{J^{-1}RD\pi}\,\norm{\iota J}
    \le \frac{2}{1-\delta}<2+\varepsilon.
    \]
    This proves that $\ell_1$ has the $2^+$-PFP.
\end{proof}

\subsubsection{Iterated constructions with $\ell_1$}\label{subsec:l_1-iterated}

We now consider iterated constructions involving $\ell_1$. The result proved below is for $X(\ell_1)$; this is the case needed in the sequel. Throughout the section, $X$ will denote a Banach space with a symmetric basis and endowed with a symmetric norm, and we assume that this basis is not equivalent to the unit vector basis of $\ell_1$.

Recall that $e_{(i,j)}$ denotes the vector $(0,\dots,0,e_j,0,\dots)\in X(\ell_1)$ with $e_j$ in the $i$th coordinate, and that $e_{(i,j)}^*$ denotes the corresponding coordinate functional. When $\mathbb N^2$ is endowed with the Cantor ordering, the family $(e_{(i,j)})_{(i,j)\in\mathbb N^2}$ forms a $1$-unconditional basis for $X(\ell_1)$.

We can now state the first step of the proof, which provides the technical ingredients needed to construct vectors spanning an isometric copy of $X(\ell_1)$ with only small interference between them.

\begin{lemma}\label{lmm: ell1-iterated-interference}
    Let $X$ be a Banach space with a symmetric basis and symmetric norm whose basis is not equivalent to the unit vector basis of $\ell_1$, and let $T \colon X(\ell_1) \to X(\ell_1)$ be an operator. Then the following hold:
\begin{enumerate}[label = (\alph*), ref = (\alph*)]
        \item \label{it: ell1-a} For every $(i,j) \in \N^2$ and every $\varepsilon > 0$, the set 
        \begin{equation*}
            \{r \in \N: \exists s \in \N \text{ such that } |e^*_{(i,j)} T e_{(r,s)}| \geq \varepsilon\}
        \end{equation*}
        is finite.
        \item \label{it: ell1-b} For every $(i,j) \in \N^2$ and every $\varepsilon >0$, the set
        \begin{equation*}
            \{(r,s) \in \N^2: |e^*_{(r,s)} T e_{(i,j)}| \geq \varepsilon\}
        \end{equation*}
        is finite.
        \item \label{it: ell1-c} For every $r \in \N$, every infinite set $M \subseteq \N$, every $\varepsilon > 0$, every finite set $\{i_1, \dots, i_K\}\subseteq \N$ and every choice of infinite sets $M_1, \dots, M_K \subseteq \N$, there exists $s \in M$ such that 
        \begin{equation*}
            \{m \in M_k: |e^*_{(r, s)} T e_{(i_k, m)}| \leq \varepsilon \}
        \end{equation*}
        is infinite for every $k = 1, \dots, K$.
    \end{enumerate}
\end{lemma}
\begin{proof}
    We first prove \ref{it: ell1-a}. Fix $(i,j) \in \N^2$ and $\varepsilon > 0$. Suppose, towards a contradiction, that there are strictly increasing rows $(r_k)_{k\in\N}$ and indices $(s_k)_{k\in\N}$ such that
    \begin{equation*}
        |e^*_{(i,j)}T e_{(r_k,s_k)}| \geq \varepsilon
        \qquad (k\in\N).
    \end{equation*}
    Passing to a subsequence and, in the complex case, to a sector, we may assume that there is a unimodular scalar $\lambda$ such that
    \begin{equation*}
        \operatorname{Re}\bigl(\lambda e^*_{(i,j)}T e_{(r_k,s_k)}\bigr)\geq \varepsilon/2
        \qquad (k\in\N).
    \end{equation*}
    Put $f=\lambda e^*_{(i,j)}T$. For finitely supported scalar sequences $(a_k)$, unconditionality gives
    \begin{align*}
        \Bigl\| \sum_k a_k e_{(r_k,s_k)}\Bigr\|_{X(\ell_1)}
        &\geq \Bigl\| \sum_k |a_k| e_{(r_k,s_k)}\Bigr\|_{X(\ell_1)} \\
        &\geq \frac{1}{\|f\|}
        \operatorname{Re} f\Bigl(\sum_k |a_k|e_{(r_k,s_k)}\Bigr)
        \geq \frac{\varepsilon}{2\|f\|}\sum_k |a_k|.
    \end{align*}
    The reverse estimate follows from the normalisation of the basis. Hence $(e_{(r_k,s_k)})_{k\in\N}$ is equivalent to the unit vector basis of $\ell_1$. Since the rows $r_k$ are distinct and the norm of $\sum_k a_k e_{(r_k,s_k)}$ is exactly $\|\sum_k |a_k|e_{r_k}\|_X$, this would make a subsequence of the basis of $X$ equivalent to the unit vector basis of $\ell_1$. By symmetry this is equivalent to saying that the whole basis of $X$ is equivalent to the unit vector basis of $\ell_1$, a contradiction.

    We next prove \ref{it: ell1-b}. Since $(e_{(r,s)})_{(r,s)\in\N^2}$, with the Cantor ordering, is a Schauder basis for $X(\ell_1)$, the scalar coordinates of the vector $T e_{(i,j)}$ tend to zero. Therefore, for each $\varepsilon>0$, only finitely many pairs $(r,s)$ can satisfy $|e^*_{(r,s)}T e_{(i,j)}|\geq \varepsilon$.

    For \ref{it: ell1-c}, we argue by induction on $K$. The case $K=1$ is the same pigeonhole argument used in Lemma~\ref{lmm: ell1-future-interference}. Choose $N \in \N$ with $N > \|T\|/\varepsilon$, and choose a finite subset $F \subseteq M$ with $|F|=N$. For each $m\in M_1$, set
    \begin{equation*}
        A(m)=\{s\in F: |e_{(r,s)}^*T e_{(i_1,m)}|>\varepsilon\}.
    \end{equation*}
    Since the projection onto the $r$th $\ell_1$-fibre has norm one,
    \begin{equation*}
        \|T\|
        \geq \|T e_{(i_1,m)}\|
        \geq \sum_{s\in A(m)} |e_{(r,s)}^*T e_{(i_1,m)}|
        > |A(m)|\varepsilon.
    \end{equation*}
    Thus $|A(m)|<|F|$. For every $m\in M_1$ there is therefore some $s\in F$ such that $|e_{(r,s)}^*T e_{(i_1,m)}|\leq\varepsilon$. Since $F$ is finite and $M_1$ is infinite, one such $s$ works for infinitely many $m\in M_1$.

    Suppose now that the assertion has been proved for $K$, and consider rows $i_1,\dots,i_K,i_{K+1}$ with infinite sets $M_1,\dots,M_K,M_{K+1}$. Choose $N\in\N$ with $N>\|T\|/\varepsilon$, and partition $M$ into pairwise disjoint infinite subsets $M^{(1)},\dots,M^{(N)}$. By the induction hypothesis, for each $n=1,\dots,N$ there is $s_n\in M^{(n)}$ such that
    \begin{equation*}
        \{m\in M_k: |e_{(r,s_n)}^*T e_{(i_k,m)}|\leq\varepsilon\}
    \end{equation*}
    is infinite for every $k=1,\dots,K$. Put $F=\{s_1,\dots,s_N\}$. For each $m\in M_{K+1}$, define
    \begin{equation*}
        A(m)=\{n\in\{1,\dots,N\}: |e_{(r,s_n)}^*T e_{(i_{K+1},m)}|>\varepsilon\}.
    \end{equation*}
    As above, $|A(m)|<N$ for every $m\in M_{K+1}$. Hence, by the pigeonhole principle, there is $n_0\in\{1,\dots,N\}$ such that
    \begin{equation*}
        \{m\in M_{K+1}: |e_{(r,s_{n_0})}^*T e_{(i_{K+1},m)}|\leq\varepsilon\}
    \end{equation*}
    is infinite. The choice of $s_{n_0}$ already gives the required infinite sets for $k=1,\dots,K$, so $s=s_{n_0}$ completes the induction.
\end{proof}

We now show that a suitable small-interference condition forces the resulting operator to have small norm.

\begin{lemma}\label{lmm: ell1-small-operator}
    Let $S \colon X(\ell_1) \to X(\ell_1)$ be an operator. Suppose that, for every $r,i \in \N$,
    \begin{equation*}
        \sup_{j \in \N} \sum_{s=1}^\infty |e^*_{(r,s)} S e_{(i,j)}|
        \leq
        \varepsilon 2^{-(r+i)}.
    \end{equation*}
    Then  $\norm{S} \leq \varepsilon$.
\end{lemma}

\begin{proof}
    For $r,i \in \N$, let
    \begin{equation*}
        S_{r,i}
        =
        Q_r S J_i \colon \ell_1 \to \ell_1,
    \end{equation*}
    where $J_i \colon \ell_1 \to X(\ell_1)$ denotes the canonical embedding of $\ell_1$ into the $i$th coordinate and $Q_r \colon X(\ell_1) \to \ell_1$ denotes the $r$th coordinate projection. Then
    \begin{equation*}
        \norm{S_{r,i}}
        =
        \sup_{j \in \N} \sum_{s=1}^\infty |e^*_{(r,s)} S e_{(i,j)}|
        \leq
        \varepsilon 2^{-(r+i)}.
    \end{equation*}

    Let $x = (x_i)_{i \in \N} \in B_{X(\ell_1)}$, so that $\norm{x_i} \leq 1$ for all $i \in \N$. For each $r \in \N$ we have
    \begin{align*}
        \norm{Q_rSx}_{\ell_1} = \norm{Q_r S \sum_{i \in \N} J_i x_i } \leq
        \sum_{i=1}^\infty \norm{S_{r,i}} \norm{x_i}_{\ell_1}
        &\leq
        \varepsilon \sum_{i=1}^\infty 2^{-(r+i)} = \varepsilon 2^{-r}.
    \end{align*}
    Hence
    \begin{equation*}
        \norm{Sx}_{X(\ell_1)} =
        \left\|
            \sum_{r=1}^\infty \norm{Q_r S x}_{\ell_1} e_r
        \right\|_X \leq
        \varepsilon \sum_{r \in \N} 2^{-r} = \varepsilon ,
    \end{equation*}
    so that $\norm{S} \leq \varepsilon$. This completes the proof.
\end{proof}

Lastly, we show that Lemma \ref{lmm: ell1-iterated-interference} allows us to pass to an isometric copy on which the interference is sufficiently small.

\begin{lemma}\label{lmm: ell1-iterated-true-diagonal}
    Let $X$ be a Banach space with a symmetric basis and symmetric norm whose basis is not equivalent to the unit vector basis of $\ell_1$, and let $T \colon X(\ell_1) \to X(\ell_1)$ be an operator. Then, for every $\varepsilon > 0$, there exist a strictly increasing sequence $(n_i)_{i \in \N}$ and infinite sets $M_i=\{m_i^j : j \in \N\}$, with $m_i^1<m_i^2<\ldots$ for every $i$, such that
    \begin{equation*}
        |e^*_{(n_r,m_r^s)} T e_{(n_i,m_i^j)}|
        \leq
        \varepsilon 2^{-(r+i+s)}
        \qquad ((r,s) \neq (i,j)).
    \end{equation*}
\end{lemma}

\begin{proof}
    Fix an enumeration $(\rho_t)_{t\in\N}$ of $\N^2$ such that $(k,1)$ appears before $(k,2)$, $(k,2)$ before $(k,3)$, and the rows are opened in the order $1,2,\dots$. Write $\rho_t=(k_t,\ell_t)$.

    We construct the coordinates recursively. After stage $n$ we have selected $m_i^j$ for all pairs $(i,j)=\rho_t$ with $t\leq n$, selected outer indices $n_i$ for the active rows, and chosen infinite reservoirs $R_i^{(n)}\subseteq\N$ for the active rows. Let $A_n$ be the set of active rows. We require:
    \begin{enumerate}[label=(\roman*), ref=(\roman*)]
        \item \label{it: ell1-true-invariant-reservoir} for every selected pair $\rho_t=(q,w)$, every active row $i\in A_n$, and every $m\in R_i^{(n)}$,
        \begin{equation*}
            |e^*_{(n_q,m_q^w)}T e_{(n_i,m)}|
            \leq \varepsilon 2^{-(q+i+w)};
        \end{equation*}
        \item \label{it: ell1-iterated-true-diagonal-ii} for every two selected pairs $\rho_s=(i,j)$ and $\rho_t=(q,w)$ with $(q,w)\neq(i,j)$,
        \begin{equation*}
            |e^*_{(n_q,m_q^w)}T e_{(n_i,m_i^j)}|
            \leq \varepsilon 2^{-(q+i+w)};
        \end{equation*}
        \item every element of $R_i^{(n)}$ is larger than all coordinates already selected in row $i$.
    \end{enumerate}
    At stage $0$ there is nothing to verify.

    Suppose the construction has been completed up to stage $n$, and let $\rho_{n+1}=(k,\ell)$. If $\ell=1$, row $k$ has not yet been opened. Choose $n_k$ as follows. If there are no previously selected pairs, take $n_k=1$. Otherwise, for each previously selected pair $(q,w)$ put
    \begin{equation*}
        B_{q,w}=\{r\in\N: \exists s\in\N \text{ such that } |e^*_{(n_q,m_q^w)}T e_{(r,s)}|>\varepsilon2^{-(q+k+w)}\}.
    \end{equation*}
    By Lemma~\ref{lmm: ell1-iterated-interference}\ref{it: ell1-a}, each $B_{q,w}$ is finite. Hence we may choose $n_k$ larger than all previously chosen outer indices and outside the finite union of the sets $B_{q,w}$. Then
    \begin{equation}\label{eq: ell1-new-row-old-targets}
        |e^*_{(n_q,m_q^w)}T e_{(n_k,m)}|
        \leq \varepsilon2^{-(q+k+w)}
    \end{equation}
    for every previously selected pair $(q,w)$ and every $m\in\N$. In this case set $\widetilde R=\N$.

    If $\ell>1$, row $k$ is already active. Set $\widetilde R=R_k^{(n)}$. Then \ref{it: ell1-true-invariant-reservoir} gives \eqref{eq: ell1-new-row-old-targets} for every previously selected pair $(q,w)$ and every $m\in\widetilde R$.

    We next make sure that the new target coordinate is not significantly affected by any previously selected source coordinate. For every previously selected pair $(q,w)$ define
    \begin{equation*}
        F_{q,w}=\{m\in\widetilde R: |e^*_{(n_k,m)}T e_{(n_q,m_q^w)}|>\varepsilon2^{-(k+q+\ell)}\}.
    \end{equation*}
    By Lemma~\ref{lmm: ell1-iterated-interference}\ref{it: ell1-b}, each $F_{q,w}$ is finite. Thus
    \begin{equation*}
        R=\widetilde R\setminus\bigcup_{\rho_t=(q,w),\,t\le n}F_{q,w}
    \end{equation*}
    is infinite, and every $m\in R$ satisfies
    \begin{equation}\label{eq: ell1-new-target-old-sources}
        |e^*_{(n_k,m)}T e_{(n_q,m_q^w)}|
        \leq \varepsilon2^{-(k+q+\ell)}
    \end{equation}
    for every previously selected pair $(q,w)$.

    Let $A_{n+1}=A_n\cup\{k\}$. For $i\in A_{n+1}$ define
    \begin{equation*}
        H_i=
        \begin{cases}
            R, & i=k,\\
            R_i^{(n)}, & i\in A_n\setminus\{k\}.
        \end{cases}
    \end{equation*}
    Put
    \begin{equation*}
        \eta=\min_{i\in A_{n+1}} \varepsilon2^{-(k+i+\ell)}.
    \end{equation*}
    Apply Lemma~\ref{lmm: ell1-iterated-interference}\ref{it: ell1-c} with source row $n_k$, source reservoir $R$, threshold $\eta$, target rows $\{n_i:i\in A_{n+1}\}$ and target reservoirs $(H_i)_{i\in A_{n+1}}$. We obtain $m_k^\ell\in R$ such that, for every $i\in A_{n+1}$,
    \begin{equation*}
        L_i=\{m\in H_i: |e^*_{(n_k,m_k^\ell)}T e_{(n_i,m)}|\leq\eta\}
    \end{equation*}
    is infinite. Since $\eta\leq \varepsilon2^{-(k+i+\ell)}$, these sets preserve the reservoir estimate for the newly selected target.

    For each $i\in A_{n+1}$, let $p_i$ be the largest coordinate already selected in row $i$ after adding $m_k^\ell$, and set
    \begin{equation*}
        R_i^{(n+1)}=\{m\in L_i:m>p_i\}.
    \end{equation*}
    These reservoirs are infinite. The estimates above verify \ref{it: ell1-true-invariant-reservoir} and \ref{it: ell1-iterated-true-diagonal-ii} at stage $n+1$: old estimates are inherited from the previous reservoirs; \eqref{eq: ell1-new-row-old-targets} handles old targets against the new source; \eqref{eq: ell1-new-target-old-sources} handles the new target against old sources; and the definition of the sets $L_i$ handles the new target against future sources.

    This completes the recursion. For each row $i$, put $M_i=\{m_i^j:j\in\N\}$. The construction makes each $M_i$ infinite and strictly increasing. Since every pair of distinct selected coordinates appears at some finite stage, \ref{it: ell1-iterated-true-diagonal-ii} gives
    \begin{equation*}
        |e^*_{(n_r,m_r^s)}T e_{(n_i,m_i^j)}|
        \leq \varepsilon2^{-(r+i+s)}
        \qquad ((r,s)\neq(i,j)),
    \end{equation*}
    as required.
\end{proof}

Finally, we are ready to prove the UPFP for $X(\ell_1)$.

\begin{theorem}\label{th: l_1-iterated}
    Let $X$ be a Banach space with a symmetric basis and symmetric norm whose basis is not equivalent to the unit vector basis of $\ell_1$. Then $X(\ell_1)$ has the $2^+$-PFP.
\end{theorem}

\begin{proof}
    Let $T \colon X(\ell_1) \to X(\ell_1)$ and fix $\varepsilon > 0$. Choose $\delta > 0$ small enough so that
    \begin{equation*}
        \frac{2}{1-\delta} < 2+\varepsilon.
    \end{equation*}

    For each row $i$, colour an index $j$ red if
    $|e^*_{(i,j)}Te_{(i,j)}|\geq\tfrac12$ and blue otherwise. In the blue case,
    $|e^*_{(i,j)}(\Id_{X(\ell_1)}-T)e_{(i,j)}|\geq\tfrac12$. In every row at least one colour occurs infinitely often, and therefore one colour occurs infinitely often in infinitely many rows. Passing to those rows and to the corresponding infinite subsets of their fibres gives a $1$-complemented isometric copy of $X(\ell_1)$. Replacing $T$ by $\Id_{X(\ell_1)}-T$ if necessary, and reindexing, we may assume that
    \begin{equation*}
        |e^*_{(i,j)} T e_{(i,j)}| \geq \tfrac12
        \qquad ((i,j)\in\N^2).
    \end{equation*}

    Apply Lemma \ref{lmm: ell1-iterated-true-diagonal} with parameter $\delta/2$, and pass to a further $1$-complemented isometric copy so that, after reindexing, we have
    \begin{equation*}
        |e^*_{(r,s)} T e_{(i,j)}|
        \leq
         (\delta/2) 2^{-(r+i+s)}
        \qquad ((r,s)\neq(i,j)).
    \end{equation*}

    Define the diagonal operator $D \colon X(\ell_1) \to X(\ell_1)$ by
    \begin{equation*}
        D(e_{(i,j)})
        =
        \frac{1}{e^*_{(i,j)} T e_{(i,j)}}\,e_{(i,j)}
        \qquad ((i,j)\in\N^2).
    \end{equation*}
    Since
    \begin{equation*}
        \left|\frac{1}{e^*_{(i,j)} T e_{(i,j)}}\right| \leq 2
        \qquad ((i,j)\in\N^2),
    \end{equation*}
    we have $\norm{D} \leq 2$.  We \emph{claim} that
    \begin{equation*}
        \norm{\Id_{X(\ell_1)} - DT} \leq \delta.
    \end{equation*}
    Indeed, let $r,i,j \in \N$. For each $s \in \N$, if $(r,s)\neq(i,j)$ then
    \begin{equation*}
        e^*_{(r,s)}(\Id_{X(\ell_1)} - DT)e_{(i,j)}
        =
        -\frac{e^*_{(r,s)} T e_{(i,j)}}{e^*_{(r,s)} T e_{(r,s)}}.
    \end{equation*}
    On the other hand,
    \begin{equation*}
        e^*_{(i,j)}(\Id_{X(\ell_1)} - DT)e_{(i,j)} = 0.
    \end{equation*}
    Therefore
    \begin{equation*}
        \sup_{j \in \N} \sum_{s=1}^\infty |e^*_{(r,s)}(\Id_{X(\ell_1)} - DT)e_{(i,j)}|
        \leq
        2 \sup_{j \in \N} \sum_{\substack{s \in \N\\(r,s)\neq(i,j)}} |e^*_{(r,s)} T e_{(i,j)}| \leq
        \delta 2^{-(r+i)}.
    \end{equation*}
    Lemma \ref{lmm: ell1-small-operator} applied to $\Id_{X(\ell_1)} - DT$ gives $\norm{\Id_{X(\ell_1)} - DT} \leq \delta$.

    Thus $DT$ is invertible and $\norm{(DT)^{-1}} \leq 1/(1 - \delta)$ and $(DT)^{-1} DT = \Id_{X(\ell_1)}$.
    Consequently, on the selected copy, $\Id_{X(\ell_1)}$ factors through the selected restriction of $T$ with factorisation constant bounded by
    \begin{equation*}
        \norm{(DT)^{-1}}\norm{D} \leq \frac{2}{1-\delta} <
        2+\varepsilon.
    \end{equation*}
    Composing with the canonical isometry onto the selected copy, its inclusion, and the norm-one coordinate projection gives the same bound for the original operator, or for $\Id_{X(\ell_1)}-T$ if that was the chosen colour. This proves that $X(\ell_1)$ has the $2^+$-PFP.
\end{proof}

As an immediate consequence, we obtain the following primarity result for $c_0(\ell_1)$ and $\ell_p(\ell_1)$, $1<p<\infty$, in the spirit of the theorem of Casazza, Kottman, and Lin \cite{CKL1976}. We note that the $\ell_p(\ell_1)$ cases also follow from an earlier paper of Samuel \cite{Samuel1978}, where the primarity of $\ell_r(\ell_s)$ is proved for $1\leq r,s\leq\infty$, with the non-standard convention that $\ell_\infty$ denotes $c_0$. The corresponding $c_0$-case is not explicitly carried out there, although Samuel indicates that it is analogous. Since these consequences do not seem to have been explicitly recorded in the later literature, we include the details here.

\begin{corollary}
    The space $c_0(\ell_1)$ and the spaces $\ell_p(\ell_1)$, $1 < p < \infty$, are primary.
\end{corollary}

\begin{proof}
    Apply Theorem~\ref{th: l_1-iterated} with $X=c_0$ and with $X=\ell_p$, $1 < p < \infty$. Since the canonical bases of $c_0$ and $\ell_p$ are symmetric, it follows that $c_0(\ell_1)$ and $\ell_p(\ell_1)$ have the $2^+$-PFP. Since they are $c_0$- and $\ell_p$-stable, respectively, Pe\l czy\'nski's decomposition method, Theorem~\ref{th: pelcdecmethod}, yields their primariness.
\end{proof}

\subsection{The \texorpdfstring{$\ell_\infty$}{l-infinity}-case}  \label{sec: l_infty-case-separable} 

We have seen that if a Banach space $X$ has a symmetric basis, then for every operator $T \colon X \to X$ one can pass to a subsequence of the original basis so that either $T$ or $\Id_X - T$ acts, up to an arbitrarily small perturbation, as a diagonal operator whose diagonal entries are uniformly bounded away from zero. This then allows us to factor the identity on $X$.

An inspection of the proof reveals that, for this construction, we require two properties. First, for each fixed $n \in \mathbb{N}$, we need
\begin{equation*}
\lim_{k \to \infty} e_k^* T e_n = 0,
\end{equation*}
meaning that each basis vector contributes, up to a small perturbation, to only finitely many coordinates. This followed easily from the fact that $(e_n)_{n \in \mathbb{N}}$ was a basis. Second, we need that
\begin{equation*}
    \lim_{k \to \infty} e_n^* T e_k = 0,
\end{equation*}
meaning that each coordinate is influenced, up to a small perturbation, by only finitely many basis vectors. This is essentially Lemma~\ref{lmm: sym-basis-affected}. Using these properties, a straightforward iterative argument yields the desired diagonal form. More concretely, one exploits the two conditions to construct a subsequence $(e_{n_k})_{k \in \N}$ so that, for every $k \in \N$, we have that $\sum_{r \not= k} |e^*_{n_k} T e_{n_r}|$ is arbitrarily small. In other words, the off-diagonal interactions between the images are negligible, which leads to the required diagonalisation. \\

In this section, we consider the case $X = \ell_\infty$. Although this space does not possess a Schauder basis in the usual sense, one may still view the canonical unit vectors $e_n = (0, \dots, 0, 1, 0, \dots)$, with $1$ in the $n$th coordinate, as playing the role of a basis. This can be formalised by identifying $\ell_\infty$ with the dual of $\ell_1$, in which case the canonical vectors $(e_n)_{n \in \mathbb{N}}$ form a weak$^*$-Schauder basis. For our purposes, we simply emphasise that the unit vectors $(e_n)_{n \in \N}$ in $\ell_\infty$ exhibit sufficient basis-like behaviour. Consequently, a suitable modification of the techniques discussed previously, combined with some ideas of Lindenstrauss \cite{Lin1967}, will enable us to establish the desired results.

\subsubsection{The space $\ell_\infty$} \label{subsec: l_infty-case-separable} Let $(e_n)_{n \in \N}$ denote the canonical unit vectors in $\ell_\infty$, and let $e_n^*$ be the coordinate functionals defined by $e_n^*(x) = x(n)$ for $x = (x(k))_{k \in \N} \in \ell_\infty$. As before, our goal is to argue that we can pass to a subsequence so that the cross-interaction between images of the unit vectors $(e_n)_{n \in \N}$ is small. The following results are streamlined version tailored to our purposes; the underlying ideas go back to the original paper of Lindenstrauss, see \cite[Proof of Lemma 5]{Lin1967}. For a vector $x\in\ell_\infty$ and $\varepsilon>0$, define
\[
    M(x,\varepsilon)=\{m\in\N: |e_m^*x|>\varepsilon\}.
\]

\begin{lemma}\label{lmm: Lind-preliminary-1}
    Let $(x_n)_{n \in \N}$ be a sequence in $\ell_\infty$ and suppose that there exists $K > 0$ such that, for every finite set $F\subseteq\N$ and every scalar family $(\lambda_n)_{n\in F}$,
    \begin{equation*}
        \norm{\sum_{n\in F} \lambda_n x_n} \leq K \max_{n\in F} |\lambda_n|.
    \end{equation*}
    Then the following hold:
    \begin{enumerate}[label = (\alph*), ref = (\alph*)]
        \item \label{it: l_infty-to-past} For every $m \in \N$, $\lim_{n \to \infty} e^*_m x_n = 0$.
        \item \label{it: l_infty-to-future} For every $\varepsilon > 0$, every infinite set $M \subseteq \N$ and every $N_0 \in \N$, there exists $n \in M$ with $n>N_0$ such that $M \setminus M(x_n, \varepsilon)$ is infinite.
    \end{enumerate}
\end{lemma}
\begin{proof}
    We prove \ref{it: l_infty-to-past} by contradiction. If it fails, then there are $m_0\in\N$, $\delta>0$ and distinct indices $n_1,n_2,\dots$ such that $|e^*_{m_0}x_{n_k}|>\delta$ for all $k$. Choose $L\in\N$ with $L\delta>K$ and choose unimodular scalars $\lambda_1,\dots,\lambda_L$ so that the numbers $\lambda_k e^*_{m_0}x_{n_k}$ have the same argument. Then
    \begin{equation*}
        K \geq \norm{\sum_{k=1}^L \lambda_k x_{n_k}}
        \geq \left|e^*_{m_0}\Bigl(\sum_{k=1}^L \lambda_k x_{n_k}\Bigr)\right|
        =\sum_{k=1}^L |e^*_{m_0}x_{n_k}|
        >L\delta,
    \end{equation*}
    a contradiction.

    We prove \ref{it: l_infty-to-future}. Assume it is false. Then there are $\varepsilon>0$, an infinite set $M\subseteq\N$ and $N_0\in\N$ such that $M\setminus M(x_n,\varepsilon)$ is finite for every $n\in M$ with $n>N_0$. Choose $L\in\N$ with $L\varepsilon>K$, and choose distinct $n_1,\dots,n_L\in M$ with $n_k>N_0$. Since each $M\setminus M(x_{n_k},\varepsilon)$ is finite, the intersection
    \begin{equation*}
        M\cap\bigcap_{k=1}^L M(x_{n_k},\varepsilon)
    \end{equation*}
    is non-empty; choose $m$ in it. Pick unimodular scalars $\lambda_1,\dots,\lambda_L$ aligning the numbers $e_m^*x_{n_k}$. Then
    \begin{equation*}
        K \geq \norm{\sum_{k=1}^L \lambda_k x_{n_k}}
        \geq \left|e_m^*\Bigl(\sum_{k=1}^L \lambda_k x_{n_k}\Bigr)\right|
        =\sum_{k=1}^L |e_m^*x_{n_k}|
        >L\varepsilon,
    \end{equation*}
    a contradiction.
\end{proof}

We will also need the following result.

\begin{lemma}\label{lmm: l_inf-separate-interference}
    Let $(x_n)_{n \in \N} \subseteq \ell_\infty$ be a sequence of vectors and suppose that there exists $K > 0$ such that, for every finite set $F\subseteq\N$ and every scalar family $(\lambda_n)_{n\in F}$,
    \begin{equation}\label{eq: l_infty_summable}
        \norm{\sum_{n\in F} \lambda_n x_n} \leq K \max_{n\in F} |\lambda_n|.
    \end{equation}
     Then for every $\varepsilon > 0$ there exists $(n_k)_{k \in \N}$ such that
     \begin{equation*}
         \sum_{j \not=k} |e_{n_k}^* x_{n_j}| < \varepsilon \quad \text{ for all } k \in \N.
     \end{equation*}
\end{lemma}
\begin{proof}
Fix $\varepsilon>0$. We recursively choose a decreasing sequence of infinite sets
\[
\N = A_0\supseteq A_1\supseteq A_2\supseteq \ldots,
\]
and integers $n_k\in A_{k-1}$ such that for every $k\in\N$,
\begin{enumerate}[label=(\roman*), ref =(\roman*)]
    \item \label{it: proof1i} $n_k\in A_{k-1}$;
    \item \label{it: proof1ii} $|e_m^*x_{n_k}|<\varepsilon/2^{k+1}$ for every $m\in A_k$;
    \item \label{it: proof1iii} $\sum_{m\in A_k}|e_{n_k}^*x_m|<\varepsilon/2^{k+1}$.
\end{enumerate}
We start with $A_0=\N$. Assume that $A_{k-1}$ has been chosen. Apply Lemma~\ref{lmm: Lind-preliminary-1}\,\ref{it: l_infty-to-future} with $M=A_{k-1}$, with $N_0=0$ if $k=1$ and $N_0=n_{k-1}$ otherwise, and with parameter $\varepsilon/2^{k+2}$. We obtain $n_k\in A_{k-1}$ such that the set
\[
B_k:=A_{k-1}\setminus M(x_{n_k},\varepsilon/2^{k+2})
\]
is infinite. Therefore
\[
|e_m^*x_{n_k}|<\frac{\varepsilon}{2^{k+1}}
\qquad(m\in B_k).
\]
By Lemma~\ref{lmm: Lind-preliminary-1}\,\ref{it: l_infty-to-past}, we have
$e_{n_k}^*x_m\to 0$ as $m\to\infty$ along $B_k$.
Hence we may choose an infinite subset
$A_k=\{m_r^{(k)}:r\in\N\}\subseteq B_k$ such that
\[
|e_{n_k}^*x_{m_r^{(k)}}|<\frac{\varepsilon}{2^{k+1+r}}
\qquad(r\in\N).
\]
Then \ref{it: proof1ii} and \ref{it: proof1iii} hold.

Now fix $k\in\N$.
If $j<k$, then $n_k\in A_j$, so by \ref{it: proof1ii} at stage $j$,
\[
|e_{n_k}^*x_{n_j}|<\frac{\varepsilon}{2^{j+1}}.
\]
On the other hand, since $n_j\in A_k$ for every $j>k$, property \ref{it: proof1iii} at stage $k$ gives
\[
\sum_{j>k}|e_{n_k}^*x_{n_j}|
\le \sum_{m\in A_k}|e_{n_k}^*x_m|
<\frac{\varepsilon}{2^{k+1}}.
\]
Therefore,
\[
\sum_{j\neq k}|e_{n_k}^*x_{n_j}|
= \sum_{j<k}|e_{n_k}^*x_{n_j}| + \sum_{j>k}|e_{n_k}^*x_{n_j}|
< \sum_{j<k}\frac{\varepsilon}{2^{j+1}} + \frac{\varepsilon}{2^{k+1}}
< \varepsilon.
\]
This proves the lemma.
\end{proof}

If $(e_n)_{n \in \N}$ were a basis, then for every $x =(x_n)_{n \in \N}$ we would have $Tx = \sum_{n \in \N} x_n\, T e_n$, so that the action of $T$ would be completely determined by its behaviour on the vectors $(e_n)_{n \in \N}$. However, this representation is no longer available in $\ell_\infty$, since $(e_n)_{n \in \N}$ does not form a Schauder basis of the space. We now see how to recover this property by exploiting a clever observation of Lindenstrauss, as seen in \cite[Proof of the main theorem]{Lin1967}. 

For an element $x \in \ell_\infty$, we define the \emph{support of} $x$ as $\supp(x) = \{n \in \N: e_n^*(x) \not = 0\}$. For an infinite subset $M \subseteq \N$, we write $\ell_\infty(M)$ for the $1$-complemented subspace of $\ell_\infty$ defined by $\ell_\infty(M) = \{ x \in \ell_\infty : \supp(x) \subseteq M \}$ and denote by $P_M: \ell_\infty \to \ell_\infty$ the canonical projection onto it.

\begin{proposition}\label{prop: l_inf-almost-basis-operators}
    Let $T\colon \ell_\infty \to \ell_\infty$ be an operator. Then there exists an infinite subset $M \subseteq \N$ such that for all $m \in \N$ and all $x = (x(n))_{n \in \N} \in \ell_\infty(M)$ we have
    \begin{equation*}
        e_m^* Tx = \sum_{n \in \N} x(n) (e_m^* T e_n).
    \end{equation*}
\end{proposition}
\begin{proof}
For each fixed $m\in\N$ and $N\in\N$, choose unimodular scalars
$\lambda_1,\dots,\lambda_N$ so that
\[
\sum_{n=1}^N |e_m^*Te_n|
= \Bigl|e_m^*T\Bigl(\sum_{n=1}^N \lambda_n e_n\Bigr)\Bigr|
\le \norm{T}.
\]
Hence
\[
\sum_{n\in\N}|e_m^*Te_n|\le \norm{T}.
\]
Therefore, for every $x=(x(n))_{n\in\N}\in\ell_\infty$, the series
\[
\sum_{n\in\N} x(n)\, e_m^*Te_n
\]
is absolutely convergent, and we may define
\[
\Lambda_m(x)=e_m^*Tx-\sum_{n\in\N} x(n)\,e_m^*Te_n,
\qquad
\Theta(m,x)=|\Lambda_m(x)|.
\]
Clearly,
\[
\Theta(m,x)\le 2\norm{T}\norm{x}_{\ell_\infty}
\qquad (m\in\N,\ x\in\ell_\infty).
\]

If $x$ is finitely supported, then by linearity
\[
Tx=\sum_{n\in\N} x(n)\,Te_n,
\]
so $\Lambda_m(x)=0$. Consequently, for every $N\in\N$,
\[
\Lambda_m(x)=\Lambda_m\bigl(P_{[N,\infty)}x\bigr),
\]
because $x-P_{[N,\infty)}x$ is finitely supported. Thus $\Theta(m,x)$ depends only on the tail of $x$.

Moreover, given $x,y\in\ell_\infty$, choose a unimodular scalar $\lambda\in\mathbb K$
so that $\Lambda_m(x)$ and $\lambda\Lambda_m(y)$ have the same argument (or one of them is zero).
Then
\[
\Theta(m,x+\lambda y)
=|\Lambda_m(x)+\lambda\Lambda_m(y)|
=|\Lambda_m(x)|+|\Lambda_m(y)|
=\Theta(m,x)+\Theta(m,y).
\]

Now let $(M_\gamma)_{\gamma\in\Gamma}$ be an uncountable almost disjoint family of infinite subsets of $\N$.
For $m\in\N$ and $\varepsilon>0$, define
\[
\Gamma(m,\varepsilon)=
\Bigl\{\gamma\in\Gamma:\exists x\in B_{\ell_\infty(M_\gamma)}\text{ with }\Theta(m,x)>\varepsilon\Bigr\}.
\]
We claim that $\Gamma(m,\varepsilon)$ is finite.

Indeed, suppose that $\gamma_1,\dots,\gamma_k\in\Gamma(m,\varepsilon)$ are distinct.
Choose $x_j\in B_{\ell_\infty(M_{\gamma_j})}$ with $\Theta(m,x_j)>\varepsilon$.
Since the family is almost disjoint, there exists $N\in\N$ such that the vectors
$P_{[N,\infty)}x_1,\dots,P_{[N,\infty)}x_k$ have pairwise disjoint supports.
Replacing each $x_j$ by $P_{[N,\infty)}x_j$, we preserve the inequality
$\Theta(m,x_j)>\varepsilon$ and may therefore assume from the outset that the supports are pairwise disjoint.

Choose unimodular scalars $\lambda_1,\dots,\lambda_k$ inductively so that
\[
\Theta\Bigl(m,\sum_{j=1}^k \lambda_j x_j\Bigr)=\sum_{j=1}^k \Theta(m,x_j).
\]
Because the supports are disjoint and each $\norm{x_j}_\infty\le 1$, the vector
\[
 x=\sum_{j=1}^k \lambda_j x_j
\]
belongs to $B_{\ell_\infty}$. Hence
\[
2\norm{T}
\ge \Theta(m,x)
=\sum_{j=1}^k \Theta(m,x_j)
>k\varepsilon.
\]
Thus $k<2\norm{T}/\varepsilon$, proving that $\Gamma(m,\varepsilon)$ is finite.

Finally, set
\[
\widetilde\Gamma=\bigcup_{m\in\N}\bigcup_{r\in\N}\Gamma(m,1/r).
\]
This is countable, while $\Gamma$ is uncountable, so choose
$\gamma_0\in\Gamma\setminus\widetilde\Gamma$ and put $M=M_{\gamma_0}$.
Then for every $m\in\N$ and every $x\in\ell_\infty(M)$ we have $\Theta(m,x)=0$, that is,
\[
 e_m^*Tx=\sum_{n\in\N}x(n)\,e_m^*Te_n.
\]
This is exactly the desired conclusion.
\end{proof}

The previous proposition shows that, after restricting to a suitable isometric copy of $\ell_\infty$, the operator $T$ can be represented as an infinite matrix. Equivalently, this restriction of $T$ may be viewed as a weak$^*$–to–weak$^*$ continuous operator. As a consequence, we obtain the following corollary, which, to the best of our knowledge, has not been noted in the literature. In the following, we see $P_M$ as landing inside of $\ell_\infty(M)$ in the natural way.

\begin{corollary}\label{cor: weak-star-continuity}
    Let $T\colon \ell_\infty \to \ell_\infty$. Then there exists an infinite subset $M \subseteq \N$ such that $P_M T|_{\ell_\infty(M)} \colon \ell_\infty(M) \to \ell_\infty(M)$ is weak$^*$-to-weak$^*$ continuous, where $\ell_\infty(M)$ is seen as the dual space of $\ell_1(M)$.
\end{corollary}
\begin{proof}
Let $M\subseteq\N$ be given by Proposition~\ref{prop: l_inf-almost-basis-operators}, and for $m,n\in M$ set
\[
a_{m,n}=e_m^*Te_n.
\]
By the first part of the proof of Proposition~\ref{prop: l_inf-almost-basis-operators},
for each fixed $m\in M$ we have
\[
\sum_{n\in M}|a_{m,n}|\le \norm{T}.
\]
Define $S\colon \ell_1(M)\to \ell_1(M)$ by
\[
(Sa)(n)=\sum_{m\in M} a_m a_{m,n},
\qquad a=(a_m)_{m\in M}\in\ell_1(M),\ \ n\in M.
\]
Then
\[
\norm{Sa}_{\ell_1(M)}
\le \sum_{n\in M}\sum_{m\in M}|a_m|\,|a_{m,n}|
= \sum_{m\in M}|a_m|\sum_{n\in M}|a_{m,n}|
\le \norm{T}\,\norm{a}_{\ell_1(M)},
\]
so $S\in\mathcal B(\ell_1(M))$.

Now let $x\in\ell_\infty(M)$ and $a\in\ell_1(M)$. By Proposition~\ref{prop: l_inf-almost-basis-operators},
\begin{align*}
\langle P_MTx,a\rangle
&= \sum_{m\in M} a_m e_m^*Tx
 = \sum_{m\in M} a_m \sum_{n\in M} x(n)a_{m,n} \\
&= \sum_{n\in M} x(n)\sum_{m\in M} a_m a_{m,n}
 = \sum_{n\in M} x(n)(Sa)(n)
 = \langle x,Sa\rangle.
\end{align*}
Thus
\[
P_MT|_{\ell_\infty(M)}=S^*,
\]
so that $P_MT|_{\ell_\infty(M)}$ is weak$^*$-to-weak$^*$ continuous.
\end{proof}

In particular, one may simply invoke the fact that $\ell_1$ has the $2^+$-PFP property to deduce the corresponding result for $\ell_\infty$. Indeed, given an operator $T \colon \ell_\infty \to \ell_\infty$, applying Corollary~\ref{cor: weak-star-continuity} and restricting to a suitable isometric copy of $\ell_\infty$, we may assume that $T$ is weak$^*$–to–weak$^*$ continuous. Consequently, there exists an operator $S \colon \ell_1 \to \ell_1$ such that $T = S^*$. Since $\ell_1$ has the $2^+$-PFP property, either $S$ or $\Id_{\ell_1} - S$ factors the identity on $\ell_1$ with appropriate constant. Dualising, it follows that the same holds for either $T$ or $\Id_{\ell_\infty} - T$. It is conceivable that similar arguments apply to spaces of the form $X(\ell_\infty)$ and $\ell_\infty(X)$ whenever $X$ is a dual space; however, we do not pursue this direction here.

Nevertheless, since we wish to emphasise the method based on bases, we provide a different proof using the previous results. Observe that the argument is now essentially the same as that of Theorem~\ref{th: Cassaza-Lin}.

\begin{theorem}\label{th: UPFP_l_infty}
    The space $\ell_\infty$ has the $2^+$-PFP.
\end{theorem}
\begin{proof}
Let $T\colon \ell_\infty\to \ell_\infty$ and fix $\varepsilon>0$.
Choose $\delta>0$ small enough so that
\[
\frac{2}{1-\delta}<2+\varepsilon.
\]

By Proposition~\ref{prop: l_inf-almost-basis-operators}, after passing to an isometric $1$-complemented copy of $\ell_\infty$ and reindexing, we may assume that
\begin{equation}\label{eq: l_infty_matrix_behaviour_patch}
    e_m^*Tx = \sum_{n\in\N} x(n)\,e_m^*Te_n
    \qquad(m\in\N,\ x\in\ell_\infty).
\end{equation}

For each $n\in\N$,
\[
1=e_n^*(\Id_{\ell_\infty}-T)e_n+e_n^*Te_n.
\]
Hence, there exists an infinite subset $M\subseteq\N$ such that either
$|e_n^*Te_n|\ge \tfrac12$ for all $n\in M$, or
$|e_n^*(\Id_{\ell_\infty}-T)e_n|\ge \tfrac12$ for all $n\in M$.
Replacing $T$ by $\Id_{\ell_\infty}-T$ if necessary, and then reindexing, we may assume that
\[
|e_n^*Te_n|\ge \tfrac12
\qquad(n\in\N).
\]

Set $x_n=Te_n$. Then for every finite set $F\subseteq\N$ and every scalar family $(\lambda_n)_{n\in F}$,
\[
\Bigl\|\sum_{n\in F} \lambda_n x_n\Bigr\|
= \Bigl\|T\Bigl(\sum_{n\in F} \lambda_n e_n\Bigr)\Bigr\|
\le \norm{T}\,\max_{n\in F}|\lambda_n|.
\]
Therefore, by Lemma~\ref{lmm: l_inf-separate-interference}, after passing to a further subsequence and reindexing, we may assume that
\[
\sum_{m\neq n}|e_n^*Te_m|<\frac{\delta}{2}
\qquad(n\in\N).
\]

Define the diagonal operator $D\colon \ell_\infty\to\ell_\infty$ by
\[
Dx=\Bigl(\frac{x(n)}{e_n^*Te_n}\Bigr)_{n\in\N}.
\]
Since $|e_n^*Te_n|\ge \tfrac12$, we have $\norm{D}\le 2$.

Let $x\in B_{\ell_\infty}$ and $n\in\N$. Using \eqref{eq: l_infty_matrix_behaviour_patch},
\begin{align*}
 e_n^*(\Id_{\ell_\infty}-DT)x
 &= x(n)-\frac{1}{e_n^*Te_n}\sum_{m\in\N}x(m)e_n^*Te_m \\
 &= -\sum_{m\neq n}x(m)\,\frac{e_n^*Te_m}{e_n^*Te_n}.
\end{align*}
Hence
\[
|e_n^*(\Id_{\ell_\infty}-DT)x|
\le 2\sum_{m\neq n}|e_n^*Te_m|
<\delta.
\]
Taking the supremum over $x\in B_{\ell_\infty}$ and $n\in\N$, we obtain
\(
\norm{\Id_{\ell_\infty}-DT}<\delta.
\)
Thus $DT$ is invertible; writing $R=(DT)^{-1}$, we have
\[
\norm{R}\le \frac{1}{1-\delta}
\qquad\text{and}\qquad
RDT=\Id_{\ell_\infty}.
\]
Therefore, on the selected copy, $\Id_{\ell_\infty}$ factors through either $T$ or $\Id_{\ell_\infty}-T$, and the factorisation constant is bounded by
\[
\norm{RD}\le \frac{2}{1-\delta}<2+\varepsilon.
\]
Composing with the canonical isometry onto the selected copy, its inclusion, and the norm-one coordinate projection gives the same bound for the original operator. This proves that $\ell_\infty$ has the $2^+$-PFP.
\end{proof}

\subsubsection{Iterated construction with $\ell_\infty$}\label{subsec: l_infty-iterated} We now argue that, in the same spirit as the iterated constructions available for spaces with a symmetric basis, an analogous procedure can be carried out for spaces of the form $X(\ell_\infty)$, where $X$ is a Banach space with a symmetric basis and symmetric norm. Throughout this subsubsection we assume that the basis of $X$ is not equivalent to the unit vector basis of $\ell_1$.

Before proceeding, we recall the relevant primariness results. The primariness of $\ell_p(\ell_\infty)$ for $1 \le p < \infty$ was established by Casazza, Kottman and Lin \cite{CKL1976}, while the primariness of $\ell_\infty(X)$, when $X$ has a symmetric basis, was proved by Capon \cite{Capon1982PLMS}.

The idea is to replicate the reasoning in Subsection \ref{subsec: symmetric-basis-separable-iterated}, but using the techniques developed in Subsection \ref{subsec: l_infty-case-separable} to bypass the lack of a Schauder basis in the space. Thus, let $(e_n)_{n \in \N}$ be the canonical unit vectors in $\ell_\infty$ and for each $(i,j) \in \N^{2}$ we let $e_{(i,j)} = (0, \dots, 0, e_j, 0, \dots) \in X(\ell_\infty)$ where the unit vector $e_j \in \ell_\infty$ is placed in the $i$th-coordinate. Naturally, we denote by $e_{(i,j)}^*$ the associated coordinate functional, so that $e^*_{(i,j)} e_{(r,k)} = 1$ if $(i,j) = (r,k)$ and $0$ otherwise. For $x \in X(\ell_\infty)$ we define the support of $x$ by $\supp{(x)} = \{(i,j) \in \N^2: e^*_{(i,j)}x \not = 0 \}$. For any $L, M \subseteq \N$, we let
\begin{align*}
X(L)(\ell_\infty(M))
    &= \{ x \in X(\ell_\infty) : \supp(x) \subseteq L \times M \} \\
    &= \{ (x_n)_{n \in \mathbb{N}} \in X(\ell_\infty) :
       x_n = 0 \text{ for } n \notin L
       \text{ and } x_n \in \ell_\infty(M) \text{ for } n \in L \}.
\end{align*}
and observe that naturally $X(L)(\ell_\infty(M))$ is $1$-complemented in $X(\ell_\infty)$. 

As in the case of $\ell_\infty$, we begin by showing that, after passing to a subsequence, the off-diagonal interactions can be made arbitrarily small. We first introduce a preliminary lemma. Throughout, for $(r,s) \in \N^2$, $i \in \N$, and $\varepsilon > 0$, define
\begin{equation*}
    M((r,s), i, \varepsilon) = \{ j \in \N : |e^*_{(i,j)} T e_{(r,s)}| > \varepsilon \}.
\end{equation*}
The set $M((r,s), i, \varepsilon)$ records the indices in the $i$-th row where the image of $e_{(r,s)}$ under $T$ has magnitude exceeding $\varepsilon$.

\begin{lemma}\label{lmm: preliminary-l-infinity-iterated}
    Let $X$ be a Banach space with a symmetric basis and symmetric norm whose basis is not equivalent to the unit vector basis of $\ell_1$, and let $T\colon X(\ell_\infty)\to X(\ell_\infty)$ be an operator. Then the following hold:
    \begin{enumerate}[label = (\alph*), ref = (\alph*)]
        \item \label{it: l_infty-iterated-finitely-many-affect} For every $(i,j) \in \N^2$ and every $\varepsilon > 0$, the set
        \begin{equation*}
            \{(r,s) \in \N^2: |e^*_{(i,j)} T e_{(r,s)}| > \varepsilon\}
        \end{equation*}
        is finite.
        \item \label{it: l_infty-iterated-future-rows} For every $r \in \N$, every infinite set $M\subseteq\N$, every $\varepsilon > 0$, every finite set $\{i_1, \dots, i_K\}\subseteq\N$, every choice of infinite sets $M_1, \dots, M_K \subseteq \N$, and every $N_0\in\N$, there exists $s\in M$ with $s>N_0$ such that
        \begin{equation*}
            M_k \setminus M((r,s), i_k, \varepsilon)
        \end{equation*}
        is infinite for every $k = 1, \dots, K$.
        \item \label{it: l_infty-iterated-influence-future} For every $(r,s) \in \N^2$ we have
        \begin{equation*}
            \lim_{i \to \infty} \sup_{j \in \N} |e^*_{(i,j)} T e_{(r,s)}| = 0.
        \end{equation*}
    \end{enumerate}
\end{lemma}
\begin{proof}
    We first prove \ref{it: l_infty-iterated-finitely-many-affect}. Fix $(i,j)\in\N^2$ and $\varepsilon>0$. If there were infinitely many rows $r$ for which some $s$ satisfies $|e_{(i,j)}^*Te_{(r,s)}|>\varepsilon$, then we could choose distinct rows $r_k$ and indices $s_k$ with this property. The vectors $(e_{(r_k,s_k)})_k$ are isometrically equivalent to a subsequence of the basis of $X$, while the bounded functional $e_{(i,j)}^*T$ is bounded below on them. The argument from Lemma~\ref{lmm: sym-basis-affected} would then force that subsequence, and hence by symmetry the whole basis of $X$, to be equivalent to the unit vector basis of $\ell_1$, a contradiction.

    Thus only finitely many outer rows can occur. For each such row $r$, suppose that $s_1,\ldots,s_m$ satisfy $|e_{(i,j)}^*Te_{(r,s_t)}|>\varepsilon$. Choose unimodular scalars $\lambda_t$ so that
    \begin{equation*}
        \lambda_t e_{(i,j)}^*Te_{(r,s_t)}=|e_{(i,j)}^*Te_{(r,s_t)}|.
    \end{equation*}
    The vector $x=\sum_{t=1}^m\lambda_t e_{(r,s_t)}$ lies in a single $\ell_\infty$-fibre and has norm one. Hence
    \begin{equation*}
        \|T\|\geq |e_{(i,j)}^*Tx|
        =\sum_{t=1}^m |e_{(i,j)}^*Te_{(r,s_t)}|
        >m\varepsilon.
    \end{equation*}
    Hence $m\leq \|T\|/\varepsilon$, which proves \ref{it: l_infty-iterated-finitely-many-affect}.

    We prove \ref{it: l_infty-iterated-future-rows}. Suppose the assertion fails. Then there are $r\in\N$, an infinite set $M\subseteq\N$, $\varepsilon>0$, rows $i_1,\ldots,i_K$, infinite sets $M_1,\ldots,M_K\subseteq\N$, and $N_0\in\N$ such that, for every $s\in M$ with $s>N_0$, at least one of the sets
    \begin{equation*}
        M_k\setminus M((r,s),i_k,\varepsilon)
    \end{equation*}
    is finite. Choose $Q\in\N$ with $Q\varepsilon>\|T\|$, and choose distinct $s_1,\dots,s_{KQ}\in M$ with $s_u>N_0$. For at least one $k_0\in\{1,\dots,K\}$, the set $M_{k_0}\setminus M((r,s_u),i_{k_0},\varepsilon)$ is finite for at least $Q$ of these indices; call them $s_{u_1},\dots,s_{u_Q}$. Since $M_{k_0}$ is infinite, choose
    \begin{equation*}
        j_0\in M_{k_0}\cap\bigcap_{q=1}^Q M((r,s_{u_q}),i_{k_0},\varepsilon).
    \end{equation*}
    Pick unimodular scalars $\lambda_q$ aligning the numbers $e_{(i_{k_0},j_0)}^*Te_{(r,s_{u_q})}$. The vector $x=\sum_{q=1}^Q\lambda_q e_{(r,s_{u_q})}$ has norm one, and therefore
    \begin{equation*}
        \|T\|\geq |e_{(i_{k_0},j_0)}^*Tx|
        =\sum_{q=1}^Q |e_{(i_{k_0},j_0)}^*Te_{(r,s_{u_q})}|
        >Q\varepsilon,
    \end{equation*}
    a contradiction.

    Finally, if $y=Te_{(r,s)}$, then the sequence of outer row norms $(\|Q_i y\|_{\ell_\infty})_{i\in\N}$ belongs to $X$. Since the basis of $X$ is Schauder, its coordinate tails tend to zero. Therefore
    \begin{equation*}
        \sup_{j\in\N}|e^*_{(i,j)}Te_{(r,s)}|\leq \|Q_i y\|_{\ell_\infty}\longrightarrow 0,
    \end{equation*}
    proving \ref{it: l_infty-iterated-influence-future}.
\end{proof}

The next selection is the only place where the mixed $X(\ell_\infty)$ geometry requires more care than the ordinary symmetric-basis argument.  Inside a fixed outer row, finite sums of coordinate vectors have uniformly bounded norm, so one cannot simply rely on coordinate-wise decay.  The reservoir condition in Lemma~\ref{lmm: preliminary-l-infinity-iterated}\ref{it: l_infty-iterated-future-rows} is designed precisely to preserve infinitely many admissible future coordinates in each row.

\begin{lemma}\label{lmm: l_inf_cross-interaction}
    Let $X$ be a Banach space with a symmetric basis and symmetric norm whose basis is not equivalent to the unit vector basis of $\ell_1$, and let $T \colon X(\ell_\infty) \to X(\ell_\infty)$ be an operator. Then, for every $\varepsilon > 0$, there exist a strictly increasing sequence $(n_i)_{i \in \N}$ and infinite sets $M_i=\{m_i^j:j\in\N\}$, with $m_i^1<m_i^2<\cdots$, such that for every $(i,j)\in\N^2$,
    \begin{equation*}
        \sum_{\substack{(r,s) \in \N^2 \\ (r,s) \ne (i,j)}} |e_{(n_i, m_i^j)}^* T e_{(n_r, m_r^s)}| < \varepsilon 2^{-i}.
    \end{equation*}
\end{lemma}
\begin{proof}
    Fix an enumeration $(\rho_t)_{t\in\N}$ of $\N^2$ such that $(k,1)$ appears before $(k,2)$, $(k,2)$ before $(k,3)$, and the rows are opened in the order $1,2,\dots$. Write $\rho_t=(k_t,\ell_t)$, and put
    \begin{equation*}
        \alpha(i,r,s)=\varepsilon 2^{-(i+r+s+2)}.
    \end{equation*}

    We construct the selected coordinates recursively. After stage $n$ we have active rows $A_n$, outer indices $n_i$ for $i\in A_n$, selected inner coordinates $m_i^j$ for the pairs $\rho_t=(i,j)$ with $t\leq n$, and infinite reservoirs $R_i^{(n)}\subseteq\N$ for $i\in A_n$. The induction requirements are:
    \begin{enumerate}[label=(\roman*), ref=(\roman*)]
        \item \label{it: linfty-cross-selected} if $\rho_a=(i,j)$ and $\rho_b=(r,s)$ are distinct selected pairs, then
        \begin{equation*}
            |e^*_{(n_i,m_i^j)}Te_{(n_r,m_r^s)}|\leq \alpha(i,r,s);
        \end{equation*}
        \item \label{it: linfty-cross-reservoir} if $\rho_b=(r,s)$ is a selected source pair, $i\in A_n$, and $m\in R_i^{(n)}$, then
        \begin{equation*}
            |e^*_{(n_i,m)}Te_{(n_r,m_r^s)}|\leq \alpha(i,r,s);
        \end{equation*}
        \item every element of $R_i^{(n)}$ is larger than all coordinates already selected in row $i$.
    \end{enumerate}
    These conditions are void at stage $0$.

    Suppose the construction has been carried out up to stage $n$, and let $\rho_{n+1}=(k,\ell)$. If $\ell=1$, choose a new outer index $n_k$ larger than all previously chosen outer indices and so large that, for every previously selected source pair $\rho_b=(r,s)$,
    \begin{equation}\label{eq: linfty-new-row-old-sources}
        \sup_{m\in\N}|e^*_{(n_k,m)}Te_{(n_r,m_r^s)}|\leq \alpha(k,r,s).
    \end{equation}
    This is possible by Lemma~\ref{lmm: preliminary-l-infinity-iterated}\ref{it: l_infty-iterated-influence-future}. Set $\widetilde R=\N$.

    If $\ell>1$, row $k$ is already active. Set $\widetilde R=R_k^{(n)}$. Then \ref{it: linfty-cross-reservoir} gives \eqref{eq: linfty-new-row-old-sources} with $m\in\widetilde R$ in place of the supremum over $\N$.

    We now protect all old target coordinates from the new source coordinate. For every previously selected target pair $\rho_a=(i,j)$ define
    \begin{equation*}
        G_{i,j}=\{m\in\widetilde R: |e^*_{(n_i,m_i^j)}Te_{(n_k,m)}|>\alpha(i,k,\ell)\}.
    \end{equation*}
    Lemma~\ref{lmm: preliminary-l-infinity-iterated}\ref{it: l_infty-iterated-finitely-many-affect} shows that each $G_{i,j}$ is finite. Hence
    \begin{equation*}
        R=\widetilde R\setminus\bigcup_{\rho_a=(i,j),\,a\le n}G_{i,j}
    \end{equation*}
    is infinite, and every $m\in R$ satisfies
    \begin{equation}\label{eq: linfty-old-target-new-source}
        |e^*_{(n_i,m_i^j)}Te_{(n_k,m)}|\leq\alpha(i,k,\ell)
    \end{equation}
    for every previously selected target pair $(i,j)$.

    Let $A_{n+1}=A_n\cup\{k\}$ and define
    \begin{equation*}
        H_i=
        \begin{cases}
            R, & i=k,\\
            R_i^{(n)}, & i\in A_n\setminus\{k\}.
        \end{cases}
    \end{equation*}
    Put $\eta=\min_{i\in A_{n+1}}\alpha(i,k,\ell)$. Apply Lemma~\ref{lmm: preliminary-l-infinity-iterated}\ref{it: l_infty-iterated-future-rows} with source row $n_k$, source reservoir $R$, threshold $\eta$, target rows $\{n_i:i\in A_{n+1}\}$, target reservoirs $(H_i)_{i\in A_{n+1}}$, and with $N_0$ equal to the largest coordinate already selected in row $k$ (or $0$ if none has been selected). We obtain $m_k^\ell\in R$ such that, for every $i\in A_{n+1}$,
    \begin{equation*}
        L_i=H_i\setminus M((n_k,m_k^\ell),n_i,\eta)
    \end{equation*}
    is infinite. Equivalently,
    \begin{equation}\label{eq: linfty-new-source-future-targets}
        |e^*_{(n_i,m)}Te_{(n_k,m_k^\ell)}|\leq \eta\leq\alpha(i,k,\ell)
        \qquad(m\in L_i).
    \end{equation}

    For each $i\in A_{n+1}$, let $p_i$ be the largest coordinate already selected in row $i$ after adding $m_k^\ell$, and set
    \begin{equation*}
        R_i^{(n+1)}=\{m\in L_i:m>p_i\}.
    \end{equation*}
    These reservoirs are infinite. Conditions \ref{it: linfty-cross-selected} and \ref{it: linfty-cross-reservoir} at stage $n+1$ follow from the induction hypotheses, \eqref{eq: linfty-new-row-old-sources}, \eqref{eq: linfty-old-target-new-source}, and \eqref{eq: linfty-new-source-future-targets}.

    This completes the recursion. Let $M_i=\{m_i^j:j\in\N\}$. For every pair of distinct selected coordinates we have
    \begin{equation*}
        |e^*_{(n_i,m_i^j)}Te_{(n_r,m_r^s)}|\leq \varepsilon2^{-(i+r+s+2)}.
    \end{equation*}
    Therefore, for fixed $(i,j)$,
    \begin{align*}
        \sum_{\substack{(r,s)\in\N^2 \\ (r,s)\ne(i,j)}} |e^*_{(n_i,m_i^j)}Te_{(n_r,m_r^s)}|
        &\leq \varepsilon2^{-i-2}\sum_{r=1}^\infty2^{-r}\sum_{s=1}^\infty2^{-s} \\
        &<\varepsilon2^{-i}.
    \end{align*}
    This proves the lemma.
\end{proof}

We argue that, after restricting to a suitable copy, the operator $T$ can be represented as an infinite matrix.

\begin{lemma}\label{lmm: preliminary-l-infty-iterated}
    Let $T \colon X(\ell_\infty) \to X(\ell_\infty)$ be an operator. Then for any $(i,j) \in \N^2$ and any $x = (x_n)_{n \in \N} \in X(\ell_\infty)$ we have
    \begin{equation*}
        \sum_{(r,s) \in \N^2} |(e^*_{(r,s)} x) (e^*_{(i,j)} T e_{(r,s)})| \leq \norm{T} \norm{x}.
    \end{equation*}
    In particular $\sum_{(r,s) \in \N^2} (e^*_{(r,s)} x) (e^*_{(i,j)} T e_{(r,s)})$ is absolutely summable.
\end{lemma}
\begin{proof}
    Let $(i,j) \in \N^2$ and $x \in X(\ell_\infty)$ be fixed. For each $(r,s) \in \N^2$ choose an unimodular scalar $\lambda_{(r,s)}$ such that
    \begin{equation*}
        \lambda_{(r,s)}(e^*_{(r,s)} x) (e^*_{(i,j)} T e_{(r,s)}) = |(e^*_{(r,s)} x) (e^*_{(i,j)} T e_{(r,s)})|. 
    \end{equation*}
    For each $N \in \N$, consider the vector $x_N \in X(\ell_\infty)$ defined coordinate-wise by $e_{(r,s)}^* x_N = \lambda_{(r,s)} e_{(r,s)}^* x$ if $r,s \leq N$ and $e_{(r,s)}^* x_N = 0$ otherwise. Since the basis of $X$ is $1$-unconditional, it is clear that $\norm{x_N} \leq \norm{x}$, so that
    \begin{align*}
        \norm{T} \norm{x} &\geq \norm{T} \norm{x_N} \geq |e^*_{(i,j)} T x_N| = \sum_{r,s \leq N}  \lambda_{(r,s)}(e^*_{(r,s)} x) (e^*_{(i,j)} T e_{(r,s)}) \\
        &= \sum_{r,s \leq N}  |(e^*_{(r,s)} x) (e^*_{(i,j)} T e_{(r,s)})|,
    \end{align*}
    and since this holds for all $N \in \N$, the result follows.
\end{proof}

We now prove the matrix-representation step needed for $X(\ell_\infty)$. The proof is the almost-disjoint-family argument of Proposition~\ref{prop: l_inf-almost-basis-operators}, applied one outer row at a time and then extended by density from finite row support.

\begin{proposition}\label{prop: l_inf-iterated-almost-basis-operators}
    Let $T \colon X(\ell_\infty) \to X(\ell_\infty)$ be an operator. Then there exists an infinite subset $M \subseteq \N$ such that, for all $x \in X(\ell_\infty(M))$ and every $(i,j) \in \N^2$,
    \begin{equation*}
        e^*_{(i,j)} T x = \sum_{(r,s) \in \N^2} (e^*_{(r,s)} x) (e^*_{(i,j)} T e_{(r,s)}).
    \end{equation*}
\end{proposition}
\begin{proof}
For each $(i,j)\in\N^2$, define a bounded linear functional
$\Lambda_{(i,j)}\colon X(\ell_\infty)\to\mathbb K$ by
\[
\Lambda_{(i,j)}(x)
= e^*_{(i,j)}Tx - \sum_{(r,s)\in\N^2}(e^*_{(r,s)}x)(e^*_{(i,j)}Te_{(r,s)}).
\]
Lemma~\ref{lmm: preliminary-l-infty-iterated} shows that the series converges absolutely, and moreover
\[
|\Lambda_{(i,j)}(x)|\le 2\norm{T}\,\norm{x}
\qquad(x\in X(\ell_\infty)).
\]
Also, by linearity, $\Lambda_{(i,j)}$ vanishes on every finitely supported vector.

Let $(M_\gamma)_{\gamma\in\Gamma}$ be an uncountable family of infinite subsets of $\N$ such that
$M_\gamma\cap M_\beta$ is finite whenever $\gamma\neq\beta$.
For $l\in\N$, $(i,j)\in\N^2$ and $\eta>0$, define
\[
\Gamma(l,(i,j),\eta)
=\Bigl\{\gamma\in\Gamma:\ \exists x\in B_{X(\{l\})(\ell_\infty(M_\gamma))}
\text{ with } |\Lambda_{(i,j)}(x)|>\eta\Bigr\}.
\]
We claim that every set $\Gamma(l,(i,j),\eta)$ is finite.

Indeed, assume that $\gamma_1,\dots,\gamma_k\in\Gamma(l,(i,j),\eta)$ are distinct, and choose
$x_t\in B_{X(\{l\})(\ell_\infty(M_{\gamma_t}))}$ with
$|\Lambda_{(i,j)}(x_t)|>\eta$ for $t=1,\dots,k$.
Since the sets $M_{\gamma_t}$ have pairwise finite intersections, there exists a finite set
$F\subseteq\N$ such that the sets $M_{\gamma_t}\setminus F$ are pairwise disjoint.
Let $y_t$ be obtained from $x_t$ by deleting the coordinates in $F$.
Because $x_t-y_t$ is finitely supported, we still have
$\Lambda_{(i,j)}(y_t)=\Lambda_{(i,j)}(x_t)$.
Thus the vectors $y_1,\dots,y_k$ are supported in the same outer coordinate $l$ and have pairwise disjoint inner supports.
In particular,
\[
\Bigl\|\sum_{t=1}^k \lambda_t y_t\Bigr\|\le 1
\]
for every choice of unimodular scalars $(\lambda_t)_{t=1}^k$.
Choose $\lambda_t$ so that
$\lambda_t\Lambda_{(i,j)}(y_t)=|\Lambda_{(i,j)}(y_t)|$.
Then, with $y=\sum_{t=1}^k \lambda_t y_t$, we obtain
\[
2\norm{T}
\ge |\Lambda_{(i,j)}(y)|
= \sum_{t=1}^k |\Lambda_{(i,j)}(y_t)|
> k\eta.
\]
Hence $k<2\norm{T}/\eta$, proving the claim.

Now set
\[
\widetilde\Gamma
= \bigcup_{l\in\N}\ \bigcup_{(i,j)\in\N^2}\ \bigcup_{m\in\N}
\Gamma\bigl(l,(i,j),1/m\bigr).
\]
This is countable. Choose $\gamma_0\in\Gamma\setminus\widetilde\Gamma$ and set $M=M_{\gamma_0}$.
Then $\Lambda_{(i,j)}$ vanishes on each one-row subspace
$X(\{l\})(\ell_\infty(M))$, for every $l\in\N$.
Since $\Lambda_{(i,j)}$ is linear, it vanishes on every finite sum of such subspaces, i.e. on every vector of $X(\ell_\infty(M))$ supported in finitely many rows.
Those vectors are dense in $X(\ell_\infty(M))$ because the outer basis of $X$ is a Schauder basis.
Therefore $\Lambda_{(i,j)}$ vanishes identically on $X(\ell_\infty(M))$.
This is exactly the desired matrix representation.
\end{proof}

Finally, we can prove the following result.

\begin{theorem}\label{th: l_inf-iterated}
    Let $X$ be a Banach space with a symmetric basis and symmetric norm, and assume that this basis is not equivalent to the unit vector basis of $\ell_1$. Then $X(\ell_\infty)$ has the $2^+$-PFP.
\end{theorem}
\begin{proof}
    Let $T\colon X(\ell_\infty)\to X(\ell_\infty)$ and fix $\varepsilon>0$.
    Choose $\delta>0$ so small that
    \[
    \frac{2}{1-\delta}<2+\varepsilon.
    \]
    For each row $i$, colour an index $j$ red if
    $|e_{(i,j)}^*Te_{(i,j)}|\ge \tfrac12$ and blue otherwise.  In the blue case we have
    $|e_{(i,j)}^*(\Id-T)e_{(i,j)}|\ge \tfrac12$.  For every $i$ at least one colour occurs infinitely often, and therefore one of the two colours occurs infinitely often in infinitely many rows.  Passing to those rows, and to infinite subsets of the corresponding fibres, we obtain a $1$-complemented isometric copy of $X(\ell_\infty)$.  Replacing $T$ by $\Id-T$ if necessary, and then reindexing, we may assume that
    \[
        |e_{(i,j)}^*Te_{(i,j)}|\ge \tfrac12
        \qquad ((i,j)\in\N^2).
    \]
    By Proposition~\ref{prop: l_inf-iterated-almost-basis-operators}, after passing to a further $1$-complemented isometric copy, we may also assume that
    \begin{equation}\label{eq: l_infinite-iterated-1-patch}
        e^*_{(i,j)}Tx
        = \sum_{(r,s)\in\N^2}(e^*_{(r,s)}x)(e^*_{(i,j)}Te_{(r,s)})
    \end{equation}
    for all $x\in X(\ell_\infty)$ and all $(i,j)\in\N^2$.
    
    Finally, apply Lemma~\ref{lmm: l_inf_cross-interaction} with parameter
    $\delta/2$, and pass to another $1$-complemented isometric copy so that
    \begin{equation}\label{eq: l_infinite-iterated-2-patch}
    \sum_{\substack{(r,s)\in\N^2\\(r,s)\neq(i,j)}}
    |e^*_{(i,j)}Te_{(r,s)}|
    < \delta\,2^{-i-1}
    \qquad((i,j)\in\N^2).
    \end{equation}
    
    Define the diagonal operator $D\colon X(\ell_\infty)\to X(\ell_\infty)$ by
    \[
    D\bigl(e_{(i,j)}\bigr)=\frac{1}{e^*_{(i,j)}Te_{(i,j)}}e_{(i,j)}
    \qquad((i,j)\in\N^2).
    \]
    Equivalently,
    \[
    D\bigl((x_i(j))_{i,j\in\N}\bigr)
    =\Bigl(\frac{x_i(j)}{e^*_{(i,j)}Te_{(i,j)}}\Bigr)_{i,j\in\N}.
    \]
    Because $|e^*_{(i,j)}Te_{(i,j)}|\ge \tfrac12$, we have $\norm{D}\le 2$.
    
    Let $x\in B_{X(\ell_\infty)}$ and define
    \[
    c_i=\sup_{j\in\N}|e^*_{(i,j)}(\Id_{X(\ell_\infty)}-DT)x|,
    \qquad i\in\N.
    \]
    Using \eqref{eq: l_infinite-iterated-1-patch}, for each $(i,j)\in\N^2$ we get
    \[
    e^*_{(i,j)}(\Id_{X(\ell_\infty)}-DT)x
    = -\sum_{\substack{(r,s)\in\N^2\\(r,s)\neq(i,j)}}
    (e^*_{(r,s)}x)\,\frac{e^*_{(i,j)}Te_{(r,s)}}{e^*_{(i,j)}Te_{(i,j)}}.
    \]
    Hence, by \eqref{eq: l_infinite-iterated-2-patch} and the estimate
    $|e^*_{(r,s)}x|\le 1$, we get
    \[
    c_i\le \delta\,2^{-i}
    \qquad(i\in\N).
    \]
    The $i$-th row of $(\Id_{X(\ell_\infty)}-DT)x$ has sup norm exactly $c_i$, so
    \[
    \norm{(\Id_{X(\ell_\infty)}-DT)x}
    = \Bigl\|\sum_{i=1}^\infty c_i e_i\Bigr\|_X
    \le \delta
    \Bigl\|\sum_{i=1}^\infty 2^{-i}e_i\Bigr\|_X
    \leq \delta.
    \]
    Therefore
    \[
    \norm{\Id_{X(\ell_\infty)}-DT}\le \delta.
    \]
    Thus $DT$ is invertible; writing $R=(DT)^{-1}$, we have
    \[
    \norm{R}\le \frac{1}{1-\delta}
    \qquad\text{and}\qquad
    RDT=\Id_{X(\ell_\infty)}.
    \]
    Consequently, on the selected copy, $\Id_{X(\ell_\infty)}$ factors through either $T$ or
    $\Id_{X(\ell_\infty)}-T$, and the factorisation constant is bounded by
    \[
    \norm{RD}\le \frac{2}{1-\delta}<2+\varepsilon.
    \]
    Composing with the canonical isometry onto the selected copy, its inclusion, and the norm-one coordinate projection gives the same bound for the original operator. This proves that $X(\ell_\infty)$ has the $2^+$-PFP.
\end{proof}

As an automatic consequence, we extend the results of Casazza, Kottman, and Lin \cite{CKL1976} to include the $c_0$-sum of $\ell_\infty$.

\begin{corollary}
    The space $c_0(\ell_\infty)$ is primary.
\end{corollary}
\begin{proof}
    Apply Theorem~\ref{th: l_inf-iterated} with $X=c_0$, whose canonical basis is symmetric.
    Thus $c_0(\ell_\infty)$ has the primary factorisation property.
    Since $c_0(\ell_\infty)\cong c_0(\N,c_0(\ell_\infty))$, Pe{\l}czy\'nski's decomposition method \ref{th: pelcdecmethod} yields that $c_0(\ell_\infty)$ is primary.
\end{proof}

%% file: 3_Uncountable_direct.tex
\section{Uncountable direct sums}\label{sec: uncountable-direct-sums}

\subsection{Overview and main results}

In the previous section we studied spaces with symmetric bases in the separable setting.  In the non-separable setting one has an additional tool: after reducing the action of an operator to countable coordinate supports, set-mapping theorems can be used to pass to a large subfamily on which all off-diagonal coordinate interactions vanish.

There are two different regimes, and it is important not to conflate them.  If both the outgoing and incoming interaction sets of each coordinate are countable, a direct transfinite recursion gives the required free set in ZFC for every uncountable index set.  This is the case for sums of separable coefficient spaces over scalar bases other than \(\ell_1\).  If only one side is known to be countable, we use Hajnal's set mapping theorem with values of size at most \(\omega\), and this requires the ambient cardinal to be strictly larger than \(\omega_1\).  For continuum-indexed sums this latter assumption is exactly \(\neg\CH\).

Let \(E\) be a Banach space with a symmetric basis \((e_n)_{n \in \mathbb{N}}\).  Then \(E\) can be realised as an \(E\)-sum of copies of the scalar field \(\mathbb{K}\), that is,
\begin{equation*}
E = \bigl(\mathbb{K} \oplus \mathbb{K} \oplus \cdots \bigr)_E.
\end{equation*}
Since \(\mathbb{K}\) trivially has the UPFP, the results of Section~\ref{sec: symmetric-like-basis} may be viewed as preservation results for countable sums.  The purpose of this section is to give corresponding transfer principles for uncountable sums of Banach spaces.

Our first result concerns primariness.  The \(\ell_1\)-part is deliberately stated with the cardinal restriction \(\abs{\Gamma}>\omega_1\); when \(\abs{\Gamma}=2^{\aleph_0}\), this is precisely the negation of the continuum hypothesis.  The \(\ell_p\)- and \(c_0\)-parts for separable coefficient spaces are ZFC statements for arbitrary uncountable \(\Gamma\).

\begin{theorem}[Primariness of \(\ell_p\)- and \(c_0\)-sums]\label{thm:A}
Let \(\Gamma\) be an uncountable set with \(\kappa=\abs{\Gamma}\).

\begin{enumerate}[label=\textup{(\roman*)}, ref=\textup{(\roman*)}]
    \item\label{it:thmA-l1}
    Let \(Y\) be a WCG Banach space and suppose that \(\kappa>\omega_1\).  Put
    \[
        X_1=\ell_1(\mathbb N,Y).
    \]
    Assume that either
    \begin{enumerate}
        \nameditem{\textup{(U)}}{it:U1} \(X_1\) is uniformly primary;
        \nameditem{\textup{(CF)}}{it:CF1} \(X_1\) is primary and \(\cf(\kappa)>\omega\).
    \end{enumerate}
    Then \(\ell_1(\Gamma,Y)\) is primary, and uniformly primary in case~\ref{it:U1}.

    \item\label{it:thmA-lpc0}
    Let \(Y\) be separable.  For \(1<p<\infty\), put
    \[
        X_p=\ell_p(\mathbb N,Y),
    \]
    and put \(X_0=c_0(\mathbb N,Y)\).  If the corresponding countable model \(X_p\), respectively \(X_0\), is uniformly primary, then \(\ell_p(\Gamma,Y)\), respectively \(c_0(\Gamma,Y)\), is uniformly primary.  If the corresponding countable model is primary and \(\cf(\kappa)>\omega\), then the corresponding uncountable sum is primary.
\end{enumerate}
\end{theorem}

The cofinality assumption \(\cf(\kappa)>\omega\) in the non-uniform alternatives is used only to extract a full-size family of blocks with uniformly bounded Banach--Mazur constants.  It is not a set-theoretic assumption such as CH; in particular, \(\cf(2^{\aleph_0})>\omega\) by K\"onig's theorem.  The separate hypothesis \(\kappa>\omega_1\) in the WCG \(\ell_1\)-case is different: it is the price paid for using a one-sided countable support reduction.

\begin{remark}\label{rem:cardinal-regimes}
The results below distinguish two mechanisms.  The two-sided countability mechanism applies to sums of separable spaces, over scalar bases not equivalent to \(\ell_1\), and works in ZFC for every uncountable \(\Gamma\).  The one-sided mechanism applies to WCG \(\ell_1\)-sums and to certain non-separable coefficient spaces; it requires \(\abs{\Gamma}>\omega_1\).  Thus, for continuum-indexed WCG sums, the hypothesis becomes exactly \(\neg\CH\).
\end{remark}

In the same spirit, we prove inheritance theorems for the PFP and UPFP.

\begin{theorem}[Primary factorisation property]\label{thm:B}
Let \(\Gamma\) be an uncountable set with \(\kappa = \abs{\Gamma}\), and let \(X\) be a Banach space.  Assume that one of the following holds:
\begin{enumerate}
    \nameditem{\textup{(UF)}}{it:UF2} \(X\) has the uniform primary factorisation property;
    \nameditem{\textup{(CF)}}{it:CF2} \(X\) has the primary factorisation property and \(\cf(\kappa) > \omega\).
\end{enumerate}
Then the following assertions hold.
\begin{enumerate}[label=\textup{(\roman*)}, ref=\textup{(\roman*)}]
    \item\label{it:thmB-l1-wcg}
    If \(\kappa>\omega_1\) and \(X\) is WCG, then \(\ell_1(\Gamma, X)\) has the primary factorisation property, uniformly in case~\ref{it:UF2}.

    \item\label{it:thmB-separable-zfc}
    Let \(E\) be a Banach space with a symmetric basis \((e_\gamma)_{\gamma\in\Gamma}\) not equivalent to the canonical basis of \(\ell_1(\Gamma)\).  If \(X\) is separable, then \(E(\Gamma,X)\) has the primary factorisation property, uniformly in case~\ref{it:UF2}.  This part has no cardinal assumption beyond the uncountability of \(\Gamma\).

    \item\label{it:thmB-countably-total-large}
    Let \(E\) be a Banach space with a symmetric basis \((e_\gamma)_{\gamma\in\Gamma}\) not equivalent to the canonical basis of \(\ell_1(\Gamma)\).  If \(\kappa>\omega_1\) and \(X^*\) contains a countable total set, then \(E(\Gamma,X)\) has the primary factorisation property, uniformly in case~\ref{it:UF2}.
\end{enumerate}
\end{theorem}

\begin{remark}
The space \(\ell_\infty(\Gamma,X)\) is not covered by the arguments below: its canonical unit vectors do not form a Schauder basis in the sense used here, and the countable-support lemmas fail in that setting.  We therefore make no assertion about \(\ell_\infty\)-sums in this section.
\end{remark}

\begin{remark}\label{rem:separable-range}
In the parts of Theorem~\ref{thm:B} where \(\kappa>\omega_1\), the WCG assumption in the \(\ell_1\)-case may be replaced by the requirement that every operator \(S\colon X\to\ell_1(\Gamma)\) has separable range.  Likewise, in the non-\(\ell_1\) \(E\)-sum case one may replace the countable-total-set assumption by the requirement that every operator \(S\colon X\to E\) has separable range.  In both replacements one uses the one-sided outgoing-coordinate alternative in Lemma~\ref{lem:coordinate-diagonalisation}, and therefore the assumption \(\kappa>\omega_1\) remains essential for this proof.  The ZFC separable-coefficient part of Theorem~\ref{thm:B}\ref{it:thmB-separable-zfc} uses a different, two-sided argument.
\end{remark}

\subsection{Applications and consequences}

We now record some immediate consequences of the transfer theorems.

\begin{corollary}[WCG \(\ell_1\)-sums at the continuum under \(\neg\CH\)]\label{cor:wcg-continuum}
Assume \(\neg\CH\), and let \(Y\) be a WCG Banach space.  Put \(\mathfrak c=2^{\aleph_0}\).  If \(\ell_1(\mathbb N,Y)\) is uniformly primary, then \(\ell_1(\mathfrak c,Y)\) is uniformly primary.  If \(\ell_1(\mathbb N,Y)\) is primary, then \(\ell_1(\mathfrak c,Y)\) is primary.
\end{corollary}

\begin{proof}
Under \(\neg\CH\) we have \(\mathfrak c>\omega_1\), while \(\cf(\mathfrak c)>\omega\) by K\"onig's theorem.  The assertion is therefore Theorem~\ref{thm:A}\ref{it:thmA-l1} applied with \(\Gamma\) of cardinality \(\mathfrak c\).
\end{proof}

\begin{corollary}\label{cor:intro-C01}
Assume \(\neg\CH\).  Then the space \(\mathcal{M}[0,1] = C[0,1]^*\) is primary.
\end{corollary}

\begin{proof}
By a classical result (see \cite[Proposition~5.6]{LT2}), the space \(C[0,1]^*\) is isomorphic to \(\ell_1(2^{\aleph_0},L_1[0,1])\).  The space \(L_1[0,1]\) is separable, hence WCG, and
\[
    \ell_1(\mathbb N,L_1[0,1])\simeq L_1[0,1]
\]
is primary.  Since \(\neg\CH\) is equivalent to \(2^{\aleph_0}>\omega_1\), the result follows from Corollary~\ref{cor:wcg-continuum}.  This gives an interesting example of a positive structural result obtained under the mere negation of the continuum hypothesis.
\end{proof}

\begin{remark}\label{rem:c01-not-zfc-from-decomposition}
Corollary~\ref{cor:intro-C01} is not a ZFC consequence of the present direct-sum argument.  The representation
\[
    C[0,1]^*\simeq \ell_1(2^{\aleph_0},L_1[0,1])
\]
writes \(C[0,1]^*\) as an \(\ell_1\)-sum of separable spaces, but separability of the coefficient space does not give the two-sided countability needed in Proposition~\ref{prop:two-sided-countable-free}.  Indeed, if \(X\neq\{0\}\), choose \(x_0\in X\) and \(x_0^*\in X^*\) with \(x_0^*(x_0)=1\).  For a fixed \(\eta_0\in\Gamma\), the formula
\[
    R_{\eta_0}T((x_\gamma)_{\gamma\in\Gamma})
    =
    \sum_{\gamma\in\Gamma} x_0^*(x_\gamma)x_0,
    \qquad
    R_\eta T=0\quad(\eta\neq\eta_0),
\]
defines a bounded operator on \(\ell_1(\Gamma,X)\), and
\[
    R_{\eta_0}TR_\gamma\neq 0
    \qquad(\gamma\in\Gamma).
\]
Thus a single target coordinate may receive non-zero contributions from all source coordinates.  This is precisely the one-sided phenomenon in the \(\ell_1\)-case, and at \(2^{\aleph_0}=\omega_1\) the Hajnal argument cannot be applied.  A ZFC proof of the primariness of \(C[0,1]^*\), if one is available, would therefore have to use additional structure of \(C[0,1]^*\), not merely the above \(\ell_1(2^{\aleph_0},L_1[0,1])\) decomposition.
\end{remark}

\begin{corollary}\label{cor:separable-coeff-zfc}
Let \(\Gamma\) be an uncountable set, let \(X\) be a separable Banach space with the UPFP, and let \(E\) be a Banach space with a symmetric basis \((e_\gamma)_{\gamma\in\Gamma}\) not equivalent to the canonical basis of \(\ell_1(\Gamma)\).  Then \(E(\Gamma,X)\) has the UPFP.  In particular, \(\ell_p(\Gamma,X)\), \(1<p<\infty\), and \(c_0(\Gamma,X)\) have the UPFP.
\end{corollary}

\begin{proof}
This is Theorem~\ref{thm:B}\ref{it:thmB-separable-zfc} in the uniform case.  The scalar bases of \(\ell_p(\Gamma)\), \(1<p<\infty\), and \(c_0(\Gamma)\) are not equivalent to the canonical basis of \(\ell_1(\Gamma)\).
\end{proof}

\begin{corollary}\label{cor:intro-C02}
Let \(\alpha\) be an ordinal.
\begin{enumerate}[label=\textup{(\roman*)}, ref=\textup{(\roman*)}]
    \item\label{it:ordinal-large}
    If \(\abs{\Gamma}>\omega_1\) and \(1\le p<\infty\), then \(\ell_p(\Gamma,C(\alpha))\) has the uniform primary factorisation property.  In particular, it is primary.

    \item\label{it:ordinal-separable-zfc}
    If \(\alpha<\omega_1\), \(\Gamma\) is uncountable, and \(1<p<\infty\), then \(\ell_p(\Gamma,C(\alpha))\) has the uniform primary factorisation property.  The same is true of \(c_0(\Gamma,C(\alpha))\).
\end{enumerate}
\end{corollary}

\begin{proof}
We first reduce to a coefficient space known to have the UPFP.  If \(C(\alpha)\) is not isomorphic to \(C(\xi\cdot n)\) for any uncountable regular ordinal \(\xi\) and any integer \(n\ge 2\), set \(\beta=\alpha\); then \(C(\beta)\) has the UPFP by \cite[Proposition~1.9]{Acuaviva2025}.  In the remaining case, \(C(\alpha)\simeq C(\xi\cdot n)\) for some such \(\xi\) and \(n\), and the standard finite-decomposition of ordinal \(C\)-spaces gives
\[
    \ell_p(\Gamma,C(\alpha))\simeq \ell_p(\Gamma,C(\xi\cdot n))\simeq \ell_p(\Gamma,C(\xi)).
\]
For \(c_0\)-sums the same finite reindexing gives
\[
    c_0(\Gamma,C(\alpha))\simeq c_0(\Gamma,C(\xi\cdot n))\simeq c_0(\Gamma,C(\xi)).
\]
Set \(\beta=\xi\) in this case.  Thus it is enough, by Proposition~\ref{prop:iso-invariance}, to prove the stated assertions with \(C(\beta)\) in place of \(C(\alpha)\), where \(C(\beta)\) has the UPFP.

Assume first that \(\abs{\Gamma}>\omega_1\) and \(1\le p<\infty\).  We verify the separable-range hypothesis from Remark~\ref{rem:separable-range}.  Let
\[
    S\colon C(\beta)\to \ell_p(\Gamma)
\]
be bounded.  We claim that \(S\) is compact.  By Pe{\l}czy\'nski's theorem \cite{pelczynski1965strictly}, every operator from a \(C(K)\)-space either fixes a copy of \(c_0\) or is weakly compact.  Since the compact ordinal interval \([0,\beta]\) is scattered, \(C(\beta)^*\cong \ell_1([0,\beta])\) has the Schur property.  Hence the adjoint of every weakly compact operator from \(C(\beta)\) is compact, and Schauder's theorem implies that every weakly compact operator from \(C(\beta)\) is compact.

The alternative that \(S\) fixes a copy of \(c_0\) cannot occur.  Indeed, \(\ell_p(\Gamma)\) contains no subspace isomorphic to \(c_0\) for \(1\le p<\infty\): every separable subspace of \(\ell_p(\Gamma)\) is contained in \(\ell_p(\Lambda)\) for some countable \(\Lambda\subseteq\Gamma\), and \(\ell_p(\Lambda)\) contains no copy of \(c_0\).  Consequently \(S\) is compact, and in particular has separable range.  By the separable-range formulation of Theorem~\ref{thm:B}, recorded in Remark~\ref{rem:separable-range}, \(\ell_p(\Gamma,C(\beta))\) has the UPFP.  Since \(\ell_p(\Gamma,C(\beta))\simeq \ell_p(\mathbb N,\ell_p(\Gamma,C(\beta)))\), primariness follows from Pe{\l}czy\'nski's decomposition method.  This proves~\ref{it:ordinal-large}.

Now assume that \(\alpha<\omega_1\), that \(\Gamma\) is uncountable, and that \(1<p<\infty\).  Then the ordinal interval \([0,\alpha]\) is countable and compact, so \(C(\alpha)\) is separable.  Part~\ref{it:ordinal-separable-zfc} follows directly from Corollary~\ref{cor:separable-coeff-zfc}.  The preceding finite-decomposition reduction is unnecessary in this separable case, but it is compatible with it.
\end{proof}

\begin{remark}
The \(\ell_1\)-case for separable coefficient spaces is not included in the ZFC statement above.  Even for scalar coefficients, a bounded functional on \(\ell_1(\Gamma)\) may depend on all coordinates, so the incoming interaction sets need not be countable.  This is exactly why the WCG \(\ell_1\)-results are stated under the one-sided hypothesis \(\abs{\Gamma}>\omega_1\), and why continuum-indexed WCG applications are stated under \(\neg\CH\).
\end{remark}

\subsection{`Small' spaces and support lemmas} 

\subsubsection{Direct sums and coordinate projections}

Let $X$ be a Banach space and let $E$ be a Banach space with a symmetric basis $(e_\gamma)_{\gamma \in \Gamma}$. We denote by $E(\Gamma, X)$ the Banach space of all families $x = (x_\gamma)_{\gamma \in \Gamma}$ with $x_\gamma \in X$ such that
\begin{equation*}
    \sum_{\gamma \in \Gamma} \norm{x_\gamma}_X e_\gamma \in E,
\end{equation*}
equipped with the norm
\begin{equation*}
    \norm{x} = \norm{\sum_{\gamma \in \Gamma} \norm{x_\gamma}_X e_\gamma}_E.
\end{equation*}

In contrast with Section~\ref{sec: symmetric-like-basis}, and to emphasise the role of $\Gamma$, we write $E(\Gamma, X)$ instead of $E(X)$. In this spirit, we write $\ell_p(\Gamma, X)$ and $c_0(\Gamma, X)$ when $(e_\gamma)_{\gamma \in \Gamma}$ is the canonical basis of $\ell_p(\Gamma)$ or $c_0(\Gamma)$, respectively, for $1 \le p < \infty$.

For $x \in E(\Gamma, X)$, we define
\begin{equation*}
\supp(x) = \{ \gamma \in \Gamma : x_\gamma \neq 0 \}.
\end{equation*} 

For $\Delta \subseteq \Gamma$, we denote by $R_\Delta$ the canonical coordinate projection onto the subspace
\begin{equation*}
E(\Delta, X) = \{ x \in E(\Gamma, X) : \supp(x) \subseteq \Delta  \}.
\end{equation*}
By renormalising, we may always assume that $E$ is equipped with a symmetric norm and thus that $R_\Delta$ is a contractive projection, and in particular $E(\Delta, X)$ is $1$-complemented in $E(\Gamma, X)$.

In the case where $\Delta$ is a singleton, say $\Delta = \{\gamma\}$, we use the standard abuse of notation and write $R_\gamma$ for $R_{\{\gamma\}}$ and $E(\gamma, X)$ for $E(\{\gamma\}, X)$.

\subsubsection{Support lemmas} For a bounded operator $T \colon E(\Gamma, X) \to E(\Gamma, X)$, the idea is to separate the influence of a given coordinate in two ways: the influence that a single coordinate has on the others or the influence that the other coordinates have on a single coordinate. This will determine when a space $X$ is ``small", and we shall formalise this through separation lemmas. We start by studying how a given coordinate may influence all the other coordinates. Before doing so, we will need some elementary results.

\begin{lemma}\label{lem:countable-support-lp-c0}
Let $X$ be a Banach space. Then every $x \in E(\Gamma, X)$ has countable support.
\end{lemma}

\begin{proof}
Fix $x = (x_\gamma)_{\gamma \in \Gamma} \in E(\Gamma, X)$. Observe that by definition $(\norm{x_\gamma})_{\gamma \in \Gamma} \in E$, in other words the net of finite partial sums
\begin{equation*}
    \sum_{\gamma \in \Delta} \norm{x_\gamma} e_\gamma \to (\norm{x_\gamma})_{\gamma \in \Gamma} \in E 
    \quad (\Delta \subseteq \Gamma,\ \Delta \text{ finite}).
\end{equation*}
Equivalently, the tails go to zero, that is, for every $\varepsilon > 0$, there exists a finite set $\Delta_\varepsilon \subseteq \Gamma$ such that for every finite $\Delta \subseteq \Gamma \setminus \Delta_\varepsilon$, then
\begin{equation*}
    \norm{\sum_{\gamma \in \Delta} \norm{x_\gamma} e_\gamma} < \varepsilon.
\end{equation*}

For $n \in \mathbb{N}$, by the tail condition, there exists a finite set $\Delta_n \subseteq \Gamma$ such that for every finite $\Delta \subseteq \Gamma \setminus \Delta_n$, we have
\begin{equation*}
    \norm{\sum_{\gamma \in \Delta} \norm{x_\gamma} e_\gamma} < \frac{1}{n}.
\end{equation*}
In particular, since $\norm{e_\gamma} = 1$, i.e the basis is normalised,  the previous condition implies
\begin{equation*}
    \{ \gamma \in \Gamma : \norm{x_\gamma} \geq 1/n\} \subseteq \Delta_n.
\end{equation*}
It follows that
\begin{equation*}
    \supp(x) \subseteq \bigcup_{n \in \mathbb{N}} \Delta_n,
\end{equation*}
which is countable, finishing the proof.
\end{proof}

\begin{lemma}[Separable sets have countable support]\label{lem:separable-support-lp-c0}
    If $Y\subseteq E(\Gamma, X)$ is separable, then there exists a countable $\Delta\subseteq\Gamma$ such that $Y\subseteq E(\Delta, X)$.
\end{lemma}

\begin{proof}
    Choose a dense sequence $(y_n)_{n\in\N}$ in $Y$ and set
    $\Delta=\bigcup_{n\in\N}\supp(y_n)$, which is countable by Lemma~\ref{lem:countable-support-lp-c0}.
    Then $y_n\in E(\Delta, X)$ for all $n$, and $E(\Delta, X)$ is closed in $E(\Gamma, X)$, hence $Y\subseteq E(\Delta, X)$.
\end{proof}

\begin{lemma}\label{lem:scalar-to-vector-support}
Let $E$ be a Banach space with a symmetric basis $(e_\gamma)_{\gamma\in\Gamma}$, and let $X$ be a Banach space.  Assume that for every bounded operator $S \colon X \to E$ there exists a countable set $\Lambda \subseteq \Gamma$ such that $S(X) \subseteq E(\Lambda)$.  Then, for every bounded operator $T \colon X \to E(\Gamma,X)$ there exists a countable set $\Lambda \subseteq \Gamma$ such that $T(X) \subseteq E(\Lambda, X)$.
\end{lemma}

\begin{proof}
Suppose, towards a contradiction, that $T\colon X\to E(\Gamma,X)$ is bounded and that $T(X)$ is not contained in $E(\Lambda,X)$ for any countable $\Lambda\subseteq\Gamma$.  Set
\[
    \Delta=\{\alpha\in\Gamma:\ R_\alpha T x\neq 0\text{ for some }x\in B_X\}.
\]
Then $\Delta$ is uncountable.  For every $\alpha\in\Delta$ choose $x_\alpha\in B_X$ with $R_\alpha T x_\alpha\neq0$.  After passing to an uncountable subset of $\Delta$, there is $\varepsilon>0$ such that
\[
    \norm{R_\alpha T x_\alpha}\ge \varepsilon\qquad(\alpha\in\Delta).
\]
Choose $x_\alpha^*\in B_{X^*}$ with
\[
    x_\alpha^*(R_\alpha T x_\alpha)=\norm{R_\alpha T x_\alpha}
\]
in the real case, and with modulus equal to this norm in the complex case.  Multiplying $x_\alpha^*$ by a scalar of modulus one in the complex case, we may assume the displayed equality is literal.

Define $\widehat T\colon X\to E(\Gamma)$ by
\[
    \widehat T x
    =\sum_{\alpha\in\Delta} x_\alpha^*(R_\alpha T x)e_\alpha.
\]
For each $x\in X$, the scalar family defining $\widehat T x$ is coordinate-wise dominated by $(\norm{R_\alpha Tx})_{\alpha\in\Gamma}$.  Since the norm of $E$ is symmetric, this gives
\[
    \norm{\widehat T x}_{E(\Gamma)}
    \le \norm{Tx}_{E(\Gamma,X)}
    \le \norm T\,\norm x,
\]
so $\widehat T$ is bounded.  On the other hand,
\[
    \widehat T x_\alpha(\alpha)\neq0\qquad(\alpha\in\Delta),
\]
and therefore $\widehat T(X)$ is not contained in $E(\Lambda)$ for any countable $\Lambda\subseteq\Gamma$.  This contradicts the hypothesis.
\end{proof}

We are finally ready to state our first support reduction lemma.

\begin{lemma}[Support reduction]\label{lem:operator-countable-general}
    Let $E$ be a Banach space with a symmetric basis $(e_\gamma)_{\gamma\in\Gamma}$, let $X$ be a Banach space, and let
    \begin{equation*}
        T \colon E(\Gamma,X)\to E(\Gamma,X)
    \end{equation*}
    be a bounded operator. Assume that every bounded operator $S\colon X\to E$ has separable range. Then, for every countable set $\Delta\subseteq\Gamma$, there exists a countable set $\Lambda\subseteq\Gamma$ such that
    \begin{equation*}
        T(E(\Delta,X))\subseteq E(\Lambda,X).
    \end{equation*}
\end{lemma}

\begin{proof}
    For each $\alpha\in\Delta$, let
    \begin{equation*}
        J_\alpha \colon X \to E(\Gamma,X)
    \end{equation*}
    denote the canonical embedding into the $\alpha$-th coordinate. Then $T\circ J_\alpha$ is a bounded operator from $X$ into $E(\Gamma,X)$, so by the assumption and Lemmas~\ref{lem:separable-support-lp-c0} and~\ref{lem:scalar-to-vector-support}, there exists a countable set $\Lambda_\alpha\subseteq \Gamma$ such that
    \begin{equation*}
        T(J_\alpha(X))\subseteq E(\Lambda_\alpha,X).
    \end{equation*}
    Set
    \begin{equation*}
        \Lambda=\bigcup_{\alpha\in\Delta}\Lambda_\alpha.
    \end{equation*}
    Since $\Delta$ is countable and each $\Lambda_\alpha$ is countable, the set $\Lambda$ is countable.

    Let $x=(x_\alpha)_{\alpha\in\Gamma}\in E(\Delta,X)$. Since $\supp(x)\subseteq\Delta$, we may write
    \begin{equation*}
        x=\sum_{\alpha\in\Delta} J_\alpha x_\alpha,
    \end{equation*}
    with convergence in $E(\Gamma,X)$. Hence
    \begin{equation*}
        T x=\sum_{\alpha\in\Delta} T(J_\alpha x_\alpha).
    \end{equation*}
    Each term belongs to $E(\Lambda_\alpha,X)\subseteq E(\Lambda,X)$, and since $E(\Lambda,X)$ is a closed subspace of $E(\Gamma,X)$, it follows that
    \(
        T x\in E(\Lambda,X).
    \)
    Therefore,
    \(
        T(E(\Delta,X))\subseteq E(\Lambda,X),
    \)
    as required.
\end{proof}

As a consequence, we obtain the following.

\begin{corollary}[Support reduction]\label{cor:operator-countable}
    Let $E$ be a Banach space with a symmetric basis $(e_\gamma)_{\gamma\in\Gamma}$, let $X$ be a Banach space, and let
    \begin{equation*}
        T \colon E(\Gamma,X)\to E(\Gamma,X)
    \end{equation*}
    be a bounded operator. Assume that one of the following holds:
    \begin{enumerate}
        \item $X$ is separable;
        \item $E=\ell_1(\Gamma)$ and $X$ is WCG.
    \end{enumerate}
    Then, for every countable set $\Delta\subseteq \Gamma$, there exists a countable set $\Lambda\subseteq \Gamma$ such that
    \begin{equation*}
        T(E(\Delta,X))\subseteq E(\Lambda,X).
    \end{equation*}
\end{corollary}

\begin{proof}
The first assertion is immediate from Lemma~\ref{lem:operator-countable-general},
since every operator from a separable Banach space has separable range.

Assume now that $E=\ell_1(\Gamma)$ and that $X$ is WCG.
Let $S\colon X\to \ell_1(\Gamma)$ be bounded.  Choose a weakly compact set $K\subseteq X$ with $X=\overline{\spn}(K)$.  Then $S(K)$ is weakly compact in $\ell_1(\Gamma)$, and hence norm compact because $\ell_1(\Gamma)$ has the Schur property.  Thus $S(K)$ is separable, and
\[
    \overline{S(X)}=\overline{\spn}(S(K))
\]
is separable.  In particular, $S(X)$ is separable.  Applying Lemma~\ref{lem:operator-countable-general} completes the proof.
\end{proof}

So far, we have explored how one coordinate affects the others. For a space $E$ with a~symmetric basis not equivalent to the unit vector basis of $\ell_1(\Gamma)$, we now establish a further separation result under the assumption that $X^*$ admits a countable total set. Roughly speaking, this condition ensures that each coordinate can be influenced by at most countably many others. Recall that if $X$ is a Banach space, a subset $M \subseteq X^*$ is said to be \emph{total for $X$} if $x^*(x) = 0$ for all $x^* \in M$ implies that $x = 0$. Before stating the separation lemma, we will need the following preliminary result.

\begin{lemma}\label{lmm: disjoint-1}
Let $E$ be a Banach space with a symmetric basis $(e_\gamma)_{\gamma \in \Gamma}$ not equivalent to the unit vector basis of $\ell_1(\Gamma)$, and let $X$ be a Banach space.  Then for every $x^*\in X^*$ and every bounded operator
$T\colon E(\Gamma,X)\to X$ there exists a countable set $\Delta\subseteq \Gamma$ such that
\[
    x^*TR_\gamma=0\qquad(\gamma\in\Gamma\setminus\Delta).
\]
\end{lemma}

\begin{proof}
The assertion is trivial if $x^*=0$ or $T=0$, so assume otherwise.  Suppose that
\[
    A=\{\gamma\in\Gamma:\ x^*TR_\gamma\neq0\}
\]
is uncountable.  Passing to an uncountable subset, choose $\varepsilon>0$ and an uncountable $J\subseteq A$ such that, for every $\gamma\in J$, there is $u_\gamma\in E(\gamma,X)$ with $\norm{u_\gamma}=1$ and
\[
    |x^*(Tu_\gamma)|\ge \varepsilon.
\]
Multiplying $u_\gamma$ by a scalar of modulus one, we may assume that $x^*(Tu_\gamma)=|x^*(Tu_\gamma)|\ge\varepsilon$ for all $\gamma\in J$.

Let $(b_\gamma)_{\gamma\in J}$ be a finitely supported family of non-negative scalars.  Since the vectors $u_\gamma$ are supported on pairwise distinct coordinates and have norm one in their respective fibres,
\[
    \Bigl\|\sum_{\gamma\in J} b_\gamma u_\gamma\Bigr\|_{E(\Gamma,X)}
    =\Bigl\|\sum_{\gamma\in J} b_\gamma e_\gamma\Bigr\|_E.
\]
Consequently,
\[
    \|x^*\|\,\|T\|\,\Bigl\|\sum_{\gamma\in J} b_\gamma e_\gamma\Bigr\|_E
    \ge
    \Bigl|x^*T\Bigl(\sum_{\gamma\in J}b_\gamma u_\gamma\Bigr)\Bigr|
    \ge \varepsilon\sum_{\gamma\in J}b_\gamma.
\]
By symmetry of the basis, this lower estimate is equivalent to
\[
    \Bigl\|\sum_{\gamma\in J} a_\gamma e_\gamma\Bigr\|_E
    \ge
    \frac{\varepsilon}{\|x^*\|\,\|T\|}
    \sum_{\gamma\in J}|a_\gamma|
\]
for every finitely supported scalar family $(a_\gamma)_{\gamma\in J}$.  The reverse estimate follows from the normalisation and the triangle inequality.  Hence $(e_\gamma)_{\gamma\in J}$ is equivalent to the unit vector basis of $\ell_1(J)$.  Since the basis is symmetric, the same two-sided estimates hold, with the same constants, on every finite subset of $\Gamma$.  Thus the full basis of $E$ is equivalent to the unit vector basis of $\ell_1(\Gamma)$, a contradiction.
\end{proof}

\begin{lemma}\label{lmm: disjoint-2}
Let $E$ be a Banach space with a symmetric basis $(e_\gamma)_{\gamma \in \Gamma}$ not equivalent to the unit vector basis of $\ell_1(\Gamma)$, and let $X$ be a Banach space which admits a countable total set $M \subseteq X^*$.  Then for every bounded operator $T \colon E(\Gamma, X) \to X$, there exists a countable set $\Delta \subseteq \Gamma$ such that
\[
    T R_{\gamma} = 0\qquad(\gamma \in \Gamma \setminus \Delta).
\]
\end{lemma}

\begin{proof}
Let $M = \{x_n^*: n \in \mathbb{N}\}$ be an enumeration of a countable total set.  By Lemma~\ref{lmm: disjoint-1}, for each $n$ there is a countable set $\Delta_n\subseteq\Gamma$ such that $x_n^*TR_\gamma=0$ whenever $\gamma\notin\Delta_n$.  Put $\Delta=\bigcup_{n\in\N}\Delta_n$.  If $\gamma\notin\Delta$, then $x_n^*TR_\gamma=0$ for every $n$.  Since $M$ is total, $TR_\gamma=0$.
\end{proof}

\subsection{A combinatorial reduction to diagonal projections}\label{sec:diagonal}

We start with the following elementary observation.

\begin{lemma}[Countable blocks]\label{lem:countable-blocks-lp-c0}
Let $Y$ be a Banach space and $\Delta$ a non-empty countable set.
\begin{enumerate}[label=(\alph*), ref=(\alph*)]
    \item\label{it:lpa} If $1\le p<\infty$ and $X=\ell_p(\N,Y)$, then $\ell_p(\Delta,X)$ is linearly isometric to $X$.
    \item \label{it:lpb} If $X=c_0(\N,Y)$, then $c_0(\Delta,X)$ is linearly isometric to $X$.
\end{enumerate}
\end{lemma}

\begin{proof}
Fix a bijection $\phi\colon \Delta\times\N\to\N$.
In case \ref{it:lpa},
\[
\ell_p(\Delta,X)=\ell_p(\Delta,\ell_p(\N,Y))\cong \ell_p(\Delta\times\N,Y)\cong \ell_p(\N,Y)=X
\]
via reindexing by $\phi$, which preserves the $\ell_p$-norm.
In case \ref{it:lpb} the same reindexing argument applies with the $c_0$-norm, since $\sup$ is preserved under bijections.
\end{proof}

We now present the standard free-set selection theorem. For a set $S$ and a cardinal $\mu$ we write $[S]^{<\mu}$ for the family of subsets of $S$ of cardinality $<\mu$. First, we record the following almost-trivial lemma with a proof for the sake of completeness.

\begin{lemma}[Cofinality pigeonhole]\label{lem:cofinal-pigeon}
Let $\kappa$ be an infinite cardinal with $\cf(\kappa)>\omega$, and let $S$ be a set with $\abs{S}=\kappa$.
If $S=\bigcup_{n\in\N} S_n$, then $\abs{S_n}=\kappa$ for some $n\in\N$.
\end{lemma}

\begin{proof}
If $\abs{S_n}<\kappa$ for all $n$, then $\abs{\bigcup_{n\in\N}S_n}<\kappa$ because $\cf(\kappa)>\omega$.
This contradicts $\abs{S}=\kappa$.
\end{proof}

A classical theorem of Hajnal provides the
combinatorial ingredient needed in our one-sided arguments. The result,
originally conjectured by Ruziewicz~\cite{Ruziewicz1936}
and proved by Hajnal~\cite{Hajnal1961}, is used below in the form where
$F:\kappa\to[\kappa]^{<\lambda}$ for some cardinal $\lambda<\kappa$.
For countable-valued mappings this means that the one-sided argument applies
when $\kappa>\omega_1$.

\begin{proposition}[Hajnal's set mapping theorem]\label{prop:hajnal}
Let $\kappa$ be an infinite cardinal and let $\lambda < \kappa$ be a cardinal. 
Let $F\colon\kappa\to[\kappa]^{<\lambda}$ be a set mapping (i.e., $\alpha \notin F(\alpha)$ for all $\alpha$).
Then there exists $H\subseteq\kappa$ with $\abs{H}=\kappa$ such that for all distinct
$\alpha,\beta\in H$ one has $\alpha\notin F(\beta)$ and $\beta\notin F(\alpha)$.
\end{proposition}
\begin{proof}
For background and further discussion see
\cite[Theorem~46.1]{EHRR} or \cite[\S3.1]{Williams1977}.
\end{proof}

\begin{proposition}[Two-sided countable free set]\label{prop:two-sided-countable-free}
Let $\kappa$ be an uncountable cardinal and let
\[
    F\colon \kappa\to[\kappa]^{\le\omega}
\]
be a set mapping.  Assume, in addition, that every inverse fibre
\[
    F^{-1}(\{\alpha\})
    =
    \{\beta<\kappa:\alpha\in F(\beta)\}
\]
is countable.  Then there is $H\subseteq\kappa$ with $\abs{H}=\kappa$ such that, for all distinct
$\alpha,\beta\in H$,
\[
    \alpha\notin F(\beta)
    \qquad\text{and}\qquad
    \beta\notin F(\alpha).
\]
\end{proposition}

\begin{proof}
We construct $H=\{\alpha_\xi:\xi<\kappa\}$ by recursion.  Suppose that
$\alpha_\eta$ has been chosen for every $\eta<\xi$.  Put
\[
    B_\xi
    =
    \{\alpha_\eta:\eta<\xi\}
    \cup
    \bigcup_{\eta<\xi} F(\alpha_\eta)
    \cup
    \bigcup_{\eta<\xi} F^{-1}(\{\alpha_\eta\}).
\]
Each set appearing in the two unions is countable, and $\xi<\kappa$.  Hence
\[
    \abs{B_\xi}\le \abs{\xi}\cdot\omega<\kappa.
\]
Choose $\alpha_\xi\in\kappa\setminus B_\xi$.  This completes the recursion.

If $\eta<\xi$, then $\alpha_\xi\notin F(\alpha_\eta)$ because
$\alpha_\xi\notin B_\xi$, and also $\alpha_\eta\notin F(\alpha_\xi)$ because
$\alpha_\xi\notin F^{-1}(\{\alpha_\eta\})$.  Thus $H$ is free.
\end{proof}

\begin{remark}\label{rem:omega-one-obstruction}
The one-sided countable-valued hypothesis is not enough at $\kappa=\omega_1$.
Indeed, the map $F(\alpha)=\{\beta:\beta<\alpha\}$ on $\omega_1$ is
countable-valued and has no two-point free set.  Thus every use below of
Hajnal's theorem with merely countable values is made under the explicit
assumption $\kappa>\omega_1$, unless a two-sided countability hypothesis is
available.
\end{remark}

\begin{lemma}\label{lem:block-decomp}
Let $X$ be a WCG Banach space, let $P\colon \ell_1(\Gamma,X)\to\ell_1(\Gamma,X)$ be a bounded projection, and put $\kappa=\abs{\Gamma}$.  Then there exist pairwise disjoint countable sets $(\Gamma_\alpha)_{\alpha<\kappa}$ with
\[
    \Gamma=\bigsqcup_{\alpha<\kappa}\Gamma_\alpha
\]
such that the following hold.
\begin{enumerate}[label=(\roman*), ref=(\roman*)]
    \item \label{item: bd-i} for each $\alpha<\kappa$,
    \[
    P(\ell_1(\Gamma_\alpha, X))\subseteq \ell_1\Bigl(\bigcup_{\beta\le \alpha}\Gamma_\beta, X\Bigr);
    \]
    \item \label{item: bd-ii} for each $\alpha<\kappa$, the operator
    \[
    Q_\alpha:= R_{\Gamma_\alpha} P|_{\ell_1(\Gamma_\alpha, X)}\colon \ell_1(\Gamma_\alpha, X)\to \ell_1(\Gamma_\alpha, X)
    \]
    is a bounded projection;
    \item\label{item: bd-iii} the set mapping
    \[
    F(\alpha)=\Bigl\{\beta<\alpha:\ R_{\Gamma_\beta} P|_{\ell_1(\Gamma_\alpha, X)}\neq 0\Bigr\}
    \]
    takes countable values.
\end{enumerate}
\end{lemma}

\begin{proof}
Fix a well-ordering $(\theta_\xi)_{\xi<\kappa}$ of $\Gamma$.  We construct the sets $\Gamma_\alpha$ by transfinite recursion.  Assume that disjoint countable sets $\Gamma_\beta$ have been chosen for $\beta<\alpha$, and put
\[
    \Delta_\alpha=\bigcup_{\beta<\alpha}\Gamma_\beta.
\]
Then $\abs{\Delta_\alpha}<\kappa$.  Let $\gamma_\alpha$ be the first element of the well-ordering which is not in $\Delta_\alpha$.

Define increasing countable sets $C_n\subseteq\Gamma\setminus\Delta_\alpha$ as follows.  Put $C_0=\{\gamma_\alpha\}$.  Once $C_n$ is known, Corollary~\ref{cor:operator-countable} gives a countable $\Lambda_n\subseteq\Gamma$ such that
\[
    P(\ell_1(C_n,X))\subseteq \ell_1(\Lambda_n,X).
\]
Set
\[
    C_{n+1}=C_n\cup(\Lambda_n\setminus\Delta_\alpha).
\]
Finally put $\Gamma_\alpha=\bigcup_{n\in\N}C_n$.  Then $\Gamma_\alpha$ is countable and disjoint from $\Delta_\alpha$.

The use of the first unused point in the fixed well-ordering ensures that the construction exhausts $\Gamma$.  Indeed, if $\theta_\xi$ were not in $\bigcup_{\alpha<\kappa}\Gamma_\alpha$, then at every stage the first unused point would have index at most $\xi$; this would give $\kappa$ many distinct chosen points among the set $\{\theta_\eta:\eta\le\xi\}$, whose cardinality is strictly smaller than $\kappa$, a contradiction.

We verify the three asserted properties.  Since $\Gamma_\alpha=\bigcup_n C_n$ and $(C_n)$ is increasing,
\[
    \ell_1(\Gamma_\alpha,X)=\overline{\bigcup_{n\in\N}\ell_1(C_n,X)}.
\]
For each $n$,
\[
    P(\ell_1(C_n,X))
    \subseteq \ell_1(\Lambda_n,X)
    \subseteq \ell_1(\Delta_\alpha\cup C_{n+1},X)
    \subseteq \ell_1(\Delta_\alpha\cup\Gamma_\alpha,X).
\]
Taking closures gives
\begin{equation}\label{eq:block-decomp-closure}
    P(\ell_1(\Gamma_\alpha,X))
    \subseteq \ell_1(\Delta_\alpha\cup\Gamma_\alpha,X)
    =\ell_1\Bigl(\bigcup_{\beta\le\alpha}\Gamma_\beta,X\Bigr),
\end{equation}
which is \ref{item: bd-i}.

Fix $\alpha<\kappa$.  From \eqref{eq:block-decomp-closure} applied to every $\beta<\alpha$ and from density of finite block sums, we obtain
\begin{equation}\label{eq:previous-blocks-invariant}
    P(\ell_1(\Delta_\alpha,X))\subseteq \ell_1(\Delta_\alpha,X).
\end{equation}
Let $x\in\ell_1(\Gamma_\alpha,X)$ and write, using \eqref{eq:block-decomp-closure},
\[
    Px=u+v,
    \qquad
    u\in\ell_1(\Delta_\alpha,X),\quad v\in\ell_1(\Gamma_\alpha,X).
\]
Then $v=R_{\Gamma_\alpha}Px=Q_\alpha x$.  Since $P^2=P$, we have $P u+P v=u+v$.  Applying $R_{\Gamma_\alpha}$ and using \eqref{eq:previous-blocks-invariant} yields
\(
    R_{\Gamma_\alpha}P v=v.
\)
Consequently,
\[
    Q_\alpha^2x
    =R_{\Gamma_\alpha}P(R_{\Gamma_\alpha}Px)
    =R_{\Gamma_\alpha}P v
    =v
    =Q_\alpha x,
\]
so $Q_\alpha$ is a projection.

Finally, for each $\alpha$, another application of Corollary~\ref{cor:operator-countable} gives a countable set $\Lambda_\alpha\subseteq\Gamma$ such that
\[
    P(\ell_1(\Gamma_\alpha,X))\subseteq \ell_1(\Lambda_\alpha,X).
\]
Since the blocks $(\Gamma_\beta)_{\beta<\kappa}$ are disjoint, only countably many of them meet $\Lambda_\alpha$.  Hence $F(\alpha)$ is countable.
\end{proof}

\begin{lemma}[A large diagonal family]\label{lem:large-diagonal}
With notation as in Lemma~\ref{lem:block-decomp}, assume in addition that $\kappa>\omega_1$.  Then there exists $H\subseteq\kappa$ with $\abs{H}=\kappa$ such that
\[
R_{\Gamma_\beta} P(\ell_1(\Gamma_\alpha, X))=\{0\}\qquad(\alpha,\beta\in H,\ \alpha\neq \beta).
\]
In particular, on the subspace
\[
W=\Bigl(\bigoplus_{\alpha\in H} \ell_1(\Gamma_\alpha, X)\Bigr)_{\ell_1(H)}\subseteq \ell_1(\Gamma,X),
\]
the operator $P$ is diagonal: if $w=\sum_{\beta\in H}x_\beta\in W$, with $x_\beta\in\ell_1(\Gamma_\beta,X)$, then
\[
R_{\Gamma_\alpha} Pw=Q_\alpha x_\alpha\qquad(\alpha\in H).
\]
\end{lemma}

\begin{proof}
Apply Proposition~\ref{prop:hajnal}, with $\lambda=\omega_1$, to the countable-valued set mapping $F$ from Lemma~\ref{lem:block-decomp}\ref{item: bd-iii}.  We obtain $H\subseteq\kappa$, $\abs{H}=\kappa$, such that for distinct $\alpha,\beta\in H$ one has $\beta\notin F(\alpha)$ and $\alpha\notin F(\beta)$.  If $\alpha<\beta$ are in $H$, Lemma~\ref{lem:block-decomp}\ref{item: bd-i} already gives
\[
    R_{\Gamma_\beta}P|_{\ell_1(\Gamma_\alpha,X)}=0,
\]
whereas $\alpha\notin F(\beta)$ gives
\[
    R_{\Gamma_\alpha}P|_{\ell_1(\Gamma_\beta,X)}=0.
\]
Thus all off-diagonal block entries vanish on the family $H$.

For the final assertion, first check the displayed identity for finite block sums.  The general case follows by approximating $w$ by finite block sums and using continuity of $P$ and of the coordinate projection $R_{\Gamma_\alpha}$.
\end{proof}

\begin{lemma}\label{lema:block-diagonal-lp-c0}
Let $X$ be a separable Banach space, and let $E_0(\Gamma,X)$ be either
$\ell_p(\Gamma,X)$, $1<p<\infty$, or $c_0(\Gamma,X)$.  If
$P\colon E_0(\Gamma,X)\to E_0(\Gamma,X)$ is a bounded projection, then there are
pairwise disjoint countable sets $(\Gamma_\alpha)_{\alpha<\kappa}$ partitioning
$\Gamma$, block projections
\[
    Q_\alpha=R_{\Gamma_\alpha}P|_{E_0(\Gamma_\alpha,X)},
\]
and a set $H\subseteq\kappa$ with $\abs H=\kappa$ such that
\[
    R_{\Gamma_\beta}P(E_0(\Gamma_\alpha,X))=\{0\}
    \qquad(\alpha,\beta\in H,\ \alpha\neq\beta).
\]
\end{lemma}

\begin{proof}
The construction of the countable blocks and the proof that the operators
$Q_\alpha$ are projections are identical to the proof of Lemma~\ref{lem:block-decomp},
using Corollary~\ref{cor:operator-countable} in place of the WCG support
reduction.  Define
\[
    F(\alpha)=\{\beta<\alpha:
    R_{\Gamma_\beta}P|_{E_0(\Gamma_\alpha,X)}\neq0\}.
\]
As before, $F(\alpha)$ is countable for every $\alpha$.

We claim that the inverse fibres of $F$ are also countable.  Fix $\beta<\kappa$.
The range space $E_0(\Gamma_\beta,X)$ is separable, so choose a countable total
set $M_\beta\subseteq E_0(\Gamma_\beta,X)^*$.  For each
$y^*\in M_\beta$, the scalar functional
\[
    y^*R_{\Gamma_\beta}P\colon E_0(\Gamma,X)\to\mathbb K
\]
can be non-zero on only countably many coordinate fibres.  This follows from the same finite-disjoint-support estimate used in the proof
of Lemma~\ref{lmm: disjoint-1}: if such a scalar functional were non-zero on
uncountably many coordinate fibres, then on an uncountable subfamily it would be
bounded below by a fixed positive constant, forcing the scalar basis of $E_0$ to
dominate the $\ell_1$-basis, which is impossible for $\ell_p$, $1<p<\infty$, and
for $c_0$.  Hence the set of blocks $\Gamma_\alpha$ on which
$y^*R_{\Gamma_\beta}P$ is non-zero is countable.  Taking the union over the
countable set $M_\beta$ shows that
\[
    F^{-1}(\{\beta\})
    =\{\alpha<\kappa:\beta\in F(\alpha)\}
\]
is countable.

Proposition~\ref{prop:two-sided-countable-free} now gives a free set
$H\subseteq\kappa$ with $\abs H=\kappa$.  If $\alpha<\beta$ are in $H$, then
triangularity of the block construction gives
$R_{\Gamma_\beta}P|_{E_0(\Gamma_\alpha,X)}=0$, while
$\alpha\notin F(\beta)$ gives
$R_{\Gamma_\alpha}P|_{E_0(\Gamma_\beta,X)}=0$.  Thus all off-diagonal block
entries vanish on $H$.
\end{proof}

\subsection{Proof of Theorem \ref{thm:A}: embedding a full copy into one of the halves}\label{sec:main-proof}

Using the previous results, we are ready for the proof of Theorem \ref{thm:A}.

\begin{proof}[Proof of Theorem \ref{thm:A}]
We give one argument which covers all the asserted sums.  Let $\mathcal E$ denote one of the functors
\[
    \ell_1,\qquad \ell_p\ (1<p<\infty),\qquad c_0,
\]
and put $X=\mathcal E(\N,Y)$ and $Z=\mathcal E(\Gamma,X)$.  Since $\abs{\Gamma\times\N}=\abs{\Gamma}$, the space $Z$ is isometric, by reindexing, to $\mathcal E(\Gamma,Y)$.  In the $\ell_1$-case, $Y$ is WCG and hence $X=\ell_1(\N,Y)$ is WCG.  In the $\ell_p$- and $c_0$-cases, $Y$ is separable and hence so is $X$.  Therefore Lemmas~\ref{lem:block-decomp} and~\ref{lem:large-diagonal} apply in the $\ell_1$-case, while Lemma~\ref{lema:block-diagonal-lp-c0} applies in the remaining cases.

Let $P\colon Z\to Z$ be a bounded projection.  Choose countable blocks $(\Gamma_\alpha)_{\alpha<\kappa}$, block projections
\[
    Q_\alpha=R_{\Gamma_\alpha}P|_{\mathcal E(\Gamma_\alpha,X)}
\]
and a set $H\subseteq\kappa$ of cardinality $\kappa$ as in the corresponding block-diagonalisation lemmas.

For each $\alpha\in H$, the block $\mathcal E(\Gamma_\alpha,X)$ is isomorphic, indeed isometric after reindexing, to $X$ by Lemma~\ref{lem:countable-blocks-lp-c0}.  Put
\(
    L=1+\norm P.
\)
If the relevant countable model $X$ is $C$-primary, where $C\colon[1,\infty)\to[1,\infty)$ is increasing, choose, for each $\alpha\in H$, one of the two projections
\(
    Q_\alpha, I-Q_\alpha
\)
whose range is $C(L)$-isomorphic to $\mathcal E(\Gamma_\alpha,X)$.  This is possible because $\norm{Q_\alpha}\le \norm P\le L$ and the block is isometric to $X$.  If instead $X$ is merely primary, choose for each $\alpha\in H$ one of the two projections $Q_\alpha$ or $I-Q_\alpha$ whose range is isomorphic to $\mathcal E(\Gamma_\alpha,X)$.  Passing to a subset of $H$ of cardinality $\kappa$, we may suppose that the same choice, $Q_\alpha$ or $I-Q_\alpha$, has been made for every remaining $\alpha$.

In the first case set $P_0=P$ and $Q_\alpha^0=Q_\alpha$; in the second set $P_0=I-P$ and $Q_\alpha^0=I-Q_\alpha$.  Then $\norm{P_0}\le L$ and $\norm{Q_\alpha^0}\le L$.

Let
\[
    Z_\alpha=Q_\alpha^0\mathcal E(\Gamma_\alpha,X),
    \qquad
    c_\alpha=d_{BM}\bigl(Z_\alpha,\mathcal E(\Gamma_\alpha,X)\bigr).
\]
In the uniformly primary case the preceding choice gives $c_\alpha\le C(L)$ for all $\alpha$ in the chosen family; fix
\(
    M=C(L)+1.
\)
In the primary, non-uniform case, where $\cf(\kappa)>\omega$, write the chosen family as
\[
    \bigcup_{n\in\N}\{\alpha:\ c_\alpha<n\}.
\]
By Lemma~\ref{lem:cofinal-pigeon}, one of these sets has cardinality $\kappa$; after passing to it, fix $M>1$ such that $c_\alpha<M$ for every remaining $\alpha$.  Thus in either case there are $H''\subseteq H$ with $\abs{H''}=\kappa$ and $M>1$ such that
\begin{equation}\label{eq:uniform-block-BM}
    d_{BM}\bigl(Z_\alpha,\mathcal E(\Gamma_\alpha,X)\bigr)<M
    \qquad(\alpha\in H'').
\end{equation}

Set
\(
    \Yspace=\Bigl(\bigoplus_{\alpha\in H''}Z_\alpha\Bigr)_{\mathcal E(H'')}
    \subseteq Z.
\)
For each $\alpha\in H''$, choose an isomorphism
\[
    \widetilde J_\alpha\colon \mathcal E(\Gamma_\alpha,X)\to Z_\alpha
\]
with $\norm{\widetilde J_\alpha}\norm{\widetilde J_\alpha^{-1}}<M$, and rescale it to $J_\alpha$ so that
\(
    \norm{J_\alpha}=\norm{J_\alpha^{-1}}<\sqrt M.
\)
The block-diagonal operator
\[
    J\colon
    \Bigl(\bigoplus_{\alpha\in H''}\mathcal E(\Gamma_\alpha,X)\Bigr)_{\mathcal E(H'')}
    \longrightarrow \Yspace,
    \qquad
    J((x_\alpha)_\alpha)=(J_\alpha x_\alpha)_\alpha,
\]
satisfies $\norm J\norm{J^{-1}}<M$.  Since $\abs{H''}=\kappa$ and all blocks are countable, the domain of $J$ is isometric to $Z$.  Hence
\begin{equation}\label{eq:Yspace-isomorphic-Z}
    \Yspace\simeq_M Z.
\end{equation}

We now embed this copy complementably into the range of $P_0$.  Define
\[
    T\colon\Yspace\to P_0Z,
    \qquad
    T\Bigl(\sum_{\alpha\in H''}z_\alpha\Bigr)=\sum_{\alpha\in H''}P_0z_\alpha.
\]
The sum converges in the corresponding $\mathcal E$-sum and $\norm T\le\norm{P_0}$.  The off-diagonal block entries of $P_0$ vanish on $H''$, and therefore
\[
    R_{\Gamma_\alpha}T\Bigl(\sum_{\beta\in H''}z_\beta\Bigr)=z_\alpha
    \qquad(\alpha\in H'').
\]
Define
\[
    S\colon P_0Z\to\Yspace,
    \qquad
    Sx=\sum_{\alpha\in H''}Q_\alpha^0(R_{\Gamma_\alpha}x).
\]
For $\ell_p$-sums this is bounded by the usual $\ell_p$ estimate, and for $c_0$ by the corresponding supremum estimate; in all cases
\[
    \norm S\le \sup_{\alpha\in H''}\norm{Q_\alpha^0}\le L.
\]
Moreover $ST=\Id_{\Yspace}$.  Hence $TS$ is a bounded projection on $P_0Z$ onto $T(\Yspace)$, and $T(\Yspace)\simeq Z$ by \eqref{eq:Yspace-isomorphic-Z}.  Thus either $PZ$ or $(I-P)Z$ contains a complemented copy of $Z$.

Since $P_0Z$ is itself complemented in $Z$, Pe{\l}czy\'nski's decomposition method, Theorem~\ref{th: pelcdecmethod}, and the stability $Z\simeq\mathcal E(\N,Z)$ imply that $P_0Z\simeq Z$.  This proves primariness.

In the uniformly primary case, the same proof gives uniform constants.  Indeed, $P_0Z$ is $L$-complemented in $Z$; the subspace $T(\Yspace)$ is at most $L^2$-complemented in $P_0Z$; and, by \eqref{eq:Yspace-isomorphic-Z}, it is $M L^2$-isomorphic to $Z$, where $M=C(L)+1$ depends only on $\norm P$.  Applying the quantitative Pe{\l}czy\'nski decomposition theorem, Theorem~\ref{thm:quant-pelc}, gives a bound for $d_{BM}(P_0Z,Z)$ depending only on $\norm P$.  Therefore the uncountable sum is uniformly primary whenever the countable model is uniformly primary.
\end{proof}

\subsection{Proof of Theorem \ref{thm:B}: primary factorisation property}\label{sec:pfp}

We now prove Theorem~\ref{thm:B}.  The only combinatorial input needed for the factorisation argument is the following coordinate-diagonalisation lemma.  It is useful to record it with two one-sided hypotheses, since the $\ell_1$-case naturally gives control of the coordinates influenced by a fixed source coordinate, whereas the non-$\ell_1$ symmetric-basis case gives control of the source coordinates influencing a fixed target coordinate.

\begin{lemma}[Coordinate diagonalisation]\label{lem:coordinate-diagonalisation}
Let $E$ be a Banach space with a symmetric basis $(e_\gamma)_{\gamma\in\Gamma}$,
let $X$ be a Banach space, let
\[
    T\colon E(\Gamma,X)\to E(\Gamma,X)
\]
be bounded, and put $\kappa=\abs{\Gamma}$.  Assume that one of the following
two alternatives holds.

\begin{enumerate}[label=\textup{(\alph*)}, ref=\textup{(\alph*)}]
    \item\label{it:coord-hajnal-large}
    $\kappa>\omega_1$, and at least one of the two sets
    \[
        \{\eta\in\Gamma:\ R_\eta T R_\gamma\neq0\},
        \qquad
        \{\eta\in\Gamma:\ R_\gamma T R_\eta\neq0\}
    \]
    is countable for every $\gamma\in\Gamma$.

    \item\label{it:coord-two-sided}
    Both sets
    \[
        \{\eta\in\Gamma:\ R_\eta T R_\gamma\neq0\},
        \qquad
        \{\eta\in\Gamma:\ R_\gamma T R_\eta\neq0\}
    \]
    are countable for every $\gamma\in\Gamma$.
\end{enumerate}

Then there exists $\Lambda\subseteq\Gamma$ with $\abs{\Lambda}=\kappa$ such that
\[
    R_\beta T R_\alpha=0
    \qquad(\alpha,\beta\in\Lambda,\ \alpha\neq\beta).
\]
\end{lemma}

\begin{proof}
Assume first~\ref{it:coord-hajnal-large}.  If the outgoing sets are countable,
put
\[
    F(\gamma)=
    \{\eta\in\Gamma\setminus\{\gamma\}:\ R_\eta T R_\gamma\neq0\}.
\]
Then $F(\gamma)\in[\Gamma]^{<\omega_1}$ for every $\gamma$, and
$\omega_1<\kappa$.  Hajnal's set mapping theorem gives a free set
$\Lambda\subseteq\Gamma$ with $\abs{\Lambda}=\kappa$.  For distinct
$\alpha,\beta\in\Lambda$ we have neither $\beta\in F(\alpha)$ nor
$\alpha\in F(\beta)$, which is precisely
\[
    R_\beta T R_\alpha=R_\alpha T R_\beta=0.
\]
The case where the incoming sets are countable is identical, using
\[
    F(\gamma)=
    \{\eta\in\Gamma\setminus\{\gamma\}:\ R_\gamma T R_\eta\neq0\}.
\]

Now assume~\ref{it:coord-two-sided}.  Define
\[
    F(\gamma)=
    \{\eta\in\Gamma\setminus\{\gamma\}:\ R_\eta T R_\gamma\neq0\}.
\]
The sets $F(\gamma)$ are countable by the outgoing hypothesis.  Moreover,
\[
    F^{-1}(\{\gamma\})
    =
    \{\eta\in\Gamma:\ R_\gamma T R_\eta\neq0\}
\]
is countable by the incoming hypothesis.  Proposition~\ref{prop:two-sided-countable-free}
therefore gives a free set $\Lambda\subseteq\Gamma$ of cardinality $\kappa$, and
the same argument as above gives the desired off-diagonal vanishing.
\end{proof}

\begin{lemma}\label{lem:block-decomp1}
Let $X$ be a WCG Banach space, let $T\colon\ell_1(\Gamma, X)\to\ell_1(\Gamma, X)$ be bounded, and put $\kappa=\abs{\Gamma}$.  Assume that $\kappa>\omega_1$.  Then there exists $\Lambda\subseteq\Gamma$ with $\abs{\Lambda}=\kappa$ such that
\[
    R_\beta T R_\alpha=0\qquad(\alpha,\beta\in\Lambda,\ \alpha\neq\beta).
\]
\end{lemma}

\begin{proof}
For each $\gamma\in\Gamma$, Corollary~\ref{cor:operator-countable}, applied to the countable set $\{\gamma\}$, gives a countable set $\Lambda_\gamma\subseteq\Gamma$ such that
\[
    T(\ell_1(\{\gamma\},X))\subseteq\ell_1(\Lambda_\gamma,X).
\]
Thus $\{\eta:\ R_\eta T R_\gamma\neq0\}$ is countable for every $\gamma$.  Lemma~\ref{lem:coordinate-diagonalisation}, under alternative~\ref{it:coord-hajnal-large}, completes the proof.
\end{proof}

\begin{lemma}\label{lem:block-decomp2}
Let $E$ be a Banach space with a symmetric basis $(e_\gamma)_{\gamma\in\Gamma}$ not equivalent to the unit vector basis of $\ell_1(\Gamma)$, let $X$ be a Banach space such that $X^*$ contains a countable total set, let $T\colon E(\Gamma, X)\to E(\Gamma, X)$ be bounded, and put $\kappa=\abs{\Gamma}$.  Assume that $\kappa>\omega_1$.  Then there exists $\Lambda\subseteq\Gamma$ with $\abs{\Lambda}=\kappa$ such that
\[
    R_\beta T R_\alpha = 0\qquad(\alpha,
    \beta\in\Lambda,\ \alpha\neq\beta).
\]
\end{lemma}

\begin{proof}
Fix $\gamma\in\Gamma$.  Applying Lemma~\ref{lmm: disjoint-2} to the operator
\[
    R_\gamma T\colon E(\Gamma,X)\to E(\{\gamma\},X)\cong X
\]
shows that the set
\[
    \{\eta\in\Gamma:\ R_\gamma T R_\eta\neq0\}
\]
is countable.  Hence Lemma~\ref{lem:coordinate-diagonalisation}, under alternative~\ref{it:coord-hajnal-large}, applies.
\end{proof}

\begin{lemma}\label{lem:block-decomp3}
Let $E$ be a Banach space with a symmetric basis $(e_\gamma)_{\gamma\in\Gamma}$ not equivalent to the unit vector basis of $\ell_1(\Gamma)$, let $X$ be separable, and let
\[
    T\colon E(\Gamma,X)\to E(\Gamma,X)
\]
be bounded.  Put $\kappa=\abs{\Gamma}$.  Then there exists $\Lambda\subseteq\Gamma$ with $\abs{\Lambda}=\kappa$ such that
\[
    R_\beta T R_\alpha = 0
    \qquad(\alpha,\beta\in\Lambda,\ \alpha\neq\beta).
\]
\end{lemma}

\begin{proof}
For each $\gamma\in\Gamma$, Corollary~\ref{cor:operator-countable}, applied to the countable set $\{\gamma\}$, gives a countable set $\Lambda_\gamma\subseteq\Gamma$ such that
\[
    T(E(\{\gamma\},X))\subseteq E(\Lambda_\gamma,X).
\]
Thus the outgoing set
\[
    \{\eta\in\Gamma:\ R_\eta T R_\gamma\neq0\}
\]
is countable for every $\gamma\in\Gamma$.

Since $X$ is separable, $X^*$ contains a countable total set.  This is trivial if $X=\{0\}$.  Otherwise, choose a dense sequence $(x_n)$ in the unit sphere of $X$ and, for each $n$, choose $x_n^*\in S_{X^*}$ with $x_n^*(x_n)=1$; then $(x_n^*)$ is total.  Applying Lemma~\ref{lmm: disjoint-2} to $R_\gamma T\colon E(\Gamma,X)\to E(\{\gamma\},X)\cong X$ shows that the incoming set
\[
    \{\eta\in\Gamma:\ R_\gamma T R_\eta\neq0\}
\]
is countable for every $\gamma\in\Gamma$.  Lemma~\ref{lem:coordinate-diagonalisation}, under alternative~\ref{it:coord-two-sided}, completes the proof.
\end{proof}

\begin{proof}[Proof of Theorem \ref{thm:B}]
We prove the assertion for any one of the spaces listed in the theorem.  Thus
\[
    Z=E_0(\Gamma,X),
\]
where one of the following alternatives holds:
\begin{enumerate}[label=\textup{(\alph*)}]
    \item $E_0=\ell_1$, $X$ is WCG, and $\kappa>\omega_1$;
    \item $E_0=E$, where $E$ is not equivalent to $\ell_1(\Gamma)$, and $X$ is separable;
    \item $E_0=E$, where $E$ is not equivalent to $\ell_1(\Gamma)$, $X^*$ contains a countable total set, and $\kappa>\omega_1$.
\end{enumerate}
Let $A\in\Bop(Z)$.  In the three cases above, respectively Lemma~\ref{lem:block-decomp1}, Lemma~\ref{lem:block-decomp3}, and Lemma~\ref{lem:block-decomp2} yield a set $\Lambda\subseteq\Gamma$ with $\abs{\Lambda}=\kappa$ such that
\begin{equation}\label{eq:T-offdiag-zero}
    R_\beta A R_\alpha=0
    \qquad(\alpha,\beta\in\Lambda,\ \alpha\neq\beta).
\end{equation}
For $\alpha\in\Lambda$ write
\[
    X_\alpha=E_0(\{\alpha\},X)
\]
for the $\alpha$-th coordinate copy of $X$, and set
\[
    A_\alpha=R_\alpha A|_{X_\alpha}\in\Bop(X_\alpha).
\]
Since $X_\alpha$ is isometric to $X$, the identity on $X_\alpha$ factors, for each $\alpha$, either through $A_\alpha$ or through $I_{X_\alpha}-A_\alpha$.  One of these two alternatives occurs on a subset of $\Lambda$ of cardinality $\kappa$.  Passing to such a subset, still denoted by $\Lambda$, define
\[
    \widetilde A=
    \begin{cases}
        A, & \text{if the first alternative occurs on }\Lambda,\\
        I_Z-A, & \text{if the second alternative occurs on }\Lambda.
    \end{cases}
\]
The off-diagonal relation \eqref{eq:T-offdiag-zero} remains valid with $A$ replaced by $\widetilde A$, because the identity has no off-diagonal coordinate entries.  Put
\[
    \widetilde A_\alpha=R_\alpha\widetilde A|_{X_\alpha}\qquad(\alpha\in\Lambda).
\]
Then for every $\alpha\in\Lambda$ there are $U_\alpha,V_\alpha\in\Bop(X_\alpha)$ such that
\begin{equation}\label{eq:coordinate-factorisation}
    U_\alpha \widetilde A_\alpha V_\alpha=I_{X_\alpha}.
\end{equation}

If \ref{it:UF2} holds and $X$ has the $C$-PFP, choose the factorisations in \eqref{eq:coordinate-factorisation} so that
\[
    \norm{U_\alpha}\norm{V_\alpha}\le C
    \qquad(\alpha\in\Lambda).
\]
After rescaling, assume $\norm{U_\alpha}=1$ and $\norm{V_\alpha}\le C$; put $M=C$.

If \ref{it:CF2} holds, first rescale each factorisation so that $\norm{U_\alpha}=1$ and put $m_\alpha=\norm{V_\alpha}$.  Since $\cf(\kappa)>\omega$ and
\[
    \Lambda=\bigcup_{n\in\N}\{\alpha\in\Lambda:\ m_\alpha\le n\},
\]
Lemma~\ref{lem:cofinal-pigeon} gives a subset of $\Lambda$ of cardinality $\kappa$ on which $m_\alpha\le M$ for some $M\in\N$.  Replacing $\Lambda$ by this subset, we have in all cases
\begin{equation}\label{eq:coordinate-uniform-bound}
    \norm{U_\alpha}=1,
    \qquad
    \norm{V_\alpha}\le M
    \qquad(\alpha\in\Lambda).
\end{equation}

Set $U_\alpha=V_\alpha=0$ on the coordinates outside $\Lambda$, and define block-diagonal operators $U,V\in\Bop(Z)$ by
\[
    U\Bigl(\sum_{\alpha\in\Gamma}z_\alpha\Bigr)=\sum_{\alpha\in\Gamma}U_\alpha z_\alpha,
    \qquad
    V\Bigl(\sum_{\alpha\in\Gamma}z_\alpha\Bigr)=\sum_{\alpha\in\Gamma}V_\alpha z_\alpha.
\]
The symmetry of the norm gives $\norm U\le1$ and $\norm V\le M$.  By the off-diagonal vanishing for $\widetilde A$ and by \eqref{eq:coordinate-factorisation}, the operator $U\widetilde A V$ is the identity on the coordinate subspace $E_0(\Lambda,X)$.  Indeed, this is immediate on finitely supported vectors, and the general case follows by continuity.

Since $\abs{\Lambda}=\abs{\Gamma}$, symmetry of the basis gives an isometric reindexing isomorphism
\[
    J\colon Z\to E_0(\Lambda,X).
\]
Therefore
\[
    I_Z=J^{-1}(U\widetilde A V)|_{E_0(\Lambda,X)}J=(J^{-1}U)\widetilde A(VJ).
\]
Thus the identity on $Z$ factors through $A$ or through $I_Z-A$, according to the choice of $\widetilde A$.  This proves the PFP.  In case~\ref{it:UF2}, the above construction gives
\[
    \norm{J^{-1}U}\,\norm{VJ}\le C,
\]
so the factorisation property is uniform.
\end{proof}

\begin{remark}
The separable-range variant mentioned in Remark~\ref{rem:separable-range} follows from the same proof.  Let $F$ denote the scalar coordinate space, so that $F=\ell_1(\Gamma)$ in the $\ell_1$-case and $F=E$ in the general symmetric-basis case.  If $\kappa>\omega_1$ and every operator $S\colon X\to F$ has separable range, then Lemma~\ref{lem:operator-countable-general}, applied to one coordinate at a time, gives the outgoing-coordinate hypothesis in alternative~\ref{it:coord-hajnal-large} of Lemma~\ref{lem:coordinate-diagonalisation}; the rest of the proof of Theorem~\ref{thm:B} is unchanged.
\end{remark}

\begin{remark}
As in Theorem~\ref{th: Cassaza-Kottman-Lin}, iterated versions of Theorem~\ref{thm:B} can be obtained.  Since the argument is a routine adaptation of the ideas above, we omit the details.
\end{remark}

The case of $\ell_p(\Gamma,C(\alpha))$ is already covered by Corollary~\ref{cor:intro-C02}, so we do not pursue it further here.

%% file: 4_PFP_other_spaces.tex
\section{PFP and UPFP in other Banach spaces}

In this section, we collect several further examples that do not fit neatly into the symmetric-basis framework of the previous sections. Some of them are naturally handled by existing primariness arguments, while others require a separate factorisation argument. We have grouped them according to the underlying mechanism rather than historical chronology.

\subsection{\texorpdfstring{$\ell_\infty/c_0 \cong C(\omega^*)$}{l_infty/c_0 ≅ C(ω*)}}

\begin{proposition}\label{prop: l_infty-quot-c0}
The space $\ell_\infty/c_0 \cong C(\omega^*)$ has the $2$-PFP.
\end{proposition}

\begin{proof}
Let $C=C(\omega^*)$ and let $T\in\mathscr B(C)$.
By \cite[Theorem~2.1]{DrewnowskiRoberts1991} (applied with $A=\omega^*$), there exist a non-empty clopen $B\subseteq \omega^*$ and a scalar $\gamma$ such that
\[
\chi_B\,T(x)=\gamma x\qquad(x\in C(B)),
\]
where we identify $C(B)$ with the closed subspace
$\{x\in C(\omega^*): x=\chi_B x\}$ of functions supported on $B$.
Equivalently, writing $I_B\colon C(B)\to C$ for the canonical inclusion
(extension by $0$ outside $B$) and $R_B\colon C\to C(B)$ for restriction, we have
\[
R_BT I_B=\gamma\,\Id_{C(B)}.
\]

Since $|\gamma|+|1-\gamma|\ge |\gamma+(1-\gamma)|=1$, we have
$\max\{|\gamma|,|1-\gamma|\}\ge \tfrac12$.

As $B$ is a non-empty clopen subset of $\omega^*$, it is homeomorphic to $\omega^*$
and hence there exists an isometric isomorphism $J\colon C\to C(B)$.
If $|\gamma|\ge \tfrac12$, define $V=I_BJ$ and $U=(1/\gamma)\,J^{-1}R_B$. Then
\[
UTV=(1/\gamma)\,J^{-1}(R_BT I_B)J=\Id_C,
\]
and $\|U\|\,\|V\|\le 2$.

If instead $|1-\gamma|\ge \tfrac12$, we apply the same construction to the operator $\Id_C-T$:
since $R_B(\Id_C-T)I_B=(1-\gamma)\Id_{C(B)}$, we obtain bounded $U,V$ with
$U(\Id_C-T)V=\Id_C$ and $\|U\|\,\|V\|\le 2$.
\end{proof}

\begin{remark}
    As noted in \cite{Drewnowski1989}, if $X$ has the PFP, then so does $C(\omega^*,X)$, and the same argument yields the corresponding statement for the UPFP.
\end{remark}

\subsection{\texorpdfstring{The $p$-James tree space $JT_p$}{The p-James tree space JT_p}}

We start by recalling the following well-known lemma.

\begin{lemma}\label{lem:bb-complemented-implies-factor}
Let $E$ be a Banach space and let $S\colon E\to E$ be bounded.
Assume that there exist complemented subspaces $Y, Z\subseteq E$ and an isomorphism
$V\colon E\to Y$ such that $S|_Y\colon Y\to Z$ is an isomorphism.
Then $\Id_E$ factors through $S$.
\end{lemma}

\begin{proof}
Let $P_Z\colon E\to Z$ be a bounded projection.
Define
\[
U = V^{-1}\,(S|_Y)^{-1}\,P_Z \in \mathscr B(E)
\qquad\text{and}\qquad
W = \iota_Y\circ V \in \mathscr B(E),
\]
where $\iota_Y\colon Y\hookrightarrow E$ is the inclusion map.
Then, since $S(Wx)=S(Vx)\in Z$ and $P_Z$ acts as the identity on $Z$,
\[
USW = V^{-1}(S|_Y)^{-1}P_Z\,S\,V
= V^{-1}(S|_Y)^{-1}S|_Y\,V
= \Id_E.
\]
\end{proof}

Fix $1 < p < \infty$ and write $q$ for the conjugate exponent, $1/p+1/q=1$. Let $\mathcal T$ denote the dyadic tree
\[
\mathcal T=\{(n,i): n\in\mathbb N_0,\ 0\le i<2^n\},
\]
partially ordered by the predecessor relation: $(n,i)\prec (m,j)$ if $(m,j)$
is a strict descendant of $(n,i)$ in the usual dyadic splitting.
A \emph{segment} is a finite chain $S=\{t_1\prec t_2\prec\cdots\prec t_k\}$ in $\mathcal T$,
and a \emph{branch} is an infinite chain (starting at some level).

For each $t\in\mathcal T$ let $x_t$ denote the canonical unit vector in $c_{00}(\mathcal T)$.
The \emph{$p$-James tree space} $JT_p$ is the completion of $c_{00}(\mathcal T)$ for the norm
\begin{equation}\label{eq:JT_p_norm}
\Bigl\|\sum_{t\in\mathcal T} a_t x_t\Bigr\|_{JT_p}
=
\sup\Biggl\{
\Bigl(
\sum_{j=1}^m
\Bigl|\sum_{t\in S_j} a_t\Bigr|^p
\Bigr)^{1/p}
:
S_1,\dots,S_m \ \text{pairwise disjoint segments of }\mathcal T
\Biggr\}.
\end{equation}
For $p=2$ this is James' original tree space $JT$; see Andrew~\cite{Andrew1983} for this formulation
and the associated segment functionals and projections.

For any segment $S\subseteq\mathcal T$ we write $f_S\in (JT_p)^*$ for the \emph{segment functional}
\[
f_S\Bigl(\sum_{t}a_t x_t\Bigr)=\sum_{t\in S}a_t,
\]
so $\|f_S\|=1$ (take the single segment $S$ in \eqref{eq:JT_p_norm}).
If $B$ is a branch, we similarly write $f_B(\sum_t a_t x_t)=\sum_{t\in B}a_t$; this is bounded because
the branch subspace is (isometric to) the corresponding $p$-James space (compare Andrew, Proposition~1(b)).

We shall use the following standard structural facts (proved in Andrew for $p=2$ and valid, with the same proofs,
for $1<p<\infty$ after replacing $\ell_2$-estimates by $\ell_p$-estimates):
\begin{itemize}
\item[(i)] For every subtree $\mathcal S\subseteq\mathcal T$, the subspace $[x_t:t\in\mathcal S]$ is \emph{isometric} to $JT_p$
and is \emph{$1$-complemented} in $JT_p$ by a norm-one projection built from a tree-like family of disjoint segments
(as in Andrew, Proposition~1(a)).
\item[(ii)] If $t_1,\dots,t_m$ are pairwise incomparable nodes, then
\[
\Bigl\|\sum_{k=1}^m a_k x_{t_k}\Bigr\|_{JT_p}=\Bigl(\sum_{k=1}^m |a_k|^p\Bigr)^{1/p},
\]
since every segment intersects $\{t_1,\dots,t_m\}$ in at most one point.
Thus, antichains span isometric copies of $\ell_p$.
\end{itemize}

\begin{theorem}\label{thm:JT_p_UPFP}
For every $1<p<\infty$, the space $JT_p$ has the uniform primary factorisation property.
That is, there exists a constant $C_p\ge 1$ such that for every bounded operator
$T\colon JT_p\to JT_p$ there are $U,V\in\mathscr B(JT_p)$ with either
$UTV=\Id_{JT_p}$ or $U(\Id_{JT_p}-T)V=\Id_{JT_p}$ and $\|U\|\,\|V\|\le C_p$.
\end{theorem}

\begin{proof}
Let $T\in\mathscr B(JT_p)$ and set $U=T$ and $V=\Id-T$.
Fix once and for all a scalar $y\in(0,1/4)$.

Andrew proved for $JT=JT_2$ that there exists a subtree $\mathcal S\subseteq\mathcal T$ such that, writing
$X=[x_t:t\in\mathcal S]$, one has:
\begin{enumerate}
\item[(a)] $X$ is isometric to $JT$ and $1$-complemented in $JT$;
\item[(b)] either $U|_X$ or $V|_X$ is an isomorphism;
\item[(c)] the corresponding range $U(X)$ or $V(X)$ is complemented in $JT$.
\end{enumerate}
Moreover, the proof is quantitative once $y$ is fixed: it produces a uniform lower bound (depending on $y$)
for the diagonal coefficients on the chosen subtree, and a uniform invertibility estimate obtained by a Neumann-series
argument for the branch multipliers (see the construction around Theorem~6 in Andrew~\cite{Andrew1983}).

We \emph{claim} that the same conclusion holds for $JT_p$, $1<p<\infty$, with constants depending only on $p$ (and the fixed $y$).
Indeed, Andrew's proof uses:
\begin{itemize}
\item a dichotomy on each basis vector using a branch functional:
\[
1=f_{B_t}(x_t)=f_{B_t}(Ux_t)+f_{B_t}(Vx_t),
\]
hence for each $t$ either $|f_{B_t}(Ux_t)|\ge 1/2$ or $|f_{B_t}(Vx_t)|\ge 1/2$;
\item a Ramsey-type subtree selection \cite[Proposition~4]{Andrew1983} and a bounded-variation selection on branches \cite[Proposition~5]{Andrew1983};
\item two ``gliding hump'' propositions \cite[Propositions~2-3]{Andrew1983} whose contradiction step exploits that large antichains behave like
$\ell_2$ inside $JT$.
\end{itemize}
In $JT_p$ the antichain estimate in the last bullet is exactly \emph{(ii)} above with $\ell_2$ replaced by $\ell_p$,
so the same contradiction goes through for $p>1$ (since the lower bound grows like $L$ while the upper bound grows like $L^{1/p}$).
All other parts of the argument are combinatorial or use only the segment/branch norming mechanism, hence are unchanged.
Therefore we obtain a subtree $\mathcal S$ and a choice $W\in\{T,\Id-T\}$ such that
$W|_X$ is an isomorphism and $W(X)$ is complemented, where $X=[x_t:t\in\mathcal S]\simeq JT_p$ is $1$-complemented.

Finally, apply Lemma~\ref{lem:bb-complemented-implies-factor} to $E=JT_p$ and $S=W$. We have complemented subspaces $Y=X$ and $Z=W(X)$, and $W|_X\colon X\to W(X)$ is an isomorphism.
The factorisation produced by Lemma~\ref{lem:bb-complemented-implies-factor} yields
$\Id_{JT_p}$ factoring through $W$ with a constant controlled by the projection constants of $X$ and $W(X)$ and by
$\|(W|_X)^{-1}\|$. All these bounds depend only on $p$ (and the fixed $y$), hence one obtains a uniform constant $C_p$
independent of $T$. Thus $\Id_{JT_p}$ factors through either $T$ or $\Id_{JT_p}-T$, with uniform control, \emph{i.e.}\ $JT_p$ has the UPFP.
\end{proof}

\subsection{\texorpdfstring{$\mathcal{B}(\ell_p)$}{B(lp)}}

\begin{remark}
    For context, Blower's theorem that $\Bop(\ell_2)$ is primary and Wark's primariness results for certain $\ell_\infty$-direct sums are formulated in terms of projections and decompositions. They therefore do not, at least not directly, yield the PFP or the UPFP. The argument below for $\Bop(\ell_p)$ is instead genuinely operator-theoretic.
\end{remark}

Fix $p\in(1,\infty)$ and write $q=\frac{p}{p-1}$. For $n\in\N$ let $X_n=\Bop(\ell_p^n)$ with the operator norm.
Set
\[
    \mathcal{M}_p \;=\;\Big(\bigoplus_{n=1}^\infty X_n\Big)_{\ell_\infty},
    \qquad
    \norm{(x_n)}_{\mathcal{M}_p}=\sup_{n\in\N}\norm{x_n}.
\]
We shall use the fact (proved using Pe{\l}czy\'nski's decomposition method in \cite[Section~2]{AriasFarmer})
that $\Bop(\ell_p)$ is (linearly) isomorphic to $\mathcal{M}_p$. Hence, it suffices to prove UPFP for $\mathcal{M}_p$. Christensen and Sinclair \cite{ChristensenSincliar} acknowledge that Lindenstrauss and Haagerup were aware of the existence of such an isomorphism at least for $p=2$.

\subsubsection{Block maps in $\Bop(\ell_p^N)$}
For $N\in\N$ let $(e_i)_{i=1}^N$ be the canonical basis of $\ell_p^N$ and write
$e_{ij}\in X_N$ for the rank-one operator $e_{ij}(e_j)=e_i$ and $e_{ij}(e_k)=0$ for $k\neq j$.
Then $\norm{e_{ij}}=1$.

Given ordered subsets $\alpha=\{\alpha_1<\cdots<\alpha_n\}$ and $\tau=\{\tau_1<\cdots<\tau_n\}$ of $\{1,\dots,N\}$, define
the coordinate embedding $I_\alpha\colon \ell_p^n\to\ell_p^N$ by $I_\alpha(e_k)=e_{\alpha_k}$ and the coordinate projection
$P_\alpha\colon \ell_p^N\to\ell_p^n$ by $P_\alpha(e_{\alpha_k})=e_k$ and $P_\alpha(e_i)=0$ if $i\notin\alpha$.
Then $\norm{I_\alpha}=\norm{P_\alpha}=1$.

Define operators
\[
J_{\alpha,\tau}\colon X_n\to X_N,\qquad J_{\alpha,\tau}(A)=I_\alpha\,A\,P_\tau,
\]
\[
K_{\alpha,\tau}\colon X_N\to X_n,\qquad K_{\alpha,\tau}(B)=P_\alpha\,B\,I_\tau.
\]
They satisfy $\norm{J_{\alpha,\tau}}=\norm{K_{\alpha,\tau}}=1$ and
\[
K_{\alpha,\tau}J_{\alpha,\tau}=\Id_{X_n}.
\]
Hence
\[
Q_{\alpha,\tau}:=J_{\alpha,\tau}K_{\alpha,\tau}\in\Bop(X_N)
\quad\text{is a norm-one projection on the Banach space }X_N,
\]
because $Q_{\alpha,\tau}^2=J_{\alpha,\tau}(K_{\alpha,\tau}J_{\alpha,\tau})K_{\alpha,\tau}=Q_{\alpha,\tau}$.

We also use entry functionals $\varphi_{kl}\in X_N^*$ given by $\varphi_{kl}(A)=e_k^*(A e_l)$; clearly $\norm{\varphi_{kl}}=1$.

\begin{lemma}\label{lem:sums-of-matrixunits}
    Let $N\in\N$ and $I,J\subseteq\{1,\dots,N\}$ be finite non-empty sets.
    Then for unimodular scalars $(\varepsilon_i)_{i\in I}$ and $(\delta_j)_{j\in J}$:
    \[
    \norm{\sum_{i\in I}\varepsilon_i e_{ij_0}} = |I|^{1/p} \quad (\text{fixed } j_0),
    \qquad
    \norm{\sum_{j\in J}\delta_j e_{i_0 j}} = |J|^{1/q} \quad (\text{fixed } i_0).
    \]
\end{lemma}

\begin{proof}
The operator $u=\sum_{i\in I}\varepsilon_i e_{i j_0}$ maps $e_{j_0}$ to $\sum_{i\in I}\varepsilon_i e_i$ and annihilates the other basis vectors,
so $\norm{u}=\norm{\sum_{i\in I}\varepsilon_i e_i}_{\ell_p^N}=|I|^{1/p}$.
Similarly, $v=\sum_{j\in J}\delta_j e_{i_0 j}$ has rank one and is $v(x) = \big(\sum_{j\in J}\delta_j x_j\big)e_{i_0}$,
so $\norm{v}=\norm{(\delta_j)_{j\in J}}_{\ell_q}=|J|^{1/q}$.
\end{proof}

\subsubsection{Finite-dimensional scalarisation (Ramsey step)}

\begin{proposition}[Finite-dimensional scalarisation]\label{prop:finite}
Fix $n\in\N$, $K\ge 1$ and $0<\varepsilon<\tfrac14$.
There exists $N_0=N_0(n,K,\varepsilon,p)\in\N$ such that whenever $N\ge N_0$ and $T\in\Bop(X_N)$ with $\norm{T}\le K$,
there exist disjoint subsets $\alpha,\tau\subseteq\{1,\dots,N\}$ with $|\alpha|=|\tau|=n$ and a scalar $c\in\C$ such that
\[
\norm{K_{\alpha,\tau}\,T\,J_{\alpha,\tau}-c\,\Id_{X_n}}<\varepsilon.
\]
Consequently, one of $K_{\alpha,\tau}TJ_{\alpha,\tau}$ and $K_{\alpha,\tau}(\Id_{X_N}-T)J_{\alpha,\tau}$
is invertible on $X_n$ and its inverse has norm at most $4$.
\end{proposition}

\begin{proof}
The proof is a finite Ramsey selection, following Blower's argument for $p=2$ \cite[\S2, Proposition~1]{Blower1990}.  The only analytic change is that the norms of row and column sums of matrix units are now those computed in Lemma~\ref{lem:sums-of-matrixunits}.

Set $\eta=\delta=\varepsilon/(8n^3)$.  Partition the disc $\{z\in\C: |z|\le K\}$ into finitely many sets of diameter at most $\eta$.

\smallskip\noindent
\emph{Diagonal coefficients.}
Colour each two-element set $\{i,j\}$, $i<j$, by the part containing $\varphi_{ij}(Te_{ij})$.  Ramsey's theorem gives, after choosing the initial Ramsey number sufficiently large, a set $\Gamma_1$ of the required large size on which all these coefficients lie in one part.  Choosing $c$ in that part, we have
\[
    |\varphi_{ij}(Te_{ij})-c|\le \eta
    \qquad (i<j,\\ i,j\in\Gamma_1).
\]

\smallskip\noindent
\emph{Off-diagonal coefficients.}
We now perform finitely many further Ramsey thinnings.  For each possible order pattern of four symbols corresponding to
\[
    \varphi_{kl}(Te_{ij}), \qquad (i,j)\ne(k,l),
\]
with the row indices ultimately chosen from the first half of the final set and the column indices from the second half, colour the relevant finite subsets as bad when
$|\varphi_{kl}(Te_{ij})|\ge\delta$.
A homogeneous bad set of arbitrarily large size is impossible.  The reason is always one of the following two estimates.

If the source column is fixed and the source row varies, then for suitable unimodular scalars $\varepsilon_r$,
\[
K\Big\|\sum_{r=1}^m \varepsilon_r e_{rj}\Big\|
\ge
\Big|\varphi_{kl}\Big(T\sum_{r=1}^m \varepsilon_r e_{rj}\Big)\Big|
\ge m\delta.
\]
By Lemma~\ref{lem:sums-of-matrixunits}, the norm on the left is $m^{1/p}$, so $K\ge\delta m^{1/q}$, impossible for large $m$.

If the source row is fixed and the source column varies, the same argument gives
\[
K\Big\|\sum_{s=1}^m \varepsilon_s e_{is}\Big\|
\ge m\delta,
\]
and now the norm is $m^{1/q}$, so $K\ge\delta m^{1/p}$, again impossible for large $m$.

These two estimates cover the same-column and same-row off-diagonal cases directly.  They also cover the four-distinct-index cases after fixing all but one of the varying source indices in the relevant order pattern.  Since there are only finitely many order patterns, a finite iteration of Ramsey's theorem yields a set $\Gamma$ with $|\Gamma|\ge 2n$ such that
\[
    |\varphi_{kl}(Te_{ij})|<\delta
\]
whenever $(i,j)\ne(k,l)$, with $i,k$ among the first $n$ points of $\Gamma$ and $j,l$ among the last $n$ points of $\Gamma$.

Let $\alpha$ be the first $n$ elements of $\Gamma$ and $\tau$ the last $n$ elements.  For $A=[a_{rs}]\in X_n$ and
$B=K_{\alpha,\tau}TJ_{\alpha,\tau}(A)$, we have
\[
    b_{kl}=\sum_{r,s=1}^n a_{rs}\,
    \varphi_{\alpha_k,\tau_l}\bigl(T e_{\alpha_r,\tau_s}\bigr).
\]
Splitting off the term $(r,s)=(k,l)$ gives
\[
    |b_{kl}-c a_{kl}|
    \le |a_{kl}|\eta+\sum_{(r,s)\ne(k,l)} |a_{rs}|\delta
    \le (\eta+n^2\delta)\|A\|,
\]
because $|a_{rs}|\le\|A\|$ for every entry.  If $U=B-cA$, then each entry of $U$ is bounded by $(\eta+n^2\delta)\|A\|$, and hence
\[
    \|U\|\le n\max_{k,l}|u_{kl}|
    \le n(\eta+n^2\delta)\|A\|<\varepsilon\|A\|.
\]
Thus
\[
    \|K_{\alpha,\tau}TJ_{\alpha,\tau}-c\Id_{X_n}\|<\varepsilon.
\]
Finally, either $|c|\ge1/2$ or $|1-c|\ge1/2$.  Since $\varepsilon<1/4$, a Neumann-series argument shows that one of
$K_{\alpha,\tau}TJ_{\alpha,\tau}$ and $K_{\alpha,\tau}(\Id_{X_N}-T)J_{\alpha,\tau}$ is invertible, with inverse norm at most $4$.
\end{proof}

\subsubsection{A block projection lemma}

For a matrix $A=[a_{ij}]\in X_N$ define its Hilbert--Schmidt norm
\[
\norm{A}_{HS}=\Big(\sum_{i,j=1}^N |a_{ij}|^2\Big)^{1/2}.
\]
(We use this auxiliary norm only for averaging; the operator norm remains the main norm.)

\begin{lemma}\label{lem:hs-vs-op}
Let $N\in\N$ and $A\in X_N$.
Put $a=\big| \frac1p-\frac12\big|<\frac12$. Then
\[
\norm{A}_{HS}\le N^{\frac12+a}\,\norm{A}.
\]
Moreover, if $B\in X_k$ then $\norm{B}\le k^{a}\norm{B}_{HS}$.
\end{lemma}

\begin{proof}
Viewing $A$ as an operator on $\ell_2^N$, we have $\norm{A}_{HS}\le \sqrt{N}\,\norm{A}_{2\to2}$.
Also
\[
\norm{A}_{2\to2}
\le \norm{\Id\colon \ell_2^N\to \ell_p^N}\,\norm{A}\,\norm{\Id\colon \ell_p^N\to \ell_2^N}
= N^{a}\,\norm{A},
\]
by the standard $\ell_p^N$--$\ell_2^N$ inequalities. Multiplying yields the first estimate.

For the second, similarly,
\[
\norm{B}\le \norm{\Id:\ell_2^k\to\ell_p^k}\,\norm{B}_{2\to2}\,\norm{\Id:\ell_p^k\to\ell_2^k}\le k^{a}\,\norm{B}_{2\to2}
\le k^{a}\norm{B}_{HS}.
\]
\end{proof}

\begin{lemma}[Block projection lemma]\label{lem:blockkill}
Fix $k,d\in\N$ and $\varepsilon\in(0,1)$.
There exists $N_0=N_0(k,d,\varepsilon,p)\in\N$ such that for all $N\ge N_0$ and every $d$-dimensional subspace $E\subseteq X_N$
there exist subsets $\alpha,\tau\subseteq\{1,\dots,N\}$ with $|\alpha|=|\tau|=k$
such that the block projection $Q_{\alpha,\tau}=J_{\alpha,\tau}K_{\alpha,\tau}$ satisfies
\[
\norm{Q_{\alpha,\tau} x}\le \varepsilon \norm{x}\qquad(x\in E).
\]
\end{lemma}

\begin{proof}
Let $a=\big|\frac1p-\frac12\big|$.
Let $m=\lfloor N/k\rfloor$ so $m\to\infty$ with $N$.
Partition the first $mk$ indices into consecutive blocks of size $k$:
\[
R_s=\{(s-1)k+1,\dots,sk\}\qquad (1\le s\le m).
\]
For each $(s,t)\in\{1,\dots,m\}^2$ define $\alpha_s=R_s$, $\tau_t=R_t$ and
\[
Q_{s,t}:=Q_{\alpha_s,\tau_t}=J_{\alpha_s,\tau_t}K_{\alpha_s,\tau_t}\in\Bop(X_N).
\]
Each $Q_{s,t}$ is a norm-one projection on $X_N$.

Moreover, the operators $Q_{s,t}(A)$ have disjoint matrix-entry supports inside the $(mk)\times(mk)$ northwest corner.
Hence, for every $A\in X_N$,
\begin{equation}\label{eq:hsdecomp}
\sum_{s,t=1}^m \norm{Q_{s,t}A}_{HS}^2 \le \norm{A}_{HS}^2.
\end{equation}

Let $S_E=\{x\in E:\norm{x}=1\}$. Choose an $(\varepsilon/2)$-net $\{x_1,\dots,x_L\}\subseteq S_E$ in operator norm.
A standard volume estimate, after viewing $E$ as a real space if necessary, yields
$L\le (1+4/\varepsilon)^{2d}$ in the complex case and $L\le (1+4/\varepsilon)^d$ in the real case.
By Lemma~\ref{lem:hs-vs-op}, $\norm{x_r}_{HS}^2\le N^{1+2a}$ for each $r$.
Summing and using \eqref{eq:hsdecomp},
\[
\sum_{s,t=1}^m \sum_{r=1}^L \norm{Q_{s,t}x_r}_{HS}^2
\le \sum_{r=1}^L \norm{x_r}_{HS}^2
\le L\,N^{1+2a}.
\]
Hence for some $(s_0,t_0)$,
\[
\sum_{r=1}^L \norm{Q_{s_0,t_0}x_r}_{HS}^2 \le \frac{L\,N^{1+2a}}{m^2}.
\]

Write $Q=Q_{s_0,t_0}=J_{\alpha,\tau}K_{\alpha,\tau}$ with $|\alpha|=|\tau|=k$.
Then for each $r$,
\[
\norm{Qx_r}=\norm{J_{\alpha,\tau}K_{\alpha,\tau}x_r}\le \norm{K_{\alpha,\tau}x_r}.
\]
Now $K_{\alpha,\tau}x_r\in X_k$, and $K_{\alpha,\tau}$ simply extracts the $k\times k$ block of entries,
so $\norm{K_{\alpha,\tau}x_r}_{HS}=\norm{Qx_r}_{HS}$. Lemma~\ref{lem:hs-vs-op} gives
\[
\norm{Qx_r}\le \norm{K_{\alpha,\tau}x_r}\le k^{a}\norm{K_{\alpha,\tau}x_r}_{HS}=k^{a}\norm{Qx_r}_{HS}.
\]
Therefore,
\[
\max_{1\le r\le L}\norm{Qx_r}^2
\le \sum_{r=1}^L \norm{Qx_r}^2
\le k^{2a}\sum_{r=1}^L \norm{Qx_r}_{HS}^2
\le k^{2a}\frac{L\,N^{1+2a}}{m^2}.
\]
Since $m\asymp N/k$, the right-hand side is bounded by $C(p)\,L\,k^{1+4a}\,m^{-(1-2a)}$, which tends to $0$ as $m\to\infty$
because $1-2a>0$. Choose $N_0$ so large that this bound is at most $(\varepsilon/2)^2$.
Then $\norm{Qx_r}\le \varepsilon/2$ for all $r$.

Finally, given $x\in S_E$ choose $r$ with $\norm{x-x_r}\le \varepsilon/2$.
Since $\norm{Q}=1$,
\[
\norm{Qx}\le \norm{Qx_r}+\norm{Q(x-x_r)}\le \varepsilon/2+\varepsilon/2=\varepsilon.
\]
Thus $Q_{\alpha,\tau}=Q$ works.
\end{proof}

Write $\pi_n\colon \mathcal{M}_p\to X_n$ for the $n$th coordinate projection, and $\iota_n\colon X_n\to\mathcal{M}_p$ for the coordinate embedding.
For $I\subseteq\N$ let $P_I\colon \mathcal{M}_p\to \mathcal{M}_p$ be the coordinate projection
\[
P_I(x_1,x_2,\dots)=(y_1,y_2,\dots),\qquad y_n=\begin{cases}x_n,&n\in I,\\ 0,&n\notin I.\end{cases}
\]
Then $\norm{P_I}=1$.

\begin{lemma}\label{lem:tail-small}
Let $F$ be finite-dimensional, let $S\colon\mathcal M_p\to F$ be bounded, and let $J\subseteq\N$ be infinite.  Then for every $\eta>0$ there exists an infinite set $I\subseteq J$ such that $\|SP_I\|<\eta$.
\end{lemma}

\begin{proof}
    First suppose that $F=\mathbb K$.  Partition $J$ into $m$ disjoint infinite sets $J_1,\ldots,J_m$.  On the band $P_J\mathcal M_p=P_{J_1}\mathcal M_p\oplus_\infty\cdots\oplus_\infty P_{J_m}\mathcal M_p$, the dual norm decomposes as an $\ell_1$-sum, hence
    \[
        \|SP_J\|=\sum_{r=1}^m\|SP_{J_r}\|.
    \]
    Thus some $J_r$ satisfies $\|SP_{J_r}\|\le \|S\|/m$.  Choosing $m$ large proves the scalar case.
    
    For finite-dimensional $F$, choose $f_1,\ldots,f_M\in B_{F^*}$ such that
    \[
        \|y\|\le 2\max_{1\le r\le M}|f_r(y)|\qquad(y\in F).
    \]
    Starting with $I^{(0)}=J$, apply the scalar case successively to the functionals
    $(f_r\circ S)P_{I^{(r-1)}}$.  This gives decreasing infinite sets
    $I^{(0)}\supset I^{(1)}\supset\cdots\supset I^{(M)}$ such that
    \[
        \|(f_r\circ S)P_{I^{(M)}}\|<\eta/2
        \qquad(r=1,\ldots,M).
    \]
    For $I=I^{(M)}$ we obtain $\|SP_I\|<\eta$.
\end{proof}

\subsubsection{Almost diagonalisation on $\mathcal{M}_p$}

\begin{lemma}[Almost diagonalisation]\label{lem:almost-diagonal}
Let $T\in\Bop(\mathcal{M}_p)$ and $\varepsilon>0$.
There exist contractions $A,B\in\Bop(\mathcal{M}_p)$ with $AB=\Id_{\mathcal{M}_p}$ and an operator
\[
T' = A T B \in\Bop(\mathcal{M}_p)
\]
such that $\norm{T'-D}<\varepsilon$ for a diagonal operator $D=\bigoplus_{k=1}^\infty D_k$.
Moreover, $T'$ factors through $T$ and $\Id-T'$ factors through $\Id-T$ (via the same $A,B$).
\end{lemma}

\begin{proof}
Set $K=\norm{T}$ and choose a summable control sequence
\[
b_k=\frac{\varepsilon}{2^{k+3}(1+K)}\qquad(k\in\N).
\]
For later use set
\[
    d_k=\max\Bigl\{1,\sum_{j<k}\dim X_j\Bigr\}
       =\max\Bigl\{1,\sum_{j<k}j^2\Bigr\}.
\]
We build inductively:
\begin{itemize}
    \item decreasing infinite sets $I_1\supset I_2\supset \cdots$,
    \item a strictly increasing sequence $m_1<m_2<\cdots$ with $m_k\in I_k$ and $m_k\ge k$,
    \item maps $r_k\colon X_k\to X_{m_k}$ and $s_k\colon X_{m_k}\to X_k$ with $\norm{r_k}=\norm{s_k}=1$ and $s_k r_k=\Id_{X_k}$,
    \item block projections $q_k=r_k s_k$ on $X_{m_k}$ of order $k$,
\end{itemize}
such that
\begin{enumerate}
\item[(i)] if
\[
E_k := \pi_{m_k}T\Big(\sum_{j<k} \iota_{m_j} r_j(X_j)\Big)\subseteq X_{m_k},
\]
then $\norm{q_k x}\le b_k \norm{x}$ for all $x\in E_k$;
\item[(ii)] $\norm{\pi_{m_k} T P_{I_{k+1}}}\le b_k$.
\end{enumerate}

Start with $I_1=\N$. Assume $I_k$ and the data up to $k-1$ have been chosen.
Pick $m_k\in I_k$ so large that $m_k>\max\{m_{k-1},k\}$ and Lemma~\ref{lem:blockkill} applies inside $X_{m_k}$
to the $d_k$-dimensional subspace $E_k$ with tolerance $b_k$ and order $k$.
Thus there exist $\alpha_k,\tau_k\subseteq\{1,\dots,m_k\}$ with $|\alpha_k|=|\tau_k|=k$ such that
$q_k=Q_{\alpha_k,\tau_k}$ satisfies (i). Set $r_k=J_{\alpha_k,\tau_k}$ and $s_k=K_{\alpha_k,\tau_k}$.
Finally apply Lemma~\ref{lem:tail-small} to $S=\pi_{m_k}T$, the infinite set $I_k\setminus\{m_k\}$, and the tolerance $b_k$.  We obtain an infinite set
$I_{k+1}\subseteq I_k\setminus\{m_k\}$ such that
\[
    \|\pi_{m_k}T P_{I_{k+1}}\|\le b_k.
\]
This completes the induction.

Define $B\in\Bop(\mathcal{M}_p)$ by
\[
(Bx)_{m_k}=r_k(x_k)\quad (k\in\N),\qquad (Bx)_n=0\ \text{if }n\notin\{m_k:k\in\N\}.
\]
Then $\norm{B}=1$. Define $A\in\Bop(\mathcal{M}_p)$ by
\[
(Ay)_k = s_k q_k(\pi_{m_k}y)\quad (k\in\N).
\]
Then $\norm{A}\le 1$ and $AB=\Id_{\mathcal{M}_p}$ since $s_k q_k r_k=s_k r_k=\Id_{X_k}$.

Set $T'=ATB$ and define the diagonal operator $D=\bigoplus_{k=1}^\infty D_k$ by $D_k=\pi_k T' \iota_k$.
Fix $x\in\mathcal{M}_p$ with $\norm{x}\le 1$ and consider
\[
\pi_k(T'-D)x=\pi_kT'(x-\iota_k\pi_k x)=s_k q_k\,\pi_{m_k}T\,B(x-\iota_k\pi_k x).
\]
Write $u=B(x-\iota_k\pi_k x)=u^{<k}+u^{>k}$ where $u^{<k}$ is supported on $\{m_j:j<k\}$ and $u^{>k}$ on $\{m_j:j>k\}$.
Then $\norm{u^{<k}}\le 1$ and $\norm{u^{>k}}\le 1$ since $\norm{B}=1$.

\smallskip\noindent
\emph{Past part.} $y^{<k}:=\pi_{m_k}T(u^{<k})$ lies in $E_k$ by definition, hence
\[
\norm{q_k y^{<k}}\le b_k\norm{y^{<k}}\le b_k \norm{\pi_{m_k}T}\,\norm{u^{<k}}\le b_k K.
\]

\smallskip\noindent
\emph{Future part.} Since $\{m_j:j>k\}\subseteq I_{k+1}$, we have $u^{>k}=P_{I_{k+1}}u^{>k}$, hence by (ii)
\[
\norm{\pi_{m_k}T(u^{>k})}\le b_k\norm{u^{>k}}\le b_k.
\]

Since $s_k q_k$ is contractive, combining the two estimates gives
\[
\norm{\pi_k(T'-D)x}\le b_kK+b_k=b_k(1+K)=\frac{\varepsilon}{2^{k+3}}.
\]
Taking the supremum over $k$ yields $\norm{T'-D}\le \varepsilon/16<\varepsilon$.

Finally, $\Id-T' = AB-ATB = A(\Id-T)B$, so $\Id-T'$ factors through $\Id-T$.
\end{proof}

\begin{lemma}\label{lem:subsequence-copy}
Let $(n_j)_{j\ge 1}$ be strictly increasing with $n_j\ge j$.
Then the coordinate subspace of $\mathcal{M}_p$ supported on $\{n_j:j\ge 1\}$ contains a $1$-complemented copy of $\mathcal{M}_p$.
\end{lemma}

\begin{proof}
For each $j$ define $J_{j,n_j}=J_{\{1,\dots,j\},\{1,\dots,j\}}\colon X_j\to X_{n_j}$ and
$K_{j,n_j}=K_{\{1,\dots,j\},\{1,\dots,j\}}\colon X_{n_j}\to X_j$.
Then $K_{j,n_j}J_{j,n_j}=\Id_{X_j}$ and $\norm{J_{j,n_j}}=\norm{K_{j,n_j}}=1$.

Define $E\colon \mathcal{M}_p\to\mathcal{M}_p$ by $(Ex)_{n_j}=J_{j,n_j}(x_j)$ and $(Ex)_n=0$ if $n\notin\{n_j\}$.
Define $P\colon \mathcal{M}_p\to \mathcal{M}_p$ by $(Py)_j=K_{j,n_j}(y_{n_j})$.
Then $\norm{E}=\norm{P}=1$ and $PE=\Id_{\mathcal{M}_p}$.
\end{proof}

\begin{proposition}[Diagonal case]\label{prop:diagonalPFP}
    Let $D=\bigoplus_{n=1}^\infty D_n\in\Bop(\mathcal{M}_p)$ be diagonal.
    Then there exist $U,V\in\Bop(\mathcal{M}_p)$ such that either $UDV=\Id$ or $U(\Id-D)V=\Id$ and
    $\norm{U}\norm{V}\le 4$.
\end{proposition}

\begin{proof}
Let $K=\norm{D}$ and fix $\varepsilon=\tfrac18$.
Choose a strictly increasing sequence $m_k$ with $m_k\ge k$ such that Proposition~\ref{prop:finite} applies to $T=D_{m_k}\in\Bop(X_{m_k})$
with parameters $(n,K,\varepsilon)=(k,K,\tfrac18)$.
Thus there exist block maps $J_k\colon X_k\to X_{m_k}$, $K_k\colon X_{m_k}\to X_k$ and scalars $c_k\in\C$ with
\[
\norm{K_k D_{m_k} J_k - c_k \Id_{X_k}}<\tfrac18,
\]
so either $S_k:=K_kD_{m_k}J_k$ is invertible with $\norm{S_k^{-1}}\le 4$ or
$S_k':=K_k(\Id-D_{m_k})J_k$ is invertible with $\norm{(S_k')^{-1}}\le 4$.
Let $\sigma_k\in\{0,1\}$ record which case holds.

By infinite pigeonhole there exist an infinite set $L=\{k_1<k_2<\cdots\}$ and a fixed $\sigma\in\{0,1\}$ such that $\sigma_{k_j}=\sigma$ for all $j$.
Let
\[
    \mathcal N=P_L\mathcal M_p
\]
be the full coordinate subspace supported on $L$. By Lemma~\ref{lem:subsequence-copy}, there are contractions
$E_0\colon\mathcal M_p\to\mathcal N$ and $P_0\colon\mathcal N\to\mathcal M_p$ such that $P_0E_0=\Id_{\mathcal M_p}$.

Define $B\colon\mathcal N\to\mathcal M_p$ by
\[
    (By)_{m_{k_j}}=J_{k_j}(y_{k_j}),
    \qquad
    (By)_n=0\quad(n\notin\{m_{k_j}:j\in\N\}).
\]
Define $A\colon\mathcal M_p\to\mathcal N$ by
\[
(Az)_{k_j}=
\begin{cases}
S_{k_j}^{-1}K_{k_j}(z_{m_{k_j}}),&\sigma=0,\\
(S_{k_j}')^{-1}K_{k_j}(z_{m_{k_j}}),&\sigma=1,
\end{cases}
\qquad
(Az)_n=0\quad(n\notin L).
\]
Then $\|B\|\le1$ and $\|A\|\le4$.

If $\sigma=0$, then for $y\in\mathcal N$,
\[
    (ADB y)_{k_j}=S_{k_j}^{-1}K_{k_j}\bigl(D_{m_{k_j}}J_{k_j}(y_{k_j})\bigr)
    =y_{k_j},
\]
so $ADB=\Id_{\mathcal N}$. If $\sigma=1$, the same computation gives
$A(\Id-D)B=\Id_{\mathcal N}$.

Finally set $V=BE_0$ and $U=P_0A$. Then either $UDV=\Id_{\mathcal M_p}$ or
$U(\Id-D)V=\Id_{\mathcal M_p}$, and
\(
    \|U\|\|V\|\le \|P_0\|\|A\|\|B\|\|E_0\|\le4.
\)
\end{proof}

\begin{theorem}\label{thm:MpUPFP}
    For every $p\in(1,\infty)$ the space $\mathcal{M}_p$ has the $8$-primary factorisation property.
\end{theorem}

\begin{proof}
Let $T\in\Bop(\mathcal{M}_p)$.
Apply Lemma~\ref{lem:almost-diagonal} with $\varepsilon=\tfrac{1}{16}$ to obtain
contractions $A_0,B_0$ with $A_0B_0=\Id$ and $T'=A_0TB_0$ satisfying $\norm{T'-D}<\tfrac{1}{16}$
for a diagonal operator $D$.

By Proposition~\ref{prop:diagonalPFP}, there exist $U_0,V_0$ with $\|U_0\|\|V_0\|\le4$ such that either
\[
    U_0DV_0=\Id
    \qquad\text{or}\qquad
    U_0(\Id-D)V_0=\Id.
\]

In the first case set $E=U_0(T'-D)V_0$. Then $\|E\|\le1/4$, so $\Id+E$ is invertible with inverse norm at most $4/3<2$. With
$U_1=(\Id+E)^{-1}U_0$ and $V_1=V_0$, we obtain
\[
    U_1T'V_1=\Id,
    \qquad
    \|U_1\|\|V_1\|\le8.
\]
Since $T'=A_0TB_0$, we have
\[
    (U_1A_0)\,T\,(B_0V_1)=\Id,
\]
so this is a factorisation through $T$ with the same bound.

In the second case set
\[
    E=U_0\bigl((\Id-T')-(\Id-D)\bigr)V_0=-U_0(T'-D)V_0.
\]
Again $\|E\|\le1/4$, and the same Neumann-series argument gives operators $U_1,V_1$ with
\[
    U_1(\Id-T')V_1=\Id,
    \qquad
    \|U_1\|\|V_1\|\le8.
\]
Since $\Id-T'=A_0(\Id-T)B_0$, we have
\[
    (U_1A_0)\,(\Id-T)\,(B_0V_1)=\Id,
\]
so this is a factorisation through $\Id-T$ with the same bound.
\end{proof}

\begin{theorem}\label{thm:BlpUPFP}
    For every $p\in(1,\infty)$ the Banach space $\Bop(\ell_p)$ has the uniform primary factorisation property.
\end{theorem}

\begin{proof}
    By \cite[Section~2]{AriasFarmer}, $\Bop(\ell_p)$ is linearly isomorphic to $\mathcal{M}_p$.
    Let $S\colon \Bop(\ell_p)\to \mathcal{M}_p$ be an isomorphism.
    By Theorem~\ref{thm:MpUPFP}, $\mathcal{M}_p$ has the $8$-PFP.
    Proposition~\ref{prop:iso-invariance} transfers this to $\Bop(\ell_p)$ with constant
    $8\,\norm{S}^2\norm{S^{-1}}^2<\infty$.
    Hence $\Bop(\ell_p)$ has UPFP.
\end{proof}

%% file: 5_Primary_filters.tex
\section{Filter-convergence spaces}\label{sec:primary-filters}

In this section, we give a filter-theoretic criterion for the primary factorisation property.  The abstract criterion is deliberately separated from a more concrete, set-theoretic sufficient condition. The latter is often the condition one checks in examples; it implies the required operator diagonalisation by Rosenthal's lemma.

Let \(\mathcal F\) be a proper filter on an infinite set \(\Gamma\), and put
\[
    X_{\mathcal F}=c_{\mathcal F}(\Gamma)
    =\{x\in \ell_\infty(\Gamma): x\to 0\text{ along }\mathcal F\}.
\]
Thus
\(
    x\in X_{\mathcal F}
    \quad\Longleftrightarrow\quad
    \{\gamma\in\Gamma: |x(\gamma)|<\varepsilon\}\in\mathcal F
    \quad(\varepsilon>0).
\)
For \(A\subseteq\Gamma\), let \(P_A\) be the coordinate projection and write
\(
    X_{\mathcal F}(A)=P_A X_{\mathcal F}
    =\{x\in X_{\mathcal F}: \operatorname{supp}x\subseteq A\}.
\)
If \(a\in\ell_\infty(\Gamma)\), we denote by \(M_a\) the corresponding multiplication operator on \(X_{\mathcal F}\).

For \(A\subseteq\Gamma\), define the trace filter
\[
    \mathcal F_A
    =\{B\subseteq A: B\cup(\Gamma\setminus A)\in\mathcal F\}.
\]
With the convention that \(c_{\mathcal F_A}(A)=\ell_\infty(A)\) if \(\mathcal F_A\) is improper, restriction to \(A\) identifies \(X_{\mathcal F}(A)\) isometrically with \(c_{\mathcal F_A}(A)\).
Consequently, whenever \(\Gamma=A\sqcup B\),
\(
    X_{\mathcal F}=X_{\mathcal F}(A)\oplus_\infty X_{\mathcal F}(B).
\)

\subsection{The abstract criterion}

\begin{definition}\label{def:copying-set}
Let \(K\geq 1\).  A subset \(A\subseteq\Gamma\) is called a
\emph{\(K\)-copying set} for \(\mathcal F\) if
\[
    X_{\mathcal F}(A)\simeq_K X_{\mathcal F}.
\]
\end{definition}

\begin{definition}\label{def:sum-primary-filter}
Let \(K\geq 1\).  We say that \(\mathcal F\) is
\emph{\(K\)-primary under sums} if, for every partition
\(\Gamma=\Gamma_0\sqcup\Gamma_1\), at least one of \(\Gamma_0\) and \(\Gamma_1\) contains a \(K\)-copying set for \(\mathcal F\).
\end{definition}

This is the filter analogue of a projection-level primariness condition.  It only sees coordinate decompositions of \(X_{\mathcal F}\).  To obtain a factorisation dichotomy for arbitrary operators, we use the following additional operator-level hypothesis.

\begin{definition}\label{def:diagonalisable-filter}
Let \(D\geq 1\).  We say that \(\mathcal F\) is
\emph{\(D\)-diagonalisable} if, for every
\(T\in\mathcal B(X_{\mathcal F})\) and every \(\eta>0\), there are
\(A\subseteq\Gamma\), an isomorphism
\[
    J\colon X_{\mathcal F}\longrightarrow X_{\mathcal F}(A),
    \qquad
    R=J^{-1}\colon X_{\mathcal F}(A)\longrightarrow X_{\mathcal F},
\]
with
\(
    \|J\|\,\|R\|\leq D,
\)
and a function \(a\in\ell_\infty(\Gamma)\) such that
\[
    \|R P_A T J-M_a\|<\eta .
\]
\end{definition}

\begin{definition}\label{def:factorisation-primary-filter}
Let \(K,D\geq 1\).  We say that \(\mathcal F\) is
\emph{\((K,D)\)-primary for factorisation} if it is \(K\)-primary under sums and \(D\)-diagonalisable.
\end{definition}

\begin{theorem}\label{thm:primary-filter-pfp}
Let \(\mathcal F\) be a \((K,D)\)-primary filter for factorisation on \(\Gamma\). Then
\[
    X_{\mathcal F}=c_{\mathcal F}(\Gamma)
\]
has the \((2KD)^+\)-primary factorisation property.  In particular, \(X_{\mathcal F}\) has the UPFP.
\end{theorem}

\begin{proof}
Let \(X=X_{\mathcal F}\), and let \(T\in\mathcal B(X)\).  Fix \(\varepsilon>0\).  Choose \(\eta>0\) so small that
\[
    2K\eta<1
    \qquad\text{and}\qquad
    \frac{2KD}{1-2K\eta}<2KD+\varepsilon .
\]
By \(D\)-diagonalisability, there exist \(A\subseteq\Gamma\), an isomorphism
\[
    J\colon X\longrightarrow X_{\mathcal F}(A),
    \qquad
    R=J^{-1}\colon X_{\mathcal F}(A)\longrightarrow X,
\]
with \(\|J\|\|R\|\leq D\), and \(a\in\ell_\infty(\Gamma)\), such that
\[
    S:=R P_A T J
    \quad\text{satisfies}\quad
    \|S-M_a\|<\eta .
\]
Partition \(\Gamma\) as
\[
    \Gamma_0=\{\gamma\in\Gamma: |a(\gamma)|\geq 1/2\},
    \qquad
    \Gamma_1=\Gamma\setminus\Gamma_0.
\]
Then \(|1-a(\gamma)|\geq 1/2\) for every \(\gamma\in\Gamma_1\).  Since
\(\mathcal F\) is \(K\)-primary under sums, one of \(\Gamma_0\) and
\(\Gamma_1\) contains a \(K\)-copying set.  Choose such a set \(C\), and choose an isomorphism
\[
    L\colon X\longrightarrow X_{\mathcal F}(C),
    \qquad
    Q=L^{-1}\colon X_{\mathcal F}(C)\longrightarrow X,
\]
with \(\|L\|\|Q\|\leq K\).

Assume first that \(C\subseteq\Gamma_0\). Consider
\[
    B=P_C S L=P_C R P_A T J L
    \colon X\longrightarrow X_{\mathcal F}(C).
\]
The operator
\(
    B_0=P_C M_a L\colon X\longrightarrow X_{\mathcal F}(C)
\)
is an isomorphism. Indeed, multiplication by \(a\) is invertible on
\(X_{\mathcal F}(C)\), with inverse multiplication by \(1/a\) on \(C\), because \(|a|\geq 1/2\) on \(C\). Hence
\(
    \|B_0^{-1}\|\leq 2\|Q\|.
\)
Moreover,
\(
    \|B-B_0\|
    \leq \|S-M_a\|\,\|L\|
    <\eta\|L\|.
\)
Therefore
\[
    \|B_0^{-1}\|\,\|B-B_0\|
    <2\eta\|Q\|\|L\|
    \leq 2K\eta<1.
\]
By the Neumann lemma, \(B\) is an isomorphism from \(X\) onto \(X_{\mathcal F}(C)\), and
\[
    \|B^{-1}\|
    \leq \frac{2\|Q\|}{1-2K\eta}.
\]
Putting
\[
    U=B^{-1}P_C R P_A,
    \qquad
    V=J L,
\]
we get
\(
    U T V=B^{-1}P_C R P_A T J L=B^{-1}B=\mathrm{Id}_X,
\)
and
\[
    \|U\|\|V\|
    \leq
    \frac{2\|Q\|\|R\|\|J\|\|L\|}{1-2K\eta}
    \leq
    \frac{2KD}{1-2K\eta}
    <2KD+\varepsilon .
\]
Thus \(\mathrm{Id}_X\) factors through \(T\) with constant less than
\(2KD+\varepsilon\).

If instead \(C\subseteq\Gamma_1\), then
\[
    R P_A(\mathrm{Id}_X-T)J=\mathrm{Id}_X-S
\]
and
\[
    \| (\mathrm{Id}_X-S)-M_{1-a}\|<\eta .
\]
Since \(|1-a|\geq 1/2\) on \(C\), the preceding argument applied to \(\mathrm{Id}_X-T\) and to the multiplier \(M_{1-a}\) shows that \(\mathrm{Id}_X\) factors through \(\mathrm{Id}_X-T\) with constant less than \(2KD+\varepsilon\).

Since \(T\) and \(\varepsilon\) were arbitrary, \(X\) has the \((2KD)^+\)-PFP.
\end{proof}

\subsection{A set-theoretic sufficient condition}

The preceding definition of diagonalisability is useful as an abstract criterion, but it is not a condition one can normally check directly. We now give a more concrete sufficient condition depending only on the existence of many coordinate copies of the filter.

\begin{definition}\label{def:filter-copy}
Let \(B\subseteq\Gamma\). We say that \(B\) contains a
\emph{coordinate copy} of \(\mathcal F\) if there are \(C\subseteq B\) and a bijection \(\theta\colon\Gamma\to C\) such that, for every \(A\subseteq\Gamma\),
\[
    A\in\mathcal F
    \quad\Longleftrightarrow\quad
    \theta(A)\cup(\Gamma\setminus C)\in\mathcal F .
\]
Equivalently, the trace filter \(\mathcal F_C\) is isomorphic to \(\mathcal F\).
\end{definition}

If \(C\) and \(\theta\) are as above, the formula
\[
    (U_\theta x)(\theta(\gamma))=x(\gamma)
    \quad(\gamma\in\Gamma),
    \qquad
    (U_\theta x)(\alpha)=0
    \quad(\alpha\notin C),
\]
defines a surjective linear isometry
\(
    U_\theta\colon X_{\mathcal F}\longrightarrow X_{\mathcal F}(C).
\)

\begin{definition}\label{def:large-set-self-similar-filter}
Let \(\kappa=|\Gamma|\).  We say that \(\mathcal F\) is
\emph{large-set self-similar} if every subset \(B\subseteq\Gamma\) with \(|B|=\kappa\) contains a coordinate copy of \(\mathcal F\).
\end{definition}

Thus, large-set self-similarity is a purely set-theoretic strengthening of primariness under sums: in a partition of \(\Gamma\) into two pieces, one piece must have cardinality \(\kappa\), and therefore contains a copy of the original filter.

We shall use the following standard form of Rosenthal's lemma from \cite{Rosenthal1970}.

\begin{lemma}[Rosenthal's lemma]\label{lem:rosenthal}
Let \(\Gamma\) be an infinite set, and let
\((\mu_\gamma)_{\gamma\in\Gamma}\) be a uniformly bounded family of finitely additive scalar measures on \(\Gamma\).  Suppose that
\(\mu_\gamma(\{\gamma\})=0\) for every \(\gamma\in\Gamma\).  Then, for every \(\eta>0\), there is \(A\subseteq\Gamma\) with \(|A|=|\Gamma|\) such that
\[
    |\mu_\gamma|(A)<\eta
    \qquad(\gamma\in A).
\]
\end{lemma}

\begin{proposition}\label{prop:large-set-self-similar-implies-diagonal-primary}
If \(\mathcal F\) is large-set self-similar, then \(\mathcal F\) is
\((1,1)\)-primary for factorisation. Consequently,
\(X_{\mathcal F}\) has the \(2^+\)-primary factorisation property.
\end{proposition}

\begin{proof}
First let \(\Gamma=\Gamma_0\sqcup\Gamma_1\).  Since \(\Gamma\) is infinite, one of \(\Gamma_0\) and \(\Gamma_1\) has cardinality \(|\Gamma|\).  By
large-set self-similarity, that piece contains a coordinate copy of \(\mathcal F\). Hence \(\mathcal F\) is \(1\)-primary under sums.

It remains to prove \(1\)-diagonalisability. Let
\(T\in\mathcal B(X_{\mathcal F})\) and let \(\eta>0\). For each
\(\alpha\in\Gamma\), the functional
\[
    x\mapsto (Tx)(\alpha)
    \qquad(x\in X_{\mathcal F})
\]
has norm at most \(\|T\|\).  By Hahn--Banach, choose a finitely additive measure \(\lambda_\alpha\in \operatorname{ba}(\Gamma)\) with \(\|\lambda_\alpha\|\leq \|T\|\) such that
\[
    (Tx)(\alpha)=\int_\Gamma x\,d\lambda_\alpha
    \qquad(x\in X_{\mathcal F}).
\]
Put
\[
    \mu_\alpha
    =\lambda_\alpha-
      \lambda_\alpha(\{\alpha\})\delta_\alpha
    \qquad(\alpha\in\Gamma).
\]
Then \((\mu_\alpha)_{\alpha\in\Gamma}\) is uniformly bounded and
\(\mu_\alpha(\{\alpha\})=0\) for every \(\alpha\).  By
Lemma~\ref{lem:rosenthal}, there is \(A\subseteq\Gamma\) with
\(|A|=|\Gamma|\) such that
\[
    |\mu_\alpha|(A)<\eta
    \qquad(\alpha\in A).
\]
By large-set self-similarity, choose a coordinate copy
\(C\subseteq A\) of \(\mathcal F\), and let \(\theta\colon\Gamma\to C\) witness this.  Put
\[
    J=U_\theta\colon X_{\mathcal F}\to X_{\mathcal F}(C),
    \qquad
    R=U_\theta^{-1}\colon X_{\mathcal F}(C)\to X_{\mathcal F}.
\]
Then \(\|J\|\|R\|=1\).  Define \(a\in\ell_\infty(\Gamma)\) by
\[
    a(\gamma)=\lambda_{\theta(\gamma)}(\{\theta(\gamma)\})
    \qquad(\gamma\in\Gamma).
\]
For \(x\in X_{\mathcal F}\) with \(\|x\|\leq 1\) and
\(\gamma\in\Gamma\), we have
\[
\begin{aligned}
    \bigl((R P_C T J-M_a)x\bigr)(\gamma)
    &= (T Jx)(\theta(\gamma))-a(\gamma)x(\gamma) \\
    &= \int_\Gamma Jx\,d\lambda_{\theta(\gamma)}
       -\lambda_{\theta(\gamma)}(\{\theta(\gamma)\})Jx(\theta(\gamma)) \\
    &= \int_C Jx\,d\mu_{\theta(\gamma)} .
\end{aligned}
\]
Since \(C\subseteq A\), it follows that
\(
    \left|\bigl((R P_C T J-M_a)x\bigr)(\gamma)\right|
    \leq |\mu_{\theta(\gamma)}|(C)
    \leq |\mu_{\theta(\gamma)}|(A)<\eta .
\)
Taking the supremum over \(\gamma\in\Gamma\) gives
\(
    \|R P_C T J-M_a\|<\eta .
\)
Thus \(\mathcal F\) is \(1\)-diagonalisable.  The final assertion follows from Theorem~\ref{thm:primary-filter-pfp} with \(K=D=1\).
\end{proof}

\begin{example}\label{ex:co-small-filters}
Let \(\kappa=|\Gamma|\), and let \(\mu\leq\kappa\) be an infinite cardinal. Define
\[
    \mathcal F_\mu
    =\{A\subseteq\Gamma: |\Gamma\setminus A|<\mu\}.
\]
Then \(\mathcal F_\mu\) is large-set self-similar.  Indeed, if \(B\subseteq\Gamma\) has cardinality \(\kappa\), then the trace filter on \(B\) is
\[
    (\mathcal F_\mu)_B
    =\{C\subseteq B: |B\setminus C|<\mu\},
\]
which is isomorphic to \(\mathcal F_\mu\) through any bijection
\(\Gamma\to B\).  Hence
\(
    c_{\mathcal F_\mu}(\Gamma)
\)
has the \(2^+\)-PFP.

For \(\mu=\omega\), this recovers the usual cofinite-filter case
\(c_{\mathcal F_\omega}(\Gamma)=c_0(\Gamma)\).  For
\(\mu=\omega_1\), it gives the co-countable-filter case
\(c_{\mathcal F_{\omega_1}}(\Gamma)=\ell_\infty^c(\Gamma)\), whenever
\(|\Gamma|\geq\omega_1\), which recovers \cite[Theorem 1.4]{JKS2016}.
\end{example}

\begin{remark}\label{rem:why-def-diag}
Definition~\ref{def:diagonalisable-filter} should therefore be read as an abstract operator criterion, not as the primary source of examples. The set-theoretic condition in Definition~\ref{def:large-set-self-similar-filter} is stronger but much easier to verify: Rosenthal's lemma turns large-set self-similarity into diagonalisation, and large set self-similarity also gives the required primariness under sums. Thus one obtains the concrete implication
\[
    \text{large-set self-similar filter}
    \quad\Longrightarrow\quad
    c_{\mathcal F}(\Gamma)\text{ has the }2^+\text{-PFP}.
\]
\end{remark}

%% file: w_Open_Questions.tex
\section{Open Problems}\label{sec:open-problems}

The results of this paper suggest a number of natural questions.
Some concern the relation between primariness and factorisation, while others arise from the limitations of the present methods in mixed and non-separable settings.

\runinhead{Primariness versus factorisation} A natural first question is whether the primary factorisation property is genuinely stronger than primariness. The following example shows that, in some sense, this is indeed the case. The authors are grateful to Niels Jakob Laustsen (Lancaster) for bringing this example to their attention.

\begin{example}[Primariness does not imply PFP]
    Let $X_{\mathrm{GM}}$ denote the prime Gowers--Maurey space, see \cite{GowersMaureySmallOperators}. Then $X_{\mathrm{GM}}$ is prime, hence primary. On the other hand, for this space the quotient of $\Bop(X_{\mathrm{GM}})$ by the ideal of strictly singular operators is isomorphic, as a Banach algebra, to $\ell_1(\mathbb Z)$ equipped with the convolution product. Let
    \begin{equation*}
        q\colon \Bop(X_{\mathrm{GM}})\to \Bop(X_{\mathrm{GM}})/\mathcal{SS}(X_{\mathrm{GM}})\cong \ell_1(\mathbb Z)
    \end{equation*}
    denote the quotient map. For every maximal ideal $M$ of $\ell_1(\mathbb Z)$, the preimage $q^{-1}(M)$ is a maximal ideal of $\Bop(X_{\mathrm{GM}})$; in fact, these are all the maximal ideals of $\Bop(X_{\mathrm{GM}})$, since every maximal ideal of $\Bop(X_{\mathrm{GM}})$ contains the ideal of strictly singular operators. Since the maximal ideal space of $\ell_1(\mathbb Z)$ is naturally identified with the unit circle $\mathbb T$, the algebra $\ell_1(\mathbb Z)$ has continuum many maximal ideals. Hence $\Bop(X_{\mathrm{GM}})$ has continuum many maximal ideals. By Proposition~\ref{prop:maximal-ideal-pfp}, it follows that $X_{\mathrm{GM}}$ cannot have the PFP. Thus, primariness does not imply the PFP.
\end{example}

By Pe{\l}czy\'nski's decomposition method, the PFP implies primariness on spaces that are stable under the relevant direct-sum operation, but outside that framework, it is unclear if PFP implies primarity.

\begin{question}
Does every Banach space with the PFP have to be primary?
\end{question}

A classical special case is still open.

\begin{question}
Is every Banach space with a symmetric basis primary?
\end{question}

The results of Casazza, Kottman and Lin show that a large class of such spaces satisfy a strong range dichotomy, but this still falls short of a complete answer. We note, however, that Lin \cite{Lin1992} resolved the problem in the reflexive case.

\runinhead{PFP versus UPFP} In all examples presently known to us, the PFP is obtained together with quantitative control,
so that one actually gets the UPFP.
It is therefore natural to ask whether the uniform version is strictly stronger.

\begin{question}
Does there exist a Banach space with the PFP but without the UPFP?
\end{question}

A positive answer would show that the uniform quantitative content of the classical arguments is a genuine extra feature rather than an artefact of the proofs.

\runinhead{Duality} Many classical primary spaces arise in dual pairs, but the known arguments are often highly asymmetric.
The dual problem, therefore, remains largely open.

\begin{question}
Let $X$ be a primary Banach space. Must $X^*$ be primary?
What if $X$ has the PFP or the UPFP?
Do either of these properties pass to the dual space?
\end{question}

Even formulating a satisfactory converse is delicate, since non-isomorphic spaces may have isomorphic duals.

\runinhead{Countable versus uncountable sums} One of the main phenomena of this paper is that uncountability makes some problems easier rather than harder:
free-selection arguments and support-reduction lemmas allow one to force exact coordinate separation.
This leads to a natural meta-question.

\begin{question}
To what extent can the arguments for uncountable sums be pushed back to the countable setting?
For example, can the uncountable methods be adapted to recover the PFP or UPFP for difficult countable mixed sums, spaces such as $C(\alpha,\ell_p)$ or related bi-parameter spaces?
\end{question}

\runinhead{Quantitative decomposition} Several of the results proved here are genuinely quantitative, and Appendix~\ref{app:quant-pelc} records a quantitative form of Pe{\l}czy\'nski's decomposition method tailored to the present paper. We do not claim originality for that appendix: it is included only to make the constant bookkeeping explicit in one convenient place. Closely related explicit bounds in the normalised square case were obtained by Korpalski and Plebanek \cite[Theorem~5.1]{KorpalskiPlebanek2025} and improved by the Lewicki and the second-named author \cite[Theorem~B; Theorem~5.1]{KaniaLewicki2026}; see also Galego's qualitative extensions of the decomposition method \cite{Galego2007,Galego2008}.

\begin{question}
Can the constants in Appendix~\ref{app:quant-pelc} be improved substantially, especially in the square case or in the $c_0$/$\ell_p$-stable case?
\end{question}

%% file: Appendix_Quantitative_Pelczynski.tex
\section{A quantitative form of Pe{\l}czy\'nski's decomposition method}\label{app:quant-pelc}

We record here, for convenience, a quantitative version of Pe{\l}czy\'nski's decomposition method in a form suited to the present paper.
We do \emph{not} claim originality for the result below: it should be viewed as a constant-tracking variant of the classical argument of Pe{\l}czy\'nski \cite{Pelczynski1960}.
The bookkeeping involved seems to be part of the folklore, but it is not especially easy to locate in exactly the form needed here.

The closest explicit quantitative results we are aware of are due to Korpalski and Plebanek \cite[Theorem~5.1]{KorpalskiPlebanek2025} and, with an improved constant, the second-named author and Lewicki \cite[Theorem~B; Theorem~5.1]{KaniaLewicki2026}.
In the normalised setting where each space is isometric to a $1$-complemented subspace of the other and both spaces are isometrically square, Korpalski and Plebanek obtain the estimate $d_{BM}(X,Y)\le (3+\sqrt2)^2$, while the second-named author and Lewicki improve this to $d_{BM}(X,Y)\le 9+6\sqrt3$.
For qualitative extensions of Pe{\l}czy\'nski's method in the direction of supplemented subspaces and finite-sum arithmetic, see Galego \cite{Galego2007,Galego2008}.

The formulation below is slightly more flexible than the normalised square case treated in those references: we allow arbitrary complementability and isomorphism constants, and we also include the familiar $c_0$- and $\ell_p$-stable alternative.

Before turning to the statement and proof, we record a few elementary facts that will be used repeatedly.
If $A \simeq_C B$ and $B \simeq_K D$, then $A \simeq_{CK} D$.
Moreover, for every $1 \leq p \leq \infty$ and every Banach space $Z$, one has
\(
    A \oplus_p Z \simeq_C B \oplus_p Z.
\)
Also, for every $1 \leq p \leq \infty$,
\(
    \ell_p(A) \simeq_C \ell_p(B),
\)
where the isomorphism is obtained by applying a fixed isomorphism $A \to B$ coordinatewise.
Likewise,
\(
    c_0(A) \simeq_C c_0(B).
\)

\begin{theorem}[A quantitative form of Pe{\l}czy\'nski's decomposition method]\label{thm:quant-pelc}
Let $X$ and $Y$ be Banach spaces such that
\begin{enumerate}[label=\textup{(\roman*)}]
    \item $X$ has a $C_1$-complemented subspace $E$ which is $C_2$-isomorphic to $Y$;
    \item $Y$ has a $C_3$-complemented subspace $G$ which is $C_4$-isomorphic to $X$.
\end{enumerate}
Suppose further that either
\begin{enumerate}[label=(\alph*), ref=(\alph*)]
    \item \label{item:quant-pelc-a}
    there exists $1 \le r \le \infty$ such that
    \[
        X \simeq_{C_5} X \oplus_r X
        \qquad\text{and}\qquad
        Y \simeq_{C_6} Y \oplus_r Y,
    \]
    or
    \item \label{item:quant-pelc-b}
    \[
        X \simeq_{C_5} c_0(X)
        \qquad\text{or}\qquad
        X \simeq_{C_5} \ell_p(X)
    \]
    for some $1 \le p \le \infty$.
\end{enumerate}
Then there exists a constant $C \ge 1$, depending only on $C_1,\dots,C_6$ in case \ref{item:quant-pelc-a}, and only on $C_1,\dots,C_5$ in case \ref{item:quant-pelc-b}, such that $X \simeq_C Y$.
Moreover, the proof below gives an explicit admissible choice of $C$.
\end{theorem}

\begin{proof}
Choose a projection $P \colon X \to E$ with $\norm{P} \le C_1$, a projection $Q \colon Y \to G$ with $\norm{Q} \le C_3$, an isomorphism $S \colon Y \to E$ with $\norm{S}\norm{S^{-1}} \le C_2$, and an isomorphism $T \colon X \to G$ with $\norm{T}\norm{T^{-1}} \le C_4$.
After rescaling $S$ and $T$, we may assume that
\[
    \norm{S}=\norm{S^{-1}} \le \sqrt{C_2},
    \qquad
    \norm{T}=\norm{T^{-1}} \le \sqrt{C_4}.
\]
Set
\[
    F = \ker P,
    \qquad
    H = \ker Q.
\]

We first show that, for each fixed $1 \le r \le \infty$, there exist constants $a=a(C_1,C_2)$ and $b=b(C_3,C_4)$ such that
\(
    X \simeq_a Y \oplus_r F,
\)
and
\(
    Y \simeq_b X \oplus_r H.
\)
To this end, define
\[
    \Phi_X \colon X \to Y \oplus_r F,
    \qquad
    \Phi_X x = \bigl(S^{-1}Px,(I_X-P)x\bigr),
\]
and
\[
    \Psi_X \colon Y \oplus_r F \to X,
    \qquad
    \Psi_X(y,f)=Sy+f.
\]
Since $f \in F=\ker P$, we have $Pf=0$, and hence
\[
    \Psi_X\Phi_X x = S(S^{-1}Px)+(I_X-P)x = Px+(I_X-P)x = x.
\]
Also, if $y \in Y$ and $f \in F$, then
\[
    \Phi_X\Psi_X(y,f)=\bigl(S^{-1}P(Sy+f),(I_X-P)(Sy+f)\bigr)=(y,f),
\]
because $P(Sy)=Sy$ and $Pf=0$.
Thus $\Phi_X$ and $\Psi_X$ are inverse isomorphisms. Moreover,
\[
    \norm{\Phi_X x}_r \le \norm{S^{-1}Px}+\norm{(I_X-P)x}
    \le \bigl(\sqrt{C_2}\,C_1+1+C_1\bigr)\norm{x}.
\]
Likewise,
\[
    \norm{\Psi_X(y,f)} \le \norm{Sy}+\norm{f}
    \le (\sqrt{C_2}+1)(\norm{y}+\norm{f})
    \le 2(\sqrt{C_2}+1)\norm{(y,f)}_r.
\]
Hence
\(
    X \simeq_a Y \oplus_r F,
\)
where
\(
    a := 2(\sqrt{C_2}+1)\bigl(\sqrt{C_2}\,C_1+1+C_1\bigr).
\)

Similarly, defining the operators 
\(
    \Phi_Y \colon Y \to X \oplus_r H, 
    \Phi_Y y = \bigl(T^{-1}Qy,(I_Y-Q)y\bigr),
\)
and
\(
    \Psi_Y \colon X \oplus_r H \to Y, 
    \Psi_Y(x,h)=Tx+h,
\)
one obtains
\(
    Y \simeq_b X \oplus_r H,
\)
where
\(
    b := 2(\sqrt{C_4}+1)\bigl(\sqrt{C_4}\,C_3+1+C_3\bigr).
\)

We now treat the two cases separately.

\smallskip
\noindent
\textit{Case \ref{item:quant-pelc-a}.}
Fix $r$ as in the hypothesis.
Since $X \simeq_a Y \oplus_r F$, we obtain
\[
    X \simeq_a Y \oplus_r F
    \simeq_{C_6} (Y \oplus_r Y)\oplus_r F
    \simeq_1 Y \oplus_r (Y \oplus_r F)
    \simeq_a Y \oplus_r X,
\]
and therefore
\(
    X \simeq_{a^2C_6} Y \oplus_r X.
\)
Likewise,
\[
    Y \simeq_b X \oplus_r H
    \simeq_{C_5} (X \oplus_r X)\oplus_r H
    \simeq_1 X \oplus_r (X \oplus_r H)
    \simeq_b X \oplus_r Y,
\]
so
\(
    Y \simeq_{b^2C_5} X \oplus_r Y.
\)
Combining the last two displays, and using the canonical isometry $Y \oplus_r X \simeq_1 X \oplus_r Y$, we get
\[
    X \simeq_{a^2C_6} Y \oplus_r X
    \simeq_1 X \oplus_r Y
    \simeq_{b^2C_5} Y.
\]
Thus
\(
    X \simeq_{a^2b^2C_5C_6} Y.
\)

\smallskip
\noindent
\textit{Case \ref{item:quant-pelc-b}.}
We treat the case
\[
    X \simeq_{C_5} \ell_p(X),
    \qquad 1 \le p \le \infty;
\]
the case $X \simeq_{C_5} c_0(X)$ is identical, replacing $\ell_p$ by $c_0$ and $\oplus_p$ by $\oplus_\infty$ throughout.

Since $\ell_p(X)$ is canonically isometric to $X \oplus_p \ell_p(X)$, we have
\[
    X \simeq_{C_5} \ell_p(X)
    \simeq_1 X \oplus_p \ell_p(X)
    \simeq_{C_5} X \oplus_p X,
\]
and so
\(
    X \simeq_{C_5^2} X \oplus_p X.
\)

Now take $r=p$ in the first part of the proof.
Then
\[
    X \simeq_a Y \oplus_p F,
    \qquad
    Y \simeq_b X \oplus_p H.
\]
We next show that
\(
    Y \simeq_{(abC_5)^2} Y \oplus_p Y.
\)
Starting from $Y \simeq_b X \oplus_p H$, we obtain
\[
    Y \simeq_b X \oplus_p H
    \simeq_{C_5} \ell_p(X)\oplus_p H
    \simeq_a \ell_p(Y \oplus_p F)\oplus_p H
    \simeq_1 \ell_p(Y)\oplus_p \ell_p(F)\oplus_p H.
\]
Since $\ell_p(Y)$ is canonically isometric to $Y \oplus_p \ell_p(Y)$, it follows that
\[
    Y \simeq_{abC_5} Y \oplus_p \bigl(\ell_p(Y)\oplus_p \ell_p(F)\oplus_p H\bigr).
\]
On the other hand,
\[
    \ell_p(Y)\oplus_p \ell_p(F)\oplus_p H
    \simeq_1 \ell_p(Y \oplus_p F)\oplus_p H
    \simeq_a \ell_p(X)\oplus_p H
    \simeq_{C_5} X \oplus_p H
    \simeq_b Y.
\]
Therefore
\(
    Y \simeq_{(abC_5)^2} Y \oplus_p Y.
\)

We now repeat the same decomposition argument with $\oplus_p$ in place of $\oplus_r$.
Using $X \simeq_a Y \oplus_p F$, we get
\[
    X \simeq_a Y \oplus_p F
    \simeq_{(abC_5)^2} (Y \oplus_p Y)\oplus_p F
    \simeq_1 Y \oplus_p (Y \oplus_p F)
    \simeq_a Y \oplus_p X,
\]
whence
\(
    X \simeq_{a^2(abC_5)^2} Y \oplus_p X.
\)
Similarly, since $X \simeq_{C_5^2} X \oplus_p X$, we have
\[
    Y \simeq_b X \oplus_p H
    \simeq_{C_5^2} (X \oplus_p X)\oplus_p H
    \simeq_1 X \oplus_p (X \oplus_p H)
    \simeq_b X \oplus_p Y,
\]
and therefore
\(
    Y \simeq_{b^2C_5^2} X \oplus_p Y.
\)
Combining the last two displays, and using the canonical isometry $Y \oplus_p X \simeq_1 X \oplus_p Y$, we obtain
\(
    X \simeq_{a^2(abC_5)^2} Y \oplus_p X
    \simeq_1 X \oplus_p Y
    \simeq_{b^2C_5^2} Y.
\)
Thus $X \simeq_{a^4b^4C_5^4} Y$.
This completes the proof.
\end{proof}

\begin{remark}\label{rem:quant-pelc-norm}
In case \ref{item:quant-pelc-a}, the choice of the norm on the two-fold direct sums is only a normalisation.
One may equally work with the max-norm, as in \cite{KorpalskiPlebanek2025,KaniaLewicki2026}, or with any other standard norm on a finite direct sum; the resulting formulation is equivalent, with constants changing by a factor depending only on that choice.
\end{remark}

%% file: Appendix_Jp_Dual_Primary.tex
\section{Duals of the $p$th James spaces}\label{app:Jp-dual-primary}

We record a short permanence argument for the duals of the $p$th James spaces.
Throughout this appendix $1<p<\infty$, and $q$ denotes the conjugate exponent, so that $1/p+1/q=1$.
We use the standard facts that $J_p$ is quasi-reflexive of order one and that every closed infinite-dimensional subspace of $J_p$ contains a complemented subspace isomorphic to $\ell_p$; see, for example, \cite[Remark~2.2, Proposition~2.3, and Section~2.6]{BirdLaustsen2010}.
For $p=2$, the primarity input below is Casazza's theorem \cite{Casazza1977}; for general $p$ we state the result conditionally on primarity of $J_p$.

\begin{lemma}\label{lem:complemented-copy-implies-hyperplanes}
Let $X$ be an infinite-dimensional Banach space. Suppose that $X$ contains a complemented subspace $E$ such that
\[
    E \simeq E\oplus \mathbb K .
\]
Then $X\simeq X\oplus \mathbb K$. In particular, $X$ is isomorphic to each of its hyperplanes.
\end{lemma}

\begin{proof}
Write $X=E\oplus F$ for some closed subspace $F$ of $X$. Then
\[
    X\oplus \mathbb K
        \simeq E\oplus F\oplus \mathbb K
        \simeq (E\oplus \mathbb K)\oplus F
        \simeq E\oplus F
        =X .
\]
Hence $X$ is isomorphic to one of its hyperplanes.

It remains only to recall the standard fact that any two hyperplanes of an infinite-dimensional Banach space are isomorphic. Indeed, let
\[
    H=\ker f,
    \qquad
    K=\ker g,
\]
where $f,g\in X^*$ are non-zero. If $H=K$, there is nothing to prove. Otherwise $f$ and $g$ are linearly independent, and hence there exists $u\in X$ such that
\[
    f(u)=g(u)=1.
\]
Define
\[
    T\colon H\to K,
    \qquad
    T h=h-g(h)u .
\]
Then $g(Th)=0$, so $T(H)\subseteq K$. Similarly,
\[
    S\colon K\to H,
    \qquad
    S k=k-f(k)u
\]
is well-defined. A direct calculation gives $ST=\Id_H$ and $TS=\Id_K$. Thus $T$ is an isomorphism from $H$ onto $K$.
Since $X$ is isomorphic to one of its hyperplanes, every hyperplane of $X$ is therefore isomorphic to $X$.
\end{proof}
Let us record the following well-known fact.
\begin{lemma}\label{lem:qref-one-dual-and-summands}
Let $X$ be quasi-reflexive of order one.
Then $X^*$ is quasi-reflexive of order one.
Moreover, if $X=Y\oplus Z$, then
\[
    X^{**}/\kappa_X(X)
        \simeq
    \bigl(Y^{**}/\kappa_Y(Y)\bigr)
        \oplus
    \bigl(Z^{**}/\kappa_Z(Z)\bigr),
\]
where $\kappa_X$, $\kappa_Y$, and $\kappa_Z$ denote the canonical embeddings.
\end{lemma}

\begin{proof}
Since $X$ is quasi-reflexive of order one,
\(
    \dim\bigl(X^{**}/\kappa_X(X)\bigr)=1.
\)
The annihilator $\kappa_X(X)^\perp$ in $X^{***}$ is therefore one-dimensional.
For each $\Phi\in X^{***}$, set $x^*=\Phi\circ\kappa_X\in X^*$. Then
\(
    \Phi-\kappa_{X^*}(x^*)
\)
vanishes on $\kappa_X(X)$. Hence
\[
    X^{***}=\kappa_{X^*}(X^*)\oplus \kappa_X(X)^\perp .
\]
Thus $\kappa_{X^*}(X^*)$ has codimension one in $X^{***}$, so $X^*$ is quasi-reflexive of order one.

Now suppose that $X=Y\oplus Z$. The canonical identification
\[
    X^{**}\simeq Y^{**}\oplus Z^{**}
\]
sends $\kappa_X(y+z)$ to $\kappa_Y(y)\oplus\kappa_Z(z)$. Hence
\[
    \kappa_X(X)
        \simeq
    \kappa_Y(Y)\oplus\kappa_Z(Z),
\]
and the displayed quotient decomposition follows.
\end{proof}

\begin{proposition}\label{prop:qref-primary-dual}
Let $X$ be a quasi-reflexive Banach space of order one. Suppose that
\begin{enumerate}[label=\textup{(\roman*)}]
    \item $X$ is primary;
    \item $X$ is isomorphic to each of its hyperplanes;
    \item $X^*$ is isomorphic to each of its hyperplanes.
\end{enumerate}
Then $X^*$ is primary.
\end{proposition}

\begin{proof}
Since $X$ is quasi-reflexive of order one, $X^{**}$ is isomorphic to $X\oplus\mathbb K$. By hypothesis, $X\oplus\mathbb K\simeq X$, and therefore
\(
    X^{**}\simeq X .
\)

Let
\(
    X^*=Y\oplus Z
\)
be a direct-sum decomposition. Passing to duals gives
\[
    X^{**}\simeq Y^*\oplus Z^* .
\]
Since $X^{**}\simeq X$ and $X$ is primary, either $Y^*\simeq X$ or $Z^*\simeq X$. We may assume, without loss of generality, that
\(
    Y^*\simeq X .
\)
Then $Y$ is not reflexive, because $X$ is not reflexive.

By Lemma~\ref{lem:qref-one-dual-and-summands}, $X^*$ is quasi-reflexive of order one and
\[
    (X^*)^{**}/\kappa_{X^*}(X^*)
        \simeq
    \bigl(Y^{**}/\kappa_Y(Y)\bigr)
        \oplus
    \bigl(Z^{**}/\kappa_Z(Z)\bigr).
\]
The left-hand side is one-dimensional. Since $Y$ is not reflexive, the quotient $Y^{**}/\kappa_Y(Y)$ is non-zero. Consequently
\(
    \dim\bigl(Y^{**}/\kappa_Y(Y)\bigr)=1.
\)
Thus $Y$ is a hyperplane in $Y^{**}$.

Dualising $Y^*\simeq X$ gives
\(
    Y^{**}\simeq X^* .
\)
Therefore $Y$ is isomorphic to a hyperplane of $X^*$. By hypothesis, every hyperplane of $X^*$ is isomorphic to $X^*$. Hence $Y\simeq X^*$.
Thus, in every decomposition of $X^*$ into two complemented summands, at least one summand is isomorphic to $X^*$, and so $X^*$ is primary.
\end{proof}

\begin{theorem}\label{thm:Jp-dual-primary}
Let $1<p<\infty$. If $J_p$ is primary, then $J_p^*$ is primary.
In particular, the dual of the classical James space $J_2$ is primary.
\end{theorem}

\begin{proof}
By the standard complemented-subspace theorem for $J_p$, the space $J_p$ contains a complemented subspace $E$ isomorphic to $\ell_p$. Since
\(
    \ell_p\simeq \ell_p\oplus\mathbb K,
\)
Lemma~\ref{lem:complemented-copy-implies-hyperplanes} shows that $J_p$ is isomorphic to each of its hyperplanes.

Let $P\colon J_p\to E$ be a bounded projection. Then $P^*\colon J_p^*\to J_p^*$ is a bounded projection, and its range is naturally isomorphic to $E^*$. Since $E\simeq\ell_p$, we have
\(
    E^*\simeq\ell_q,
\)
where $q$ is conjugate to $p$. Thus $J_p^*$ contains a complemented copy of $\ell_q$. Since
\(
    \ell_q\simeq\ell_q\oplus\mathbb K,
\)
Lemma~\ref{lem:complemented-copy-implies-hyperplanes} shows that $J_p^*$ is also isomorphic to each of its hyperplanes.

The space $J_p$ is quasi-reflexive of order one. If, in addition, $J_p$ is primary, then Proposition~\ref{prop:qref-primary-dual}, applied with $X=J_p$, gives that $J_p^*$ is primary.
For $p=2$, the required primarity of $J_2$ is precisely Casazza's theorem \cite{Casazza1977}.
\end{proof}

\begin{remark}\label{rem:Jp-dual-primary-literature}
The proof above uses primarity of $J_p$ only as an input. For $p=2$, this input is classical by \cite{Casazza1977}. The remaining structural ingredients used here are the quasi-reflexivity of $J_p$ of order one and the complemented $\ell_p$-subspace theorem for $J_p$; both are recorded in \cite[Remark~2.2, Proposition~2.3, and Section~2.6]{BirdLaustsen2010}. Notice that the dual hyperplane reduction uses the complemented copy of $\ell_q$ obtained by dualising a projection onto a complemented copy of $\ell_p$ in $J_p$; the exact exponent is immaterial for the argument, since every $\ell_r$ with $1<r<\infty$ is isomorphic to its hyperplanes.
\end{remark}